\documentclass[11pt]{article}

\usepackage[dvipdfmx]{graphicx}

\usepackage{amsmath}
\usepackage{amsfonts}
\usepackage{amssymb}
\usepackage{amsthm}
\usepackage{ascmac}
\usepackage{cases}
\usepackage{color}%%
\usepackage{mathtools}
\usepackage{ulem}

\numberwithin{equation}{section}

\theoremstyle{plain}
\newtheorem{thm}{Theorem}[section]
\newtheorem{prop}[thm]{Proposition}
\newtheorem{lem}[thm]{Lemma}
\newtheorem{cor}[thm]{Corollary}

\theoremstyle{definition}
\newtheorem{df}[thm]{Definition}
\newtheorem{rem}[thm]{Remark}
\newtheorem{exmp}[thm]{Example}
\newtheorem{assumption}[thm]{Assumption}

\newcommand{\R}{\mathbb R}

\newcommand{\ep}{\varepsilon}
\newcommand{\restr}{\raisebox{0.1pt}[0pt]{$|$}}
\newcommand{\PR}{\mathcal{P}}

\title{Front propagation on a general metric graph}

\author{
Hiroshi Matano\footnote{Meiji Institute for Advanced Study of Mathematical Sciences, Meiji University, Tokyo, Japan}
\and 
Shuichi Jimbo\footnote{Department of Mathematics, Hokkaido University, Sapporo, Japan} 
}

\date{}

\begin{document}

\maketitle

%%\vspace{-8pt}
%%\begin{center}
%%\textit{Dedicated to the memory of Prof. Masaya Yamaguti\\ on the 100th anniversary of his birth}
%%\end{center}

\begin{abstract}
We consider a bistable reaction-diffusion equation on a metric graph that is a generalization of the so-called star graphs. More precisely, our graph $\Omega$ consists of a bounded finite metric graph $D$ of arbitrary configuration and a finite number of branches $\Omega_1,\ldots,\Omega_N\,(N\geq 2)$ of infinite length emanating from some of the vertices of $D$. Each $\Omega_i\,(i=1,\ldots,N)$ is called an ``outer path''. Our goal is to investigate the behavior of the front coming from infinity along a given outer path $\Omega_i$ and to discuss whether or not the front propagates into other outer paths $\Omega_j\,(j\ne i)$.  Unlike the case of star graphs, where $D$ is a single vertex, 
the dynamics of solutions can be far more complex and may depend sensitively on the configuration of the center graph $D$. We first focus on general principles that hold regardless of the structure of the center graph $D$. Among other things, we introduce the notion ``limit profile'', which allows us to define ``propagation'' and ``blocking'' without ambiguity, then we prove transient properties, that is, propagation $\Omega_i\to \Omega_j$ and $\Omega_j\to \Omega_k$ imply propagation $\Omega_i\to \Omega_k$.  
Next we consider perturbations of the graph $D$ while fixing the outer paths $\Omega_1,\ldots,\Omega_N$ and prove that if, for a given choice of $i,j$,  propagation 
$\Omega_i\to \Omega_j$ occurs for a graph $D$, then the same holds for any graph $D'$ that is sufficiently close to $D$ (robustness under perturbation). We also consider several specific classes of graphs, such as those with a ``reservoir'' type subgraph, and study their intriguing properties.
\end{abstract}

%%\footnotetext{Date: \today}



\tableofcontents

%%%%%%%%%%%%%%%%%%%%%%%%%%%
%%%%%%%%%%%%%%%%%%%%%%%%%%%
\section{Introduction}\label{s:introduction}

In recent years, reaction-diffusion equations on metric graphs are gaining growing attention. In this paper, we consider reaction-diffusion equations with bistable nonlinearities and discuss propagation and blocking of solution fronts. In the earlier papers \cite{JM2019, JM2021}, one of the present authors studied front propagation on star graphs and derived sharp conditions that tell whether or not the solution front can propagate beyond the junction point. In the present paper we first deal with a more general class of metric graphs and establish universal principles that hold regardless of the specific structure of the graphs, such as ``propagation/blocking dichotomy'' and ``transient properties''. We then consider perturbation of graphs and discuss whether or not the blocking and the non-blocking properties of the graph are robust under small perturbations. Next we present examples of metric graphs that exhibit such behavoirs as ``partial propagation'' and ``one-way propagation''. Finally we introduce the notion of ``complate invasion'' and ``incomplete invasion'', and show that incomplete invasion never occurs on star graphs or their small perturbartions, while it can occur if the graph possesses a ``reservoir'' type subgraph.

%%%%%%%%%%%%%
\subsection{Equation on a metric graph}\label{ss:metric-graph}

Roughly speaking, a metric graph is a directed graph in which each edge has a length scale. Thus the graph consists of vertices and directed edges, and each edge can be identified with an interval $[0,L]\subset\R$ for some $L>0$ if its length is finite or with $[0,\infty)\subset\R$ if its length is infinite. This identification gives local coordinates on each edge, hence the notion of neighborhoods around each point on the edge and around the vertices is well defined.

Let $G$ be a metric graph. We consider a reaction-diffusion equation on $G$ of the form
\begin{equation}\label{RD-G}
\partial_t u = \Delta_G u + f(u),\quad t\in(t_0,t_1),\  x\in G.
\end{equation}
Here, by a solution of \eqref{RD-G}, we mean a function $u(t,x)$ defined on $(t_0,t_1)\times G$ such that
\begin{itemize}\setlength\itemsep{0pt}
\item[(a)] $u$ is continuous on  $(t_1,t_2)\times G$;
\item[(b)] on each edge $E$ of $G$, when $E$ is identified with an interval in $\R$ (either with $[0,L]$ or with $[0,\infty)$), $u$ satisfies the one-dimensional equation $\partial_t u=\partial_x^2 u + f(u)$ on $E$ except at the endpoints, and $\partial_x u$ is continuous on $E$ up to the endpoints;
\item[(c)] at each vertex $V$ of $G$, $u$ satisfies the Kirchhoff condition.
\end{itemize}

%%The last condition (c) is important when one considers diffusion equation on a metric graph, 
Let us explain the meaning of Kirchhoff condition. Given a vertex $V$ of $G$, 
%%we focus on a local configuration of $G$ in a small neighborhood of $V$. Let 
let $E_1,\ldots,E_m$ be the edges that have $V$ as an endpoint. (If there is a loop whose both endpoints are $V$, we count them twice as separate edges, as we are focusing on a local configuration of $G$ around $V$.)   
For each $i=1,\ldots,m$, let $\partial u/\partial \nu_i(V)$ denote the derivative of $u$ at $V$ along the edge $E_i$ (that is, with respect to the coordinates of $E_i$ when $E_i$ is identified with an interval on $\R$) 
in the direction pointing toward the interior of $E_i$, as shown by the arrows in Figure \ref{fig:Kirchhoff}. We say that $u$ satisfies the {\it Kirchhoff condition} at $V$ if the following holds:
\begin{equation}\label{Kirchhoff}
 \sum_{i=1}^m \frac{\partial u}{\partial \nu_i}(V) =0.
\end{equation}
\begin{figure}[h]
\begin{center}
\includegraphics[scale=0.41]
{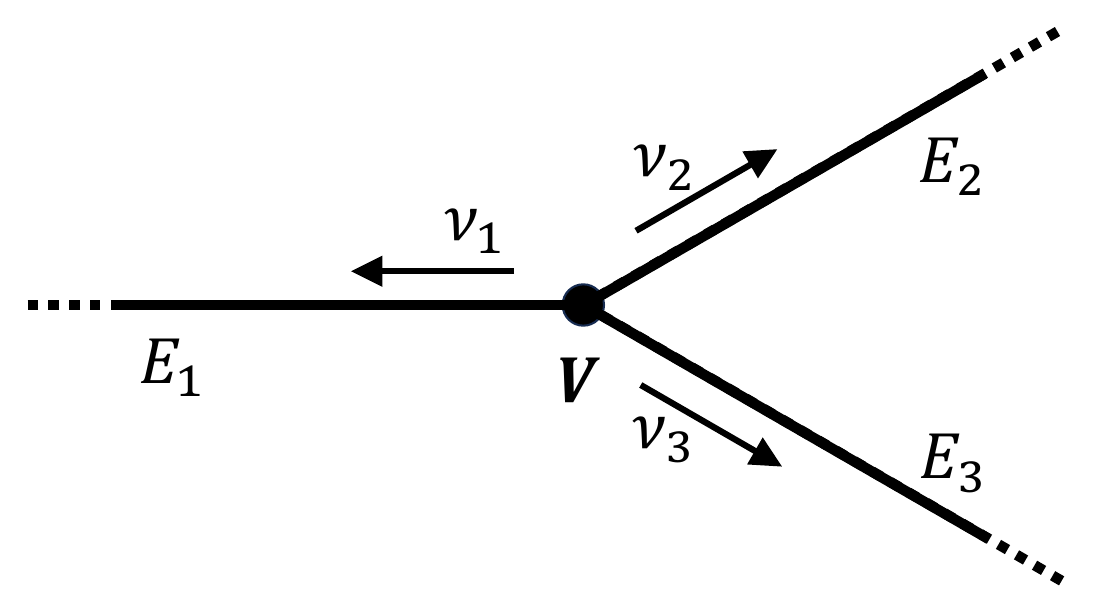}
\end{center}
\vspace{-10pt}
\caption{Kirchhoff condition.}
\label{fig:Kirchhoff}
\end{figure}

By Fick's law, the mass flux along each edge is proportional to $-\partial_x u$, therefore \eqref{Kirchhoff}  implies that the total mass flux flowing into $V$ is zero; in other words, mass is preserved at the junction point $V$. This condition is crucial when one considers diffusion equations on a metric graph. 

Note that, as we see from the conditions (a), (b), (c) above, the direction of the edges of $G$ does not play any role in the behavior of solutions of \eqref{RD-G}. The directions of the edges matter only when we express the solution using the local coordinates of the edges.

\begin{rem}\label{rem:ut-continuous}
If a function $u(t,x)$ satisfies the conditions (a), (b), (c) above, then it automatically holds that $\partial_t u$ and $\Delta_G u$ are continuous on $G$; see Lemma ~\ref{lem:ut-continuous}. \qed
\end{rem}

Incidentally, we allow vertices of degree $1$. In this case, the Kirchhoff condition reduces to the $0$ Neumann condition $\partial_x u=0$ at the degree 1 vertex (Figure \ref{fig:degree1}). 
\begin{figure}[h]
\begin{center}
\includegraphics[scale=0.48]
{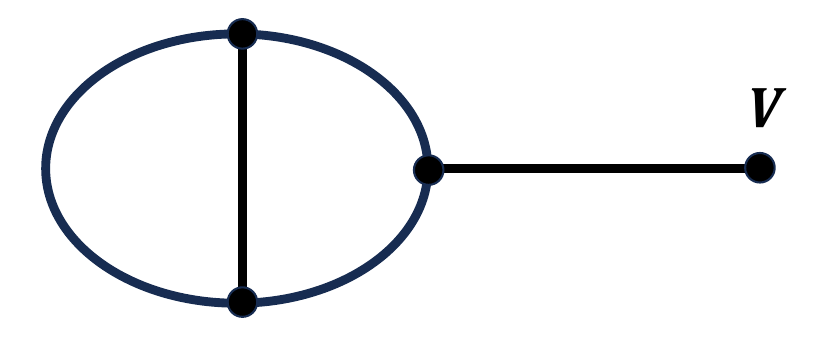}
\end{center}
\vspace{-10pt}
\caption{Graph with a degree 1 vertex.}
\label{fig:degree1}
\end{figure}

Next we consider the case where the edges of $G$ do not have uniform ``thickness''. By this we mean that 
the integral of a fucntion $h$ on a edge $E$ is given by the following form, 
when $E$ is identified with an interval $[0,L_E]\subset\R$, with $L_E$ being the length of $E$:
\begin{equation}\label{integral-E}
\int_{E} h(x)dx := \int_0^{L_E} \rho h(x) dx,
\end{equation}
where $\rho$ denotes the thickness of $E$. 
Accordingly, the mass flux on each edge is proportional to $-\rho\,\partial_x u$. 
For simplicity we assume that the thickness $\rho$ is constant on each edge. In this case, the Kirchhoff condition is given in the following form: 
\begin{equation}\label{Kirchhoff2}
 \sum_{i=1}^m \rho_i\frac{\partial u}{\partial \nu_i}(V) =0 \quad \hbox{(Kirchhoff condition for unequal thickness)},
\end{equation}
where $\rho_i$ denotes the thickness of the ege $E_i\,(i=1,\ldots,m)$. 
In this paper, unless otherwise stated, we assume that the edges of the graph may have unequal edge thickness.

\begin{rem}[Edges with unequal diffusion coefficients]\label{rem:unequal-diffusion}
We can also consider an equation of the form $\partial_t u = d\Delta_G u +f(u)$, in which the diffusion coefficient $d$ may take possibly different values among the edges. In this case, since the mass flux is proportional to $-d\hspace{1pt}\partial_x u$ on each edge, the Kirchhoff condition at each vertex $V$ is given by
\[
\sum_{i=1}^m d_i\frac{\partial u}{\partial \nu_i}(V) =0,
\] 
where $d_i$ denotes the diffusion coefficient of the edge $E_i\,(i=1,\ldots,m)$ in Figure \ref{fig:Kirchhoff}. However, by introducing a new variable $y=\sqrt{d}\hspace{1pt}x$ on each edge, we can normalize the diffusion coefficient to $1$ everywhere, thus the equation is converted to the form \eqref{RD-G}. With this new variable $y$, the above Kirchhoff condition is expressed in the following form:
\[
\sum_{i=1}^m \sqrt{d_i}\,\frac{\partial u}{\partial \nu_i}(V) =0.
\]
Therefore, considering an equation with unequal diffusion coefficients $d_i$ among edges is equivalent to considering \eqref{RD-G} with edges of unequal thickness $\sqrt{d_i}$. \qed
\end{rem}

Let us consider the Cauchy problem for \eqref{RD-G} under the initial condition
\begin{equation}\label{u0-G}
u(0,x)=u_0(x)\quad x\in G.
\end{equation}
In this paper, for the most part, we assume that $u_0$ is a bounded continuous function on $G$. In this case, the meaning of \eqref{u0-G} is understood as $\lim_{t\to 0}u(t,x)=u_0(x)$ uniformly on $G$ (or locally uniformly on $G$ if $G$ is unbounded). Occasionally we also consider the problem in the $C^1$ class, such as in Sections~\ref{ss:energy} and \ref{ss:estimates-reservoir}. In that case, we require that $u_0$ is continuous on $G$ and is $C^1$ on each edge of $G$ up to the end points, and that the Kirchhoff condition holds at every vertex. Then  $\lim_{t\to 0}u(t,x)=u_0(x)$ in the $C^1$ sense on every edge up to the end points.

Well-posedness of reaction-diffusion equations including \eqref{RD-G} or \eqref{RD-Omega} below can be shown by a standard approach that is known for semilinear diffusion equations on Euclidean domains. More precisely, using the heat kernel on the graph $G$ (or $\Omega$), we can convert the equation into an integral equation (Duhamel's formula), which can be solved by an iteration scheme or by the contraction mapping principle. See \cite{Ca1999} for heat kernels on metric graphs.

Another apporach for the well-posedness is given by von Below~\cite{Be1992}, who studied much more general quasilinear equations on metric graphs. The existence is shown by the Leray-Schauder principle, and the uniqueness is shown separately. Note that \cite{Be1992} only deals with bounded metric graphs, therefore we cannot apply its resluts directly to the graph $\Omega$ defined in \eqref{Omega}. However, one can easily modify the arguments of \cite{Be1992} to cover our case. (For example, existence of the solution can be shown by approximating $\Omega$ by a bounded graph and taking the limit.)

%%%%%%%%%%%%%
\subsection{Formulation of the problem}\label{ss:formulation}

We consider a metric graph $\Omega$ that consists of a bounded finite metric graph $D$ of an arbitrary configuration and a finite number of branches $\Omega_1,\Omega_2,\ldots,\Omega_N\,(N\geq 2)$ of infinite length that emanate from some of the vertices of $D$, denoted by ${\rm P}_1, {\rm P}_2, \ldots,{\rm P}_N$. We call these vertices ${\rm P}_1, {\rm P}_2, \ldots,{\rm P}_N$ the {\it exit points} and $\Omega_1,\Omega_2,\ldots,\Omega_N$ the {\it outer paths}; see Figure \ref{fig:Omega-examples}. Thus we have
\begin{equation}\label{Omega}
\Omega=D\cup \Omega_1 \cup \Omega_2 \cup \cdots \cup\Omega_N, 
\quad D\cap \Omega_i=\{{\rm P}_i\}\ (i=1,\ldots,N).
\end{equation}
Here ${\rm P}_1, {\rm P}_2, \ldots,{\rm P}_N$ are not necessarily distinct points. 
The coordinates of each outer path $\Omega_i\,(i=1,\ldots,N)$ are given by $x_i\in[0,\infty)$, with $x_i=0$ corresponding to the exit point ${\rm P}_i$.

\begin{figure}[h]
\begin{center}
\includegraphics[scale=0.49]
{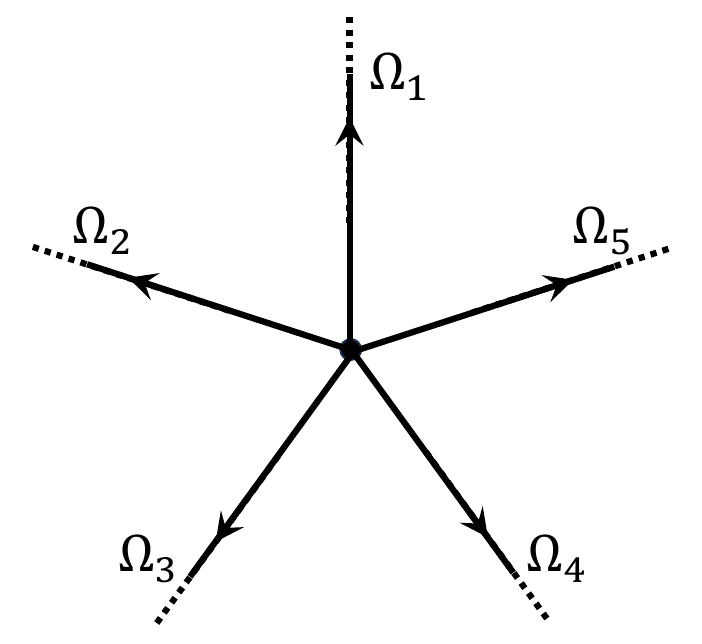}
\hspace{36pt}
\includegraphics[scale=0.49]
{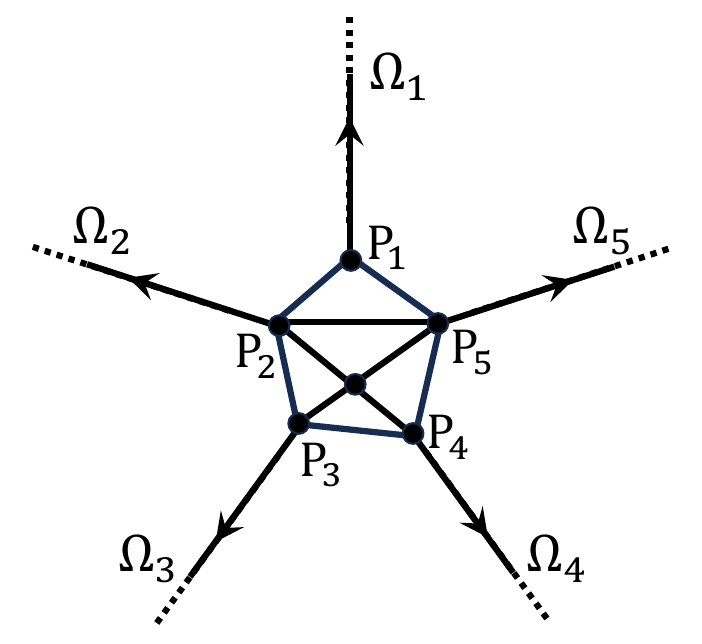}
\end{center}
\vspace{-5pt}
\caption{Examples of $\Omega$ for $N=5$. ({\it left}) a star graph;  ({\it right}) a more general graph.}
\label{fig:Omega-examples}
\end{figure}

We consider the following reaction-diffusion equation on the graph $\Omega$:
\begin{equation}\label{RD-Omega}
\partial_t u = \Delta_{\Omega} u + f(u).
\end{equation}
Here $f\in C^1$ is a bistable nonlinearity satifsying the following for some $0<a<1$:
\begin{equation}\label{f}
\begin{split}
& f(0)=f(a)=f(1)=0, \ \ \ f(s)<0 \ (0<s<a), \ \ f(s)>0\ (a<s<1),\\
& \hspace{40pt}f'(0)<0,\ f'(a)>0, \ f'(1)<0, \quad \int_0^1 f(s)ds>0.
\end{split}
\end{equation}
Let $F(s)=\int_0^s f(r)dr$. 
Then from \eqref{f} we see that $F$ satisfies the following:
\begin{equation}\label{F}
F(0)=0,\ \ F(a)<0<F(1), \ \ F(a)=\min_{s\in[0,1]} F(s),\ \ F(1)=\max_{s\in[0,1]} F(s).
\end{equation}
Furthermore, there exists a unique value $\beta$ with $a<\beta<1$ such that
\begin{equation}\label{beta}
F(\beta)=0.
\end{equation}
This value $\beta$ is equal to the maximum of the peaking solution $U(x)$ given in Figure \ref{fig:pulse-solution} in Section~\ref{ss:phase-portrait}, or, equivalently, the right-most position of the homoclinic orbit in Figure \ref{fig:phase-portrait}.  

Under the assumption \eqref{f}, it is well-known that the equation $\partial_t u=\partial_x^2 u+f(u)$ on $\R$ has a traveling wave solution of the form $\phi(x-ct)$ that connects $0$ to $1$. Here $c>0$ represents the speed of the traveling wave and the profile function $\phi(z)$ satisfies the following:
\begin{equation}\label{phi}
\begin{cases}
\,\phi''+c\phi'+f(\phi)=0%%\ \ 0<\phi(z)<1
\ \ (z\in\R),\\
\,0<\phi(z)<1\ \ (z\in\R),\ \ \phi(-\infty)=1,\ \ \phi(+\infty)=0.
\end{cases}
\end{equation}
It is well known that such a traveling wave is unique up to translation and that $\phi'(z)<0$; see \cite{FM1977}. Hereafter we specify the value of $\phi$ at $0$ as follows, which makes $\phi$ uniquely determined.
\begin{equation}\label{phi(0)}
\phi(0)=a.
\end{equation}

The above traveling wave propagates toward the right direction, from $x=-\infty$ to $x=+\infty$, so we can call it a right-bound traveling wave. A left-bound traveling wave, which travels from $x=+\infty$ to $x=-\infty$ is given in the form $\phi(-x-ct)$.

In this paper, we discuss the behavior of a solution whose front travels along an arbitrarily chosen outer path $\Omega_i$ from the direction $x_i=+\infty$ like a left-bound traveling wave when $t$ is sufficiently negative. The following theorem confirms that such a solution exits uniquely:

\begin{thm}[front-like solution]\label{thm:u-hat}
For each $i\in\{1,\ldots,N\}$, there exists a unique entire solution $\widehat{u}_i(t,x)$ of \eqref{RD-Omega} such that $0<\widehat{u}_i(t,x)<1$ for $(t,x)\in\R\times\Omega$ and that
\begin{equation}\label{u-hat}
\lim_{t\to -\infty} \left(\sup_{x_i\in\Omega_i}\big|\widehat{u}_i{}_{\restr\Omega_i}(t,x_i)-\phi(-x_i-ct)\big|+ \sup_{x\in\Omega\setminus\Omega_i}|\widehat{u}_i(t,x)|\right)=0,
\end{equation}
where $\widehat{u}_i{}_{\restr\Omega_i}$ denotes the restriction of $\widehat{u}_i$ to $\Omega_i$. 
Furthermore, $\widehat{u}_i(t,x)$ is strictly monotone increasing in $t$. 
\end{thm}

In the case where $\Omega$ is a star graph, the existence of such an entire solution $\widehat{u}_i$ is proved in \cite{JM2019}, though they did not discuss the uniqueness and time-monotonicity. 
Our proof of existence is similar to that of \cite{JM2019} and is based on a super-subsolution method and a limiting argument. The uniqueness and time-minotonicity follow from Lemma~\ref{lem:comparison-ancient} in Section \ref{ss:proof-front}. 

Since $\widehat{u}_i(t,x)$ is monotone increasing in $t$, the following limit exists, which we call the {\it limit profile}:
\begin{equation}\label{v-hat}
\widehat{v}_i(x):=\lim_{t\to\infty}\widehat{u}_i(t,x).\quad\hbox{(limit profile)}
\end{equation}
Since $\widehat{u}_i$ is unique, the profile $\widehat{v}_i$ is uniquely determined by $\Omega$, $f$ and $i$. 
By parabolic estimates, the derivatives of $\widehat{u}_i$ also converge to those of $\widehat{v}_i$, therefore $\widehat{v}_i$ is a stationary solution of \eqref{RD-Omega}:
\begin{equation}\label{eq:v-hat}
\Delta_\Omega \widehat{v}_i+f(\widehat{v}_i)=0\quad\hbox{on}\ \ \Omega.
\end{equation}
This means that, on each edge $E$ of $\Omega$, $\widehat{v}_i$ satisfies the one-dimensional equation $\partial_x^2 \widehat{v}_i+f(\widehat{v}_i)=0$ except at the endpoints, while it is continuous and piecewise $C^1$ on $\Omega$ and satisfies the Kirchhoff condition \eqref{Kirchhoff2} at every vertex of $\Omega$. 

\begin{rem}\label{rem:limit-profile}
The notion of ``limit profile'' is adopted from one of the present authors' earlier paper \cite{BHM2025}, which studies propagation of bistable fronts through a perforated wall. 
\end{rem}

Given an arbitrary pair of indices $i,j\in\{1,\ldots,N\}$ with $i\ne j$, we define the notion of propagation and blocking between the outer paths $\Omega_i$ and $\Omega_j$. This is done by looking at the behavior of $\widehat{v}_i$ along the outer path $\Omega_j$. In what follows, we let $0\leq x_j<\infty$ be the coordinates of $\Omega_j$ and let $\widehat{v}_i{}_{\restr\Omega_j}(x_j)$ be the restriction of $\widehat{v}_i$ onto $\Omega_j$. 

\begin{df}\label{def:propagation}
Let $i,j\in\{1,\ldots,N\}$, $i\ne j$. We say that {\it propagation} (resp. {\it blocking})  occurs from $\Omega_i$ to $\Omega_j$ if
\begin{equation}\label{i-j:propagation}
\lim_{x_j\to\infty}\widehat{v}_i{}_{\restr\Omega_j}(x_j)=1
\quad \hbox{(resp.}\  
\lim_{x_j\to\infty}\widehat{v}_i{}_{\restr\Omega_j}(x_j)=0 \hbox{\,)}.
\end{equation}
%%We say that blocking occurs from $\Omega_i$ to $\Omega_j$ if
%%\begin{equation}\label{i-j:blocking}
%%\lim_{x_j\to\infty}\widehat{v}_i|_{\Omega_j}(x_j)=0.
%%\end{equation}
\end{df} 

As we shall see from Theorem \ref{thm:dichotomy} in Section \ref{ss:dichotomy}, $\widehat{v}_i{}_{\restr\Omega_j}(x_j)$ converges either to $1$ or $0$ as $x_j\to\infty$. Therefore, we have either propagation or blocking from $\Omega_i$ to $\Omega_j$, and there is no other intermediate behavior of $\widehat{v}_i$.  We define the {\it propagation index} $\PR(i,j)$ as follows:
\begin{equation}\label{Pij}
\PR(i,j)=
\begin{cases}
1 & \hbox{if propagation occurs from $\Omega_i$ to $\Omega_j$;}\\
0 & \hbox{if blocking occurs from $\Omega_i$ to $\Omega_j$.}
\end{cases}
\end{equation}
By the uniqueness of $\widehat{v}_i$, the value of $\PR(i,j)$ is well-defined.

Unlike the case of star graphs studied in \cite{JM2019,JM2021}, the behavior of solutions can be much more complicated depending on the configuration of the center graph $D$.  
For example, as we shall show in Section \ref{ss:examples}, it may happen that $\PR(i,j)=1$ for some $j$ while $\PR(i,j)=0$ for other $j$ (partial propagation). Also, it may happen that $\PR(i,j)=1$ while $\PR(j,i)=0$ (one-way propagation). 
Despite such complexity, there are certain properties that hold universally regardless of the structure of the center graph $D$, such as the following transient property (Theorem \ref{thm:transient}):
\begin{equation}\label{transient}
\PR(i,j)=1,\ \ \PR(j,k)=1 \ \ \Longrightarrow\ \ \PR(i,k)=1.
\end{equation}

In this paper, we first focus on such general properties as \eqref{transient} that hold for a large class of metric graphs. Later, we pick up various specific classes of graphs and study their properties, including star graphs, perturbed star graphs, graphs with a ``reservoir'', {\it etc}.

Let us mention some earlier works on propagation and blocking for reaction-diffusion equations on a metric graph. As we already mentioned, the case of star graphs has been studied in detail in \cite{JM2019, JM2021}. More precisely, they derived, among other things, a sharp criterion for propagation and blocking, which is given in \eqref{star-graph-equal} of the present paper. Roughly speaking, their result shows that the chances of blocking increase as $N$ becomes larger. 
Note that they also proved the existence of an entire solution satisfying \eqref{u-hat} for a star graph, though they did not show the uniqueness of such a solution as we stated in Theorem~\ref{thm:u-hat}.

We also mention a much earlier pioneering work of Pauwelussen \cite{Pa1981, Pa1982}, who studied propagation and blocking phenomena on networks with weighted edges. Among other things, the author derived a sharp criterion for propagation and blocking for a weighted 2 star graph, which is essentially the same condition as given in Corollary~\ref{cor:unequal-diffusion} of the present paper. 
In \cite{Pa1982}, the author also derived sufficient conditions for blocking in FitzHugh-Nagumo system.

In \cite{IJM2022}, the authors consider a graph $\Omega$ consisting of two $k$-star graphs whose center points ${\rm P}_1, {\rm P}_2$ are connected by an extra edge (Figure~\ref{fig:IJM-JM} {\it left}). In this case, $D$ consists of the edge ${\rm P}_1{\rm P}_2$ and the vertices ${\rm P}_1, {\rm P}_2$, and there are $N=2k$ outer paths. If the length of the edge ${\rm P}_1{\rm P}_2$ is small, $\Omega$ would look like a $2k$-star graph. As the length of ${\rm P}_1{\rm P}_2$ becomes larger, $\Omega$ would look more like two separate $k+1$ graphs, therefore one may suspect that the chances of propagation would increase. In \cite{IJM2022}, the authors confirm that it is indeed the case. In \cite{JM2024}, the authors consider a graph $\Omega$ as in Figure~\ref{fig:IJM-JM} ({\it right}). In this case, $D$ consists of the edges ${\rm P}_1{\rm P}_2$, ${\rm P}_1{\rm P}_3$ along with the vertices ${\rm P}_1, {\rm P}_2, {\rm P}_3$. When the length of the edges ${\rm P}_1{\rm P}_2$, ${\rm P}_1{\rm P}_3$ is very small, $\Omega$ would look like a 5-star graph, while, if the edges ${\rm P}_1{\rm P}_2$, ${\rm P}_1{\rm P}_3$ are long, it would look like three separate 3-star graphs. Consequently, one may speculate that the larger the length of $D$, the higher the chances of propagation, and \cite{JM2024} proves that this is indeed the case. These examples show that altering the length of the edges of $D$ can largely affect the propagation properties of $\Omega$. 

\begin{figure}[h]
\begin{center}
\includegraphics[scale=0.36]
{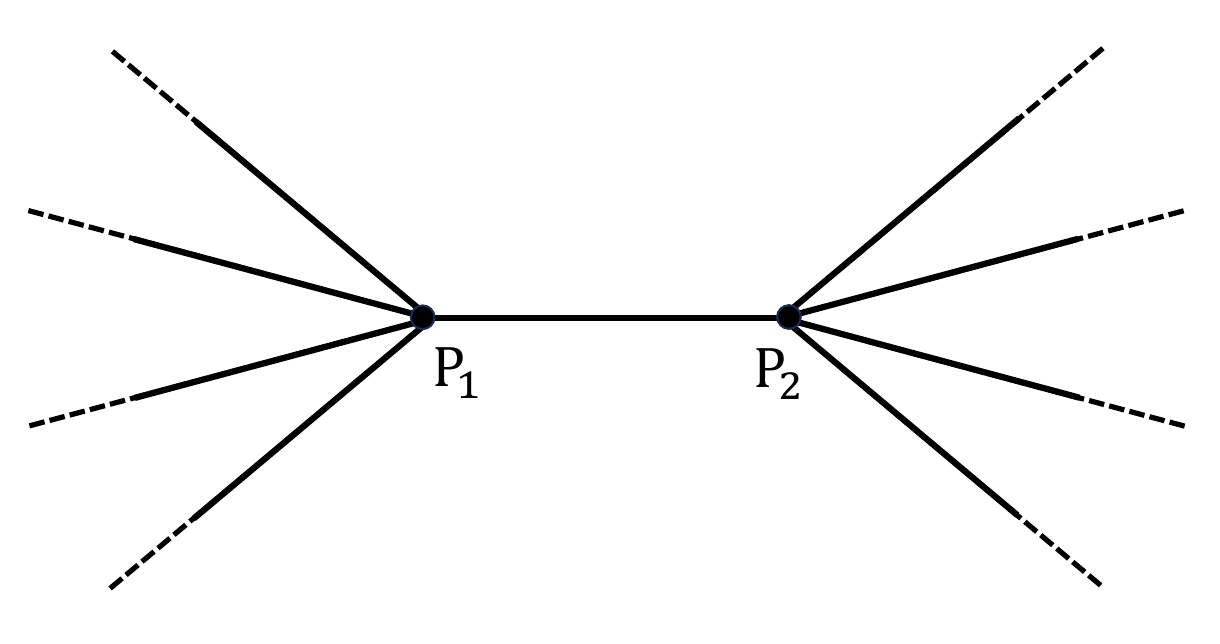}
\hspace{5pt}
\includegraphics[scale=0.36]
{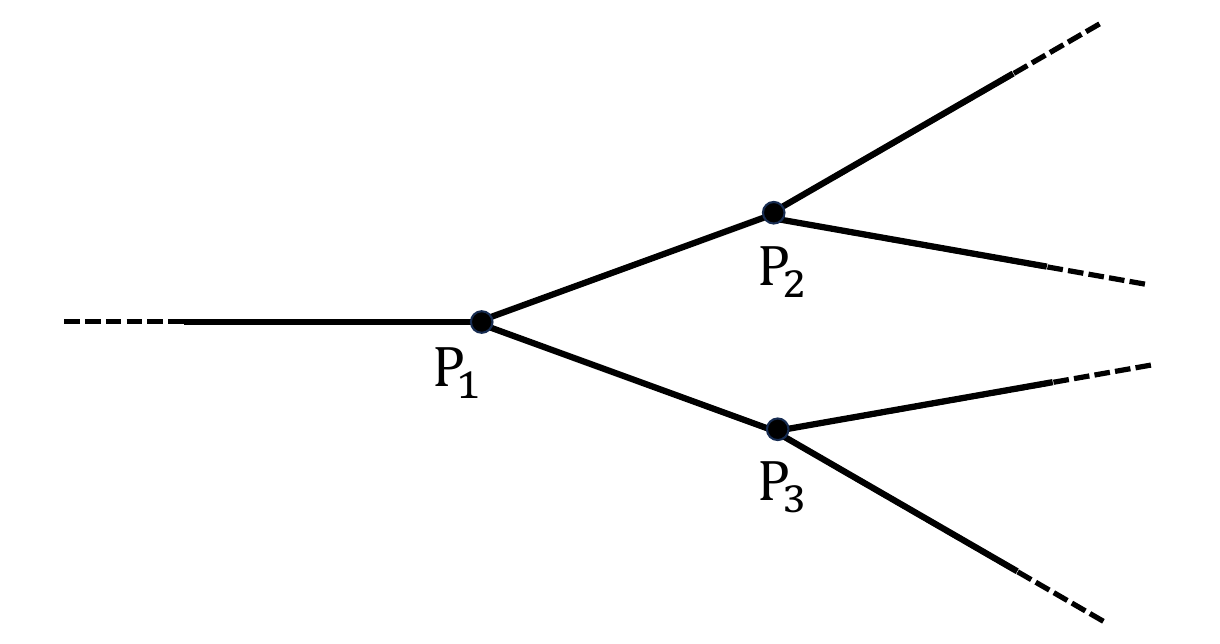}
\end{center}
\vspace{-5pt}
\caption{Two star graphs connected by an edge \cite{IJM2022} ({\it left}); double branching \cite{JM2024} ({\it right}).}
\label{fig:IJM-JM}
\end{figure}

In \cite{LM2024}, the authors consider a rather complex graph which they call a ``root-like graph''. This graph is obtained by letting the outer paths of a 3-star graph bifurcate into triple paths twice.  The resulting graph has 27 outer paths and the center graph $D$ possesses 12 quadruple junctions and one triple junction. The authors establish spreading-transition-vanishing trichotomy for solutions of bistable reaction-diffusion equations on this graph.

Lastly, we also mention the work \cite{MNO2023}, which considers bistable competition-diffusion systems on a star graph.and derive conditions for propagation and blocking.

%%%%%%%%%%%%%
\subsection{Organization of the paper}\label{ss:organization}

This paper is organized as follows. In Section~\ref{s:main}, we state our main theorems. First, in Section~\ref{ss:dichotomy}, we present our key result on the propagation-blocking dichotomy (Theorem~\ref{thm:dichotomy}). 
Next, in Section~\ref{ss:minimality}, we state important properties of the limit profile $\widehat{v}_i$, namely the minimality (Theorem~\ref{thm:minimal}) and stability (Theorem~\ref{thm:stability}). Then we state the transient properties of $\PR(i,j)$ (Theorem~\ref{thm:transient}). This last theorem is a direct consequence of Theorem~\ref{thm:minimal}.  In Section~\ref{ss:perturbation}, we consider perturbation of the graph and discuss the influence of the perturbations. Among other things, we show that the limit of a sequence of blocking graphs is again a blockign graph (Theorem~\ref{thm:limit-blocking}). As a contraposition of this theorem, we see that non-blocking property is robust under small perturbations (Corollary~\ref{cor:propagation-open}).

In Section~\ref{ss:star-graphs}, we first extend the results of \cite{JM2019, JM2021} on star graphs to those with possibly unequal edge thickness (Theorem~\ref{thm:star}). 
Then we consider small perturbation of star graphs and discuss whether or not non-blocking and blocking properties are preserved under small perturbations. In Section~\ref{ss:examples}, we present examples of partial and one-way propagation.

In Section~\ref{ss:reservoir}, we consider the case where $\Omega$ has a ``reservoir'' type subgraph such as the one in Figure~\ref{fig:3star-reservoir}, and show that the value of the limit profile $\widehat{v}_i$ is small in the reservoir, even if the front propagates from $\Omega_i$ to all other outer paths (Theorem~\ref{thm:reservoir}). The proof is based on the construction of an upper barrier, which is done by a variational argument (or an energy estimate)  A similar notion of reservoir is found in \cite[Theorem 9.1]{BHM2025}, which studies front propagation through a perforated wall. In Section~\ref{ss:general}, we show that a large class of solutions of equation \eqref{RD-Omega} whose initial support is contained in $\Omega_i$ converges to $\widehat{v}_i$ as $t\to\infty$. Therefore the function $\widehat{v}_i$ represents the long-time behavior of a large class of solutions.

In Section~\ref{s:preliminaries}, we recall some basic properties of equations on the graph, such as the comparison principle, Gauss--Green theorem, Poincar\'e inequality, and so on. The proof of the main results will be carried out in Section~\ref{s:proof-main}.

%%%%%%%%%%%%%%%%%%%%%%%%%%%
%%%%%%%%%%%%%%%%%%%%%%%%%%%
\section{Main results}\label{s:main}

In this section we present the main results of our paper. Unless otherwise stated, we allow the edges of  $\Omega$ to have unequal thickness, thus the Kirchhoff condition is given by \eqref{Kirchhoff2}.

%%%%%%%%%%%%%
\subsection{Propagation-blocking dichotomy}\label{ss:dichotomy}

Our first main result is the following dichotomy theorem.  The notion of propagation and blocking is based on this theorem, which will play a key role throughout the present paper.

\begin{thm}[Dichotomy theorem]\label{thm:dichotomy}
Let $i,j\in\{1,\ldots,N\}$, and let $\widehat{v}_i$ denote the limit profile for $i$ and $\widehat{v}_i{}_{\restr\Omega_j}(x_j)\,(0\leq x_j<\infty)$ its restriction on $\Omega_j$. Then either of the following holds:
\begin{equation}\label{dichotomy}
\lim_{x_j\to\infty}\widehat{v}_i{}_{\restr\Omega_j}(x_j)=1 \quad \hbox{or}\quad \lim_{x_j\to\infty}\widehat{v}_i{}_{\restr\Omega_j}(x_j)=0. 
\end{equation}
Furthemore, the former (resp. the latter) in \eqref{dichotomy} holds if and only if $\widehat{v}_i({\rm P}_j)>\beta$ (resp. $\leq\beta$), where $\beta$ is the value defined in \eqref{beta}. 
If the former (resp. the latter) holds, $\widetilde{v}_i{}_{\restr\Omega_j}$ is monotone increasing (resp. decreasing) on $\Omega_j$. 
\end{thm}

The above theorem will be proved by the phase portrait analysis (Figure~\ref{fig:phase-portrait}) and a special comparison principle on ancient solutions (Lemmas~\ref{lem:comparison-ancient}, \ref{lem:comparison-ancient2}). 

%%%%%%%%%%%%%
\subsection{Minimality, stability and transient properties}\label{ss:minimality}

In this section, we present results on general properties of the limit profile $\widehat{v}_i$ that hold universally regardless of the structure of the center graph $D$, namely, minimality, stability and transient properties. We start with the minimality theorem.

%%First we explain the minimality. 
%%From the definition of $\widehat{v_i}$, it is clear that $\widehat{v}_i{}_{\restr\Omega_i}(x_i)\to 1$ as $x_i\to 1$. The following proposition states that $\widehat{v_i}$ is the minimal among all the stationary solution having this property.

\begin{thm}[Minimality of $\widehat{v}_i$]\label{thm:minimal}
Let $v$ be a supersolution of the stationary problem
\begin{equation}\label{stationary}
\Delta_\Omega v+ f(v)=0\quad\hbox{on}\ \ \Omega
\end{equation}
such that $0\leq v\leq 1$ and that 
\begin{equation}\label{Omega-i-1}
\lim_{x_i\to\infty}v{}_{\restr\Omega_i}(x_i)=1.
\end{equation}
Then $v\geq \widehat{v}_i$, where $\widehat{v}_i$ denotes the limit profile associated with $\Omega_i$. 
In particular, $\widehat{v}_i$ is minimal among all the solutions of \eqref{stationary} satisfying \eqref{Omega-i-1}.
\end{thm}

Here by $v$ being a supersolution of \eqref{stationary} we mean that $v$ satisfies $\partial_x^2 v+f(v)\leq 0$ on each edge of $\Omega$ (at least in the weak sense) and that it satisfies super-Kirchhoff condition \eqref{super-Kirchhoff} at every vertex. See Definition~\ref{def:super-sub} for details.

The above theorem will play a key role in the proof of the transient properties (Theorem~\ref{thm:transient}) and in the proof of Theorem~\ref{thm:limit-blocking} concerning the limit of a sequence of blocking graphs. It is also used to study the behavior of more general solutions of the Cauchy problem (Theorem~\ref{thm:general}).

Next we discuss stability. Since $\widehat{v}_i(x)$ is the monotone increasing limit of the solution $\widehat{u}_i(t,x)$, it is natural to expect that $\widehat{v}_i(x)$ possesses some sort of stability properties. We show below that $\widehat{v}_i(x)$ is linearly stable with respect to compactly supported perturbations.

Let us first introduce some notation. For each $R>0$ and $i\in\{1,\ldots,N\}$, we define
\[
\Omega_i[R]:=\{x\in\Omega_i\mid 0\leq x_i\leq R\}, \quad {\rm P}_i^R:=\{x\in\Omega_i\mid x_i= R\}
\]
\begin{equation}\label{Omega-R}
\Omega^R:=D\cup\Omega_i[R]\cup\cdots\cup\Omega_i[R].
%%\quad \partial\Omega^R:=\{{\rm P}_1^R,\ldots,{\rm R}_N^R\}.
\end{equation}
We regard ${\rm P}_1^R,\ldots, {\rm P}_N^R$ as temporary vertices. Then $\Omega^R$ is a bounded subgraph of $\Omega$.  We define the boundary of $\Omega$ by $\partial\Omega^R:=\{{\rm P}_1^R,\ldots, {\rm P}_N^R\}$. Next we consider the following eigenvalue problem:
\begin{equation}\label{EP-OmegaR}
\begin{cases}
\,-\Delta_{\Omega^R}\varphi - f'\left(\widehat{v}_i\right)\varphi=\lambda \varphi & \hbox{in}\ \ \Omega^R,\\
\,\varphi=0\ \ & \hbox{on}\ \ \partial\Omega^R.
\end{cases}
\end{equation}
Let $\lambda^R$ denote the principal engenvalue of \eqref{EP-OmegaR}, and $\varphi^R(x)>0$ be the corresponding eigenfunction. The existence of the principal eigenvalue for such a problem is easily seen. Incidentally, \cite{Be1988} studies more general Sturm-Liouville eigenvalue problems on metric graphs and shows, among other thing, that all eigenvalues are discrete.

We say that the stationary solution $\widehat{v}_i$ is called {\it linearly stable} on $\Omega^R$ if $\lambda^R>0$, while it is called {\it linearly unstable} on $\Omega^R$ if $\lambda^R<0$. If $\lambda^R>0$, then $\widehat{v}_i$ is asymptotically stable as a stationary solution of \eqref{RD-Omega} under any small perturbation whose support is contained in $\Omega^R$. Note that, just as in the case of classical Dirichlet eigenvalue problems in Euclidean domains, $\lambda^R$ is stricly decreasing in $R$, which can easily be shown by the strong maximum principle.

\begin{thm}[Linear stability of $\widehat{v}_i$]\label{thm:stability}
$\lambda^R>0$ for all $R>0$.
\end{thm}

The above theorem implies, in particular, that $\widehat{v}_i$ is stable under small perturbations whose support is compact. Note that the above theorem holds for any graph $\Omega$ and nonlinearity $f$ satisfying \eqref{Omega} and \eqref{f}. As a corollary to the above theorem, we obtain the following:

\begin{cor}\label{cor:no-two-stationary}
Let $E$ by any edge of $\Omega$. Then $\widehat{v}_i$ does not posses more than one stationary point (that is, a point where $\partial_x \widehat{v}_i=0$) on $E$ unless $\widehat{v}_i\equiv a$ or $\widehat{v}_i\equiv 1$ on $E$..
\end{cor}

A result similar to the above corollary is found in \cite[Theorem 3.2]{Mag1978}, in which the author showed that a stable solution of the equation $\partial_x^2 u+f(u)=0$ on a finite interval in $\R$ under arbitrary Dirichlet boundary conditions do not have more than one stationary points.

Finally, we present our results on the transiend properties. 

\begin{thm}[Transient properties]\label{thm:transient}
Let $i,j,k\in\{1,\ldots,N\}$. If $\PR(i,j)=1$ and $\PR(j,k)=1$, then $\PR(i,k)=1$.
\end{thm}

As we shall see, the above theorem followis easily from Theorem~\ref{thm:minimal}.

%%%%%%%%%%%%%
\subsection{Perturbation and deformation of graphs}\label{ss:perturbation}

In this section we consider perturbation of the graph $\Omega$ and the nonlinearity $f$ and discuss the influence of the perturbation on the propagation and blocking properties. 
For the clarity of the arguments, we discuss the perturbation of $\Omega$ and that of $f$ separately, but the combination of the two types perturbations can be handled just similarly. 

We start with the perturbation of $\Omega$.  
Let $\Omega$ be as in \eqref{Omega} and $\Omega^{(m)}\,(m=1,2,3,\ldots)$ be the same type of metric graphs as $\Omega$ with $N$ outer paths $\Omega_{m,i}\,(i=1,\ldots,N)$ and a center graph $D_m$:
\[
\Omega^{(m)}=D_m\cup \Omega_{m,1} \cup \Omega_{m,2} \cup \cdots \cup\Omega_{m,N}, 
\quad D_m\cap \Omega_{m,i}=\{{\rm P}_{m,i}\}\ (i=1,\ldots,N).
\]
Here the center graph $D_m$ is a bounded finite metric graph that converges to $D$ as $m\to\infty$ in the sense specified in Assumption~\ref{ass:perturbation-D} below.

\begin{assumption}[Perturbation of $D$]\label{ass:perturbation-D} $D_m$ is a bounded finite metric graph that is obtained by perturbing $D$ by one of the (or combination of the) following operations:
\begin{itemize}\setlength{\itemsep}{0pt}
\item[(a)] Modify the length of the edges of $D$;
\item[(b)] Replace a point on an edge of $D$ by a small finite metric graph $\Sigma_m$ (Figure~\ref{fig:perturbation-edge} ({\it above}));
\item[(c)] Replace a vertex of $D$ by a small finite metric graph $\Sigma_m$ (Figure~\ref{fig:perturbation-edge} ({\it below})).
\end{itemize}
Furthermore, as $m\to\infty$, we assume that the total length of the edges of $\Sigma_m$ in operations (b), (c) tends to $0$, and that the change of the length of edges in operation (a) tends to $0$. 
\end{assumption}

\begin{figure}[h]
\begin{center}
\includegraphics[scale=0.43]
{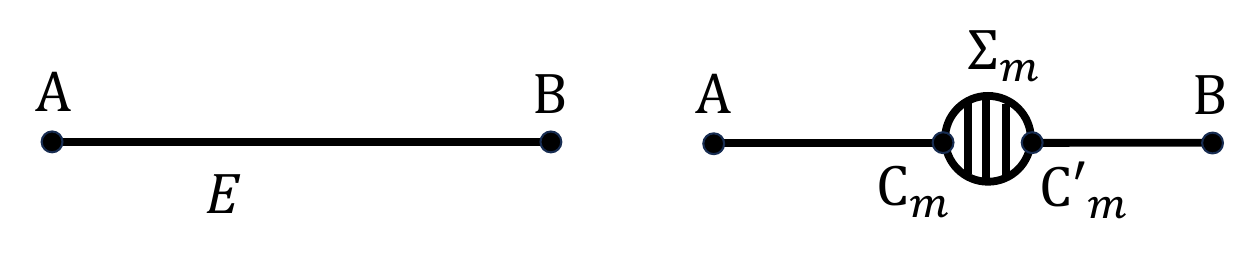}

\vspace{8pt}
\includegraphics[scale=0.43]
{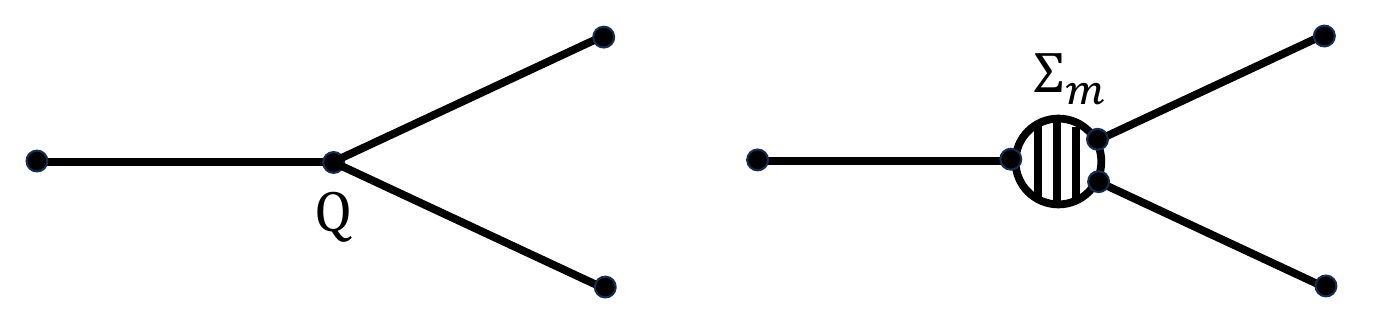}
\end{center}
\vspace{-5pt}
\caption{Replacing a point on an edge or a vertex by an arbitrary small graph $\Sigma_m$.}
\label{fig:perturbation-edge}
\end{figure}

An example of the operation (b) is shown in Figure~\ref{fig:perturbation-edge} ({\it above}). Here a point on the edge $E={\rm AB}$ is replaced by a small graph $\Sigma_m$. As $m\to\infty$, we assume that the total length of the edges of $\Sigma_m$, denoted by $|\Sigma_m|$, tends to $0$ and that the sum of the length of the edges ${\rm AC}_m$ and ${\rm BC}'_m$ converges to the length of ${\rm AB}$. 
Similarly, in the operation (c), a vertex ${\rm Q}$ is replaced by a small graph $\Sigma_m$, and we assume that $|\Sigma_m|\to 0$ as $m\to\infty$ and that the length of the edges emanating from $\Sigma_m$ converges to that of the correspondingl edges emanating from ${\rm Q}$.

Now we consider the same equation as \eqref{RD-Omega} on $\Omega^{(m)}$:
\begin{equation}\label{RD-m}
\partial_t u = \Delta_{\Omega^{(m)}} u + f(u).
\end{equation} 
For each pair $i,j\in\{1,\ldots,N\}$ with $i\ne j$, we define the propagation index $\PR_m(i,j)$ for equation \eqref{RD-m} in the same manner as \eqref{Pij}. More precisely, we define
\begin{equation}\label{Pij-m}
\PR_m(i,j)=
\begin{cases}
1 & \hbox{if propagation occurs from $\Omega_{m,i}$ to $\Omega_{m,j}$;}\\
0 & \hbox{if blocking occurs from $\Omega_{m,i}$ to $\Omega_{m,j}$.}
\end{cases}
\end{equation}
The following theorem states that the limit of blocking graphs is again a blocking graph:

\begin{thm}[Limit of blocking graphs]\label{thm:limit-blocking}
Let $i,j\in\{1,\ldots,N\}$ with $i\ne j$ and let $\PR_m(i,j)\,(m=1,2,3,\ldots)$ be as defined above. If $\PR_m(i,j)=0$ for all $m=1,2,3,\ldots$, then $\PR(i,j)=0$.
\end{thm}

Next we discuss perturbation of the nonlinearity $f$.  Let $f_m\,(m=1,2,3,\ldots)$ be a sequence of $C^1$ functions converging to $f$ in $C^1$ as $m\to\infty$ and satisfying the following for some $0<a_m<1$:
\begin{equation}\label{fm}
\begin{split}
& f_m(0)=f_m(a_m)=f_m(1)=0, \ \ \ f_m(s)<0 \ (0<s<a_m), \ \ f_m(s)>0\ (a_m<s<1),\\
& \hspace{40pt}f'_m(0)<0,\ f_m'(a_m)>0, \ f_m'(1)<0, \quad \int_0^1 f_m(s)ds>0.
\end{split}
\end{equation}
%%$a_m\to a$ as $m\to\infty$
We then consider the following equation on $\Omega$ and define $\PR_m(i,j)$ as in \eqref{Pij} for \eqref{RD-m2}:
\begin{equation}\label{RD-m2}
\partial_t u = \Delta_{\Omega} u + f_m(u).
\end{equation} 
The following theorem is an analogue of Theorem~\ref{thm:limit-blocking} for the pertubation of the nonlinearity:

\begin{thm}[Limit of blocking nonlinearities]\label{thm:limit-blocking2}
Let $i,j\in\{1,\ldots,N\}$ with $i\ne j$ and let $\PR_m(i,j)\,(m=1,2,3,\ldots)$ be as above. If $\PR_m(i,j)=0$ for all $m=1,2,3,\ldots$, then $\PR(i,j)=0$.
\end{thm}

In the proof of the above theorems, Theorem~\ref{thm:minimal} on the minimality of the limit profile will play a key role. The following corollary, which is a contraposition of Theorems~\ref{thm:limit-blocking} and \ref{thm:limit-blocking2}, asserts that propagation is robust under small perturbations.

\begin{cor}[Robustness of propagation]\label{cor:propagation-open}
Let $i,j\in\{1,\ldots,N\}$ with $i\ne j$ and and let $\PR_m(i,j)$ denote the propagation index for \eqref{RD-m} or \eqref{RD-m2}. 
Assume that $\PR(i,j)=1$ for the problem \eqref{RD-Omega} on $\Omega$. Then $\PR_m(i,j)=1$ for all sufficiently large $m$.
\end{cor}

\begin{rem}\label{rem:perturbation}
Theorems~\ref{thm:limit-blocking} and \ref{thm:limit-blocking2} imply that the family of blocking graphs (or blocking nonlinearities) forms a closed set, while Corollary~\ref{cor:propagation-open} implies that the family of non-blocking graphs (or non-blocking nonlinearities) forms an open set. \qed
\end{rem}

Next we consider deformation of $\Omega$ by unifying several outer paths into a single path beyond certain points. More precisely, as shown in Figure~\ref{fig:path-binding}, we choose points ${\rm Q}_{j_1},\ldots,{\rm Q}_{j_m}$ on the outer paths ${\Omega}_{j_1},\ldots,{\Omega}_{j_m}$, where $i, j_1,\ldots, j_m$ are mutually distinct indices, and unify these paths beyond the points ${\rm Q}_{j_1},\ldots,{\rm Q}_{j_m}$, to form a single outer path $\widetilde{\Omega}_{j_0}$. 
Here, for simplicity, we assume that all the outer paths, including $\widetilde{\Omega}_{j_0}$, have the same thickness.  
%%The following theorem holds:

\begin{figure}[h]
\begin{center}
\includegraphics[scale=0.35]
{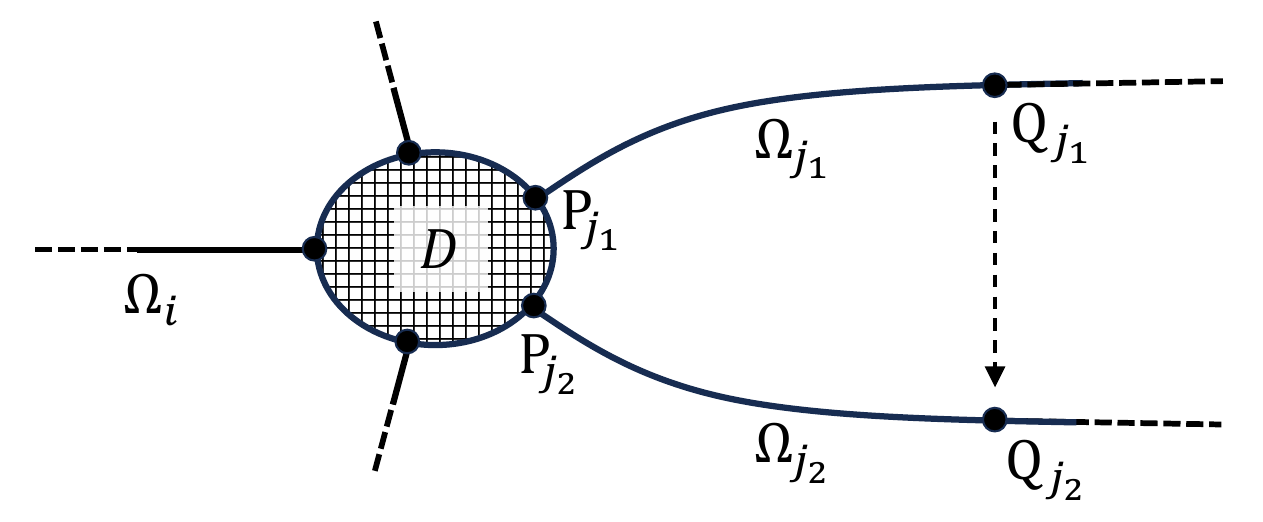}
\hspace{5pt}
\includegraphics[scale=0.35]
{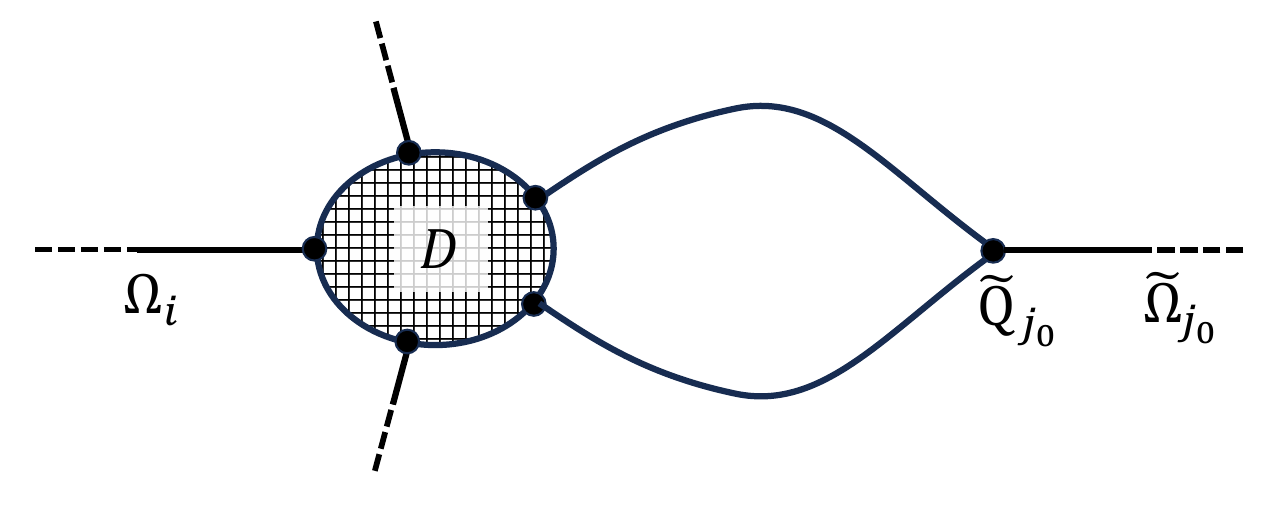}
\end{center}
\vspace{-5pt}
\caption{Unification of outer paths ${\Omega}_{j_1},\ldots,{\Omega}_{j_m}$ beyond the points ${\rm Q}_{j_1},\ldots,{\rm Q}_{j_m}$.}
\label{fig:path-binding}
\end{figure}

\begin{thm}[Unification of outer paths]\label{thm:unification}
Let $i, j_1,\ldots, j_m\in\{1,\ldots,N\}$ be mutually distinct indices and let ${\rm Q}_{j_1},\ldots,{\rm Q}_{j_m}$ be points on the outer paths ${\Omega}_{j_1},\ldots,{\Omega}_{j_m}$. Let $\widetilde{\Omega}_{j_0}$ denote a new outer path obtained by unifying ${\Omega}_{j_1},\ldots,{\Omega}_{j_m}$ beyond the points ${\rm Q}_{j_1},\ldots,{\rm Q}_{j_m}$ as shown in Figure~\ref{fig:path-binding}. Assume that $\PR(i,j_1)=\cdots=\PR(i,j_m)=1$. 
%%Here we assume that all the outer paths including $\tilde{\Omega}_{j_0}$ have the same thickness. 
If the length of the portion of $\Omega_{j_k}$ between the points ${\rm P}_{j_k}$ and ${\rm Q}_{j_k}$ is sufficiently large for all $k=1,\ldots,m$, then propagation occurs from $\Omega_i$ to $\tilde{\Omega}_{j_0}$.
\end{thm}

\begin{rem}\label{rem:symmetric-binding}
If there is symmetry among ${\Omega}_{j_1},\ldots,{\Omega}_{j_m}$, more precisely, if there is a set of isometries from $\Omega$ into itself that leave $\Omega_i$ invariant and interchange ${\Omega}_{j_1},\ldots,{\Omega}_{j_m}$ with arbitrary permutations, and if the length of the paths ${\rm P}_{j_k}{\rm Q}_{j_k}\,(k=1,\ldots,m)$ are qual, then the conclusion of Theorem~\ref{thm:unification} holds even if the lengh of the paths ${\rm P}_{j_k}{\rm Q}_{j_k}\,(k=1,\ldots,m)$ is not very long; 
see Remark~\ref{rem:symmetric-binding2}. Without such symmetry, we do not know if we can drop the length condition.
\qed
\end{rem}

%%%%%%%%%%%%%
\subsection{Star graphs and their perturbations}\label{ss:star-graphs}

In this section we consider star graphs and their perturbations. The first half of this section, in which we derive sharp criteria for propagation and blocking on star graphs  (Theorem~\ref{thm:star} and its corollaries), is for a large part restatement of known results that are reformulated in our framework of propagation/blocking dichotomy.  In the second half of this section, we consider perturbations of star graphs (Figure \ref{fig:star-perturbation}). We also discuss the case where an arbitrary graph is attached to a star graph at a far distance (Figure~\ref{fig:local-star}). 

Propagation and blocking on star graphs were studied in detail by Jimbo and Morita \cite{JM2019, JM2021}. The paper \cite{JM2019} deals with star graphs with equal edge thickness, while \cite{JM2021} deals with star graphs whose edges have unequal diffusion coefficients, which is equivalent to considering unequal edge thickness (after a suitable change of variables as mentioned in Remark~\ref{rem:unequal-diffusion}). 

Among other things, \cite{JM2019, JM2021} proved existence of a front-like entire solution on star graphs, which is basically the same as $\widehat{u}_i(t,x)$ of our Theorem~\ref{thm:u-hat}, though \cite{JM2019, JM2021} did not show the uniqueness nor time-monotonicity of this entire solution, which is crucial in the present paper in defining the notion of limit profile $\widehat{v}_i$ in \eqref{v-hat}.  Then \cite{JM2019, JM2021} derived a sharp criterion on the existence and non-existence of a barrier stationary solution (\cite[Theorem 3.2]{JM2021}). 
However, \cite{JM2019,JM2021} did not discuss whether or not their barrier stationary solution really blocks the front-like entire solution. In that sense, there is ambiguity in their notion of propagation and blocking. 
Our Theorem~\ref{thm:star} below is a restatement of \cite[Theorem 3.2]{JM2021} in a much clearer framework based on our notion of propagation/blocking dichotomy given in Definition~\ref{def:propagation}. 
On the other hand, the perturbation results in the second half of this section are totally new.

Before stating our results, we recall that
\[
F(a)<0<F(1),
\]
as stated in \eqref{F}. 
The following theorem gives a sharp criterion for propagation and blocking:

\begin{thm}[Propagation/blocking criterion]\label{thm:star}
Let $\Omega$ be a star graph and $\rho_i\,(i=1,\ldots,N)$ be the thickness of the outer path $\Omega_i$. Define
\[
R_i:=\frac{\sum_{k\ne i}\rho_k}{\rho_i}\ \ \ (i=1,\ldots,N).
\]
Then the following criterion holds:
\begin{subnumcases}{\label{star-graph-unequal}}
\label{star-graph-unequal-a} %%%%%%%%%% LABEL %%%%%%%%%%
\ F(1)+\left(R_i^2-1\right)F(a)>0 \ \ \Rightarrow \ \PR(i,j)=1 \ \ \hbox{for all}\ j\ne i,\\[4pt]
\label{star-graph-unequal-b} %%%%%%%%%% LABEL %%%%%%%%%%
\ F(1)+\left(R_i^2-1\right)F(a)\leq 0 \ \ \Rightarrow \ \PR(i,j)=0 \ \ \hbox{for all}\ j\ne i.
\end{subnumcases}
%%
%%\begin{equation}\label{star-graph-unequal}
%%F(1)+\left(R_i^2-1\right)F(a)\,
%%\begin{cases}
%%\ > 0 &\Rightarrow \ \PR(i,j)=1 \ \ \hbox{for all}\ j\ne i\\[3pt]
%%\ \leq 0 &\Rightarrow \ \PR(i,j)=0 \ \ \hbox{for all}\ j\ne i.
%%\end{cases}
%%\end{equation}
Furthermore, in the former case, $\widetilde{v}_i\equiv 1$ on $\Omega$. 
\end{thm}

Since $F(a)<0$, we have the following immediate corollary:

\begin{cor}\label{cor:star-propagation}
Assume $R_i\leq 1$. Then $\PR(i,j)=1$ for all $j\in\{1,\ldots,N\}$ with $j\ne i$.
\end{cor}

The above corollary states that if the thickness of the outer path $\Omega_i$ is not smaller than the total sum of the thickness of other outer paths, then propagation occurs from $\Omega_i$ to all other outer paths regardless of the choice of the nonlinearity $f$ so long as it satisfies \eqref{f}.

\begin{rem}\label{rem:star-graph}
Theorem \ref{thm:star} implies, in particular, that partial propagation does not occur on star graphs, even if the the edges have mutually unequal thickness. If propagation (or blocking) occurs from $\Omega_i$ to $\Omega_j$ for some $j\ne i$, then the same holds from $\Omega_i$ to $\Omega_k$ for all $k\ne i$.
\qed
\end{rem}

\begin{rem}\label{rem:star-graph2}
As we see in \eqref{star-graph-unequal}, the condition for propagation is given by the strict inequality ``$>0$'', while that for blocking is given by ``$\leq 0$''. This difference reflects the fact that blocking nonlinearities form a closed set while non-blocking ones form an open set; see Remark~\ref{rem:perturbation}.
\qed
\end{rem}

In the special case where $\rho_1=\cdots=\rho_N$, the assertion \eqref{star-graph-unequal} is expressed as
\begin{subnumcases}{\label{star-graph-equal}}
\label{star-graph-equal-a} %%%%%%%%%% LABEL %%%%%%%%%%
\ F(1)+\left((N-1)^2-1\right)F(a)>0 \ \ \Rightarrow \ \ \hbox{propagation},\\
\label{star-graph-equal-b} %%%%%%%%%% LABEL %%%%%%%%%%
\ F(1)+\left((N-1)^2-1\right)F(a)\leq 0 \ \ \Rightarrow \ \ \hbox{blocking},
\end{subnumcases}
%%\begin{equation}%%\label{star-graph-equal}
%%\begin{split}
%%& F(1)+\left((N-1)^2-1\right)F(a)>0 \ \ \Rightarrow \ \ \hbox{propagation},\\[2pt]
%%& F(1)+\left((N-1)^2-1\right)F(a)\leq 0 \ \ \Rightarrow \ \ \hbox{blocking},
%%\end{split}
%%\end{equation}
which is precisely the criterion given in \cite{JM2019}. (Note that \cite{JM2019} deals with the case where there are $N+1$ outer paths $\Omega_0,\Omega_1,\ldots,\Omega_N$, therefore $(N-1)^2$ is replaced by $N^2$ in \cite{JM2019,JM2021}.) 

Now we apply Theorem \ref{thm:star} to the case $N=2$ and derive conditions for one-way propagation. Let $\Omega$ be a 2-star graph as shown in Figure~\ref{fig:2star}. In this case, we have $R_1=\rho_2/\rho_1$, $R_2=\rho_1/\rho_2$. Without loss of generality, we may assume that $\rho_1\leq \rho_2$. Our result is the following:

\begin{figure}[h]
\vspace{8pt}
\begin{center}
\includegraphics[scale=0.46]
{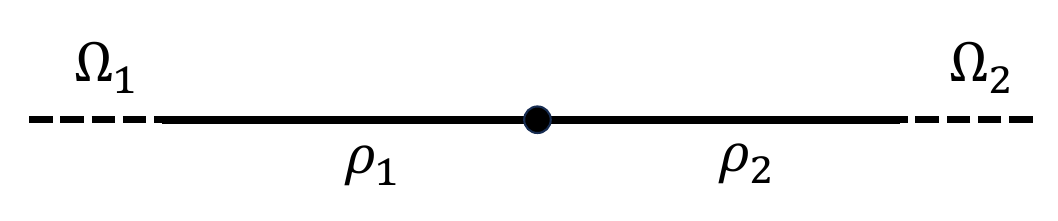}
\end{center}
\vspace{-8pt}
\caption{2-star graph.}
\label{fig:2star}
\end{figure}

\begin{cor}\label{cor:unequal-2graph}
Let $N=2$ and assume $\rho_1\leq \rho_2$. Then $\PR(2,1)=1$. Furthermore,
\begin{equation}\label{2star}
\PR(1,2)=0 \ \ \Longleftrightarrow\ \ F(1)+\left((\rho_2/\rho_1)^2-1\right)F(a)\leq 0,
\end{equation}
\end{cor}

The above corollary gives a necessary and sufficient condition for a one-way propagation on a 2-star graph. With a suitable change of variables, we can apply the above corollary to the following problem on $\R$: 
\begin{equation}\label{RD-unequal-diffusion}
\partial_t u=
\begin{cases} 
\, d_1\partial_x^2 u + f(u) & (x<0),\\[2pt]
\, d_2\partial_x^2 u+ f(u) & (x>0),
\end{cases}
\ \quad 
u(t, -0)=u(t,+0),\ \ d_1\partial_x u(t,-0)=d_2\partial_x u(t,+0).
\end{equation}
The last condition in \eqref{RD-unequal-diffusion} implies continuity of mass flux at $x=0$.

By the change of coordinates $x\to \sqrt{d_1}\,x$ for $x<0$ and $x\to \sqrt{d_2}\,x$ for $x>0$, the above problem is converted into the following form (cf. Remark \ref{rem:unequal-diffusion}):
\[
\partial_t u= \partial_x^2 u + f(u) \ \  (x\ne0), \ \quad 
u(t,-0)=u(t,+0),\ \ \sqrt{d_1}\hspace{1pt}\partial_x u(t,-0)=\sqrt{d_2}\hspace{1pt}\partial_x u(t,+0).
\]
Consequently, by Corollary \ref{cor:unequal-2graph}, we obtain the following:

\begin{cor}\label{cor:unequal-diffusion}
Assume $d_1\leq d_2$ in \eqref{RD-unequal-diffusion}. Then propagation of fronts from the region $x>0$ to the region $x<0$ is not blocked. On the other hand, propagation of fronts from the region $x<0$ to the region $x>0$ is blocked if and only if the following condition holds:
\begin{equation}\label{diffusion-blocking}
F(1)+\left(d_2/d_1 -1\right)F(a)\leq 0.
\end{equation}
\end{cor}

Note that the same condition as \eqref{diffusion-blocking} is found in the pioneering work of Pauwelussen \cite{Pa1982}, thus \eqref{diffusion-blocking} is a restatement of the result of \cite{Pa1982} in our clearer framework. 

%%%%%%%%%%%%%
Next we discuss small perturbation of star graphs (Figure~\ref{fig:star-perturbation} ({\it right})). 
We consider two cases: the case where the base star graph allows propagation between the outer paths, and the case where blocking occurs in the star graph. We begin with the first case.

\begin{figure}[h]
\begin{center}
\includegraphics[scale=0.45]
{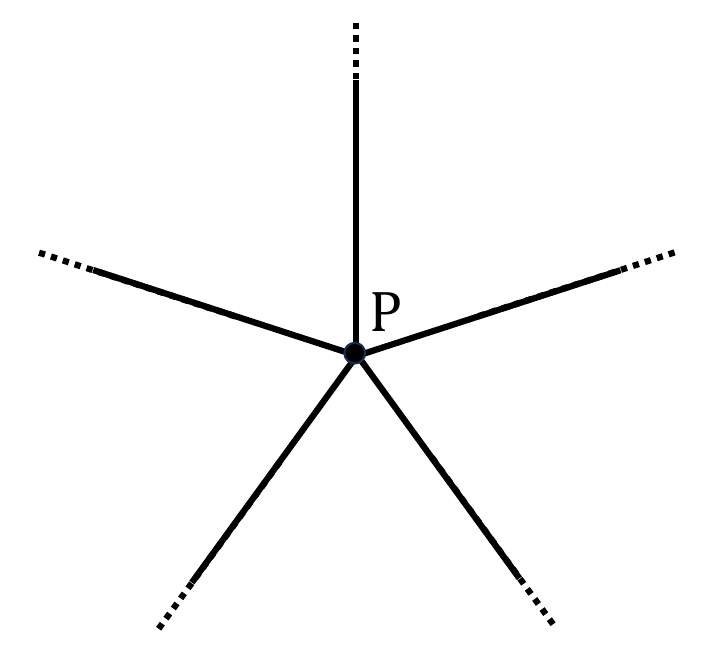}
\hspace{36pt}
\includegraphics[scale=0.45]
{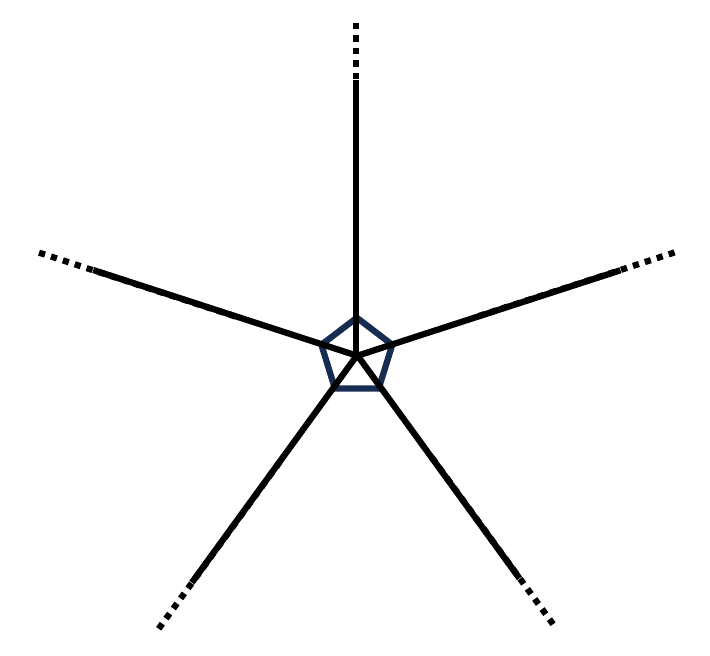}
\end{center}
\vspace{-5pt}
	\caption{An example of a star graph ({\it left}) and its small perturbation ({\it right}).}
\label{fig:star-perturbation}
\end{figure}

\begin{thm}[Perturbation of a non-blocking star graph]\label{thm:perturbation-star1}
Let $\Omega$ be a star graph with outer paths $\Omega_1,\ldots,\Omega_N$, and assume that \eqref{star-graph-equal-a} holds. Let $\Omega'$ be a graph which is obtained by replacing the center point ${\rm P}$ of $\Omega$ by a bounded finite graph $D$, and let $\PR'(i,j)$ be the propagation index for $\Omega'$ that is defined in the same manner as \eqref{Pij}.  Then $\PR'(i,j)=1$ for any $i\ne j$  
if the total length of the edges of $D$ is sufficiently small. 
\end{thm}

The above theorem is a special case of Corollary~\ref{cor:propagation-open}, which asserts the robustness of propagation under small perturbations in a much more general setting. On the other hand, in the second case where blocking occurs in the base star graph, whether or not blocking occurs in the perturbed graph does not follow from a general theory. We therefore impose the following condition which is a slightly stronger version of \eqref{star-graph-unequal-b}:
\begin{equation}\label{star-graph-unequal-bb}
F(1)+\left(R_i^2-1\right)F(a)< 0
\end{equation}

\begin{thm}[Perturbation of a blocking star graph]\label{thm:perturbation-star2}
Let $\Omega$ be a star graph with outer paths $\Omega_1,\ldots,\Omega_N$, and assume that \eqref{star-graph-unequal-bb} holds. Let $\Omega'$ be a graph which is obtained by replacing the center point ${\rm P}$ of $\Omega$ by a bounded finite graph $D$, and let $\PR'(i,j)$ be the propagation index for $\Omega'$ that is defined in the same manner as \eqref{Pij}.  Then $\PR'(i,j)=0$ for any $i\ne j$ if the total length of the edges of $D$ is sufficiently small. 
\end{thm}

The above theorem will be proved by constructing an upper barrier that blocks the fronts. The Gauss-Green formula \eqref{Green1} and harmonic functions will play a key role in the construction.

%%%%%%%%%%%%%
\vskip 8pt
\underbar{\bf Faraway perturbation of a star graph}

\vspace{5pt}
Next we consider a graph that is obtained by attaching an arbitrary graph $D_0$ to a star graph at a far distance, either in front of the center point or behind it (see Figure~\ref{fig:local-star}). 

\begin{figure}[h]
\begin{center}
\includegraphics[scale=0.35]
{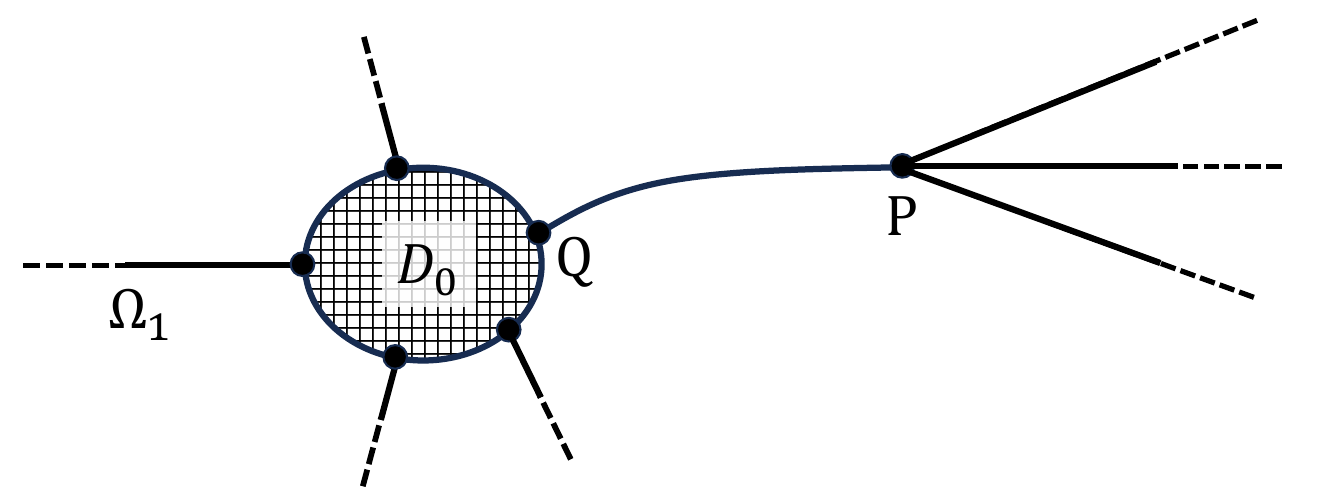}
\hspace{2pt}
\includegraphics[scale=0.35]
{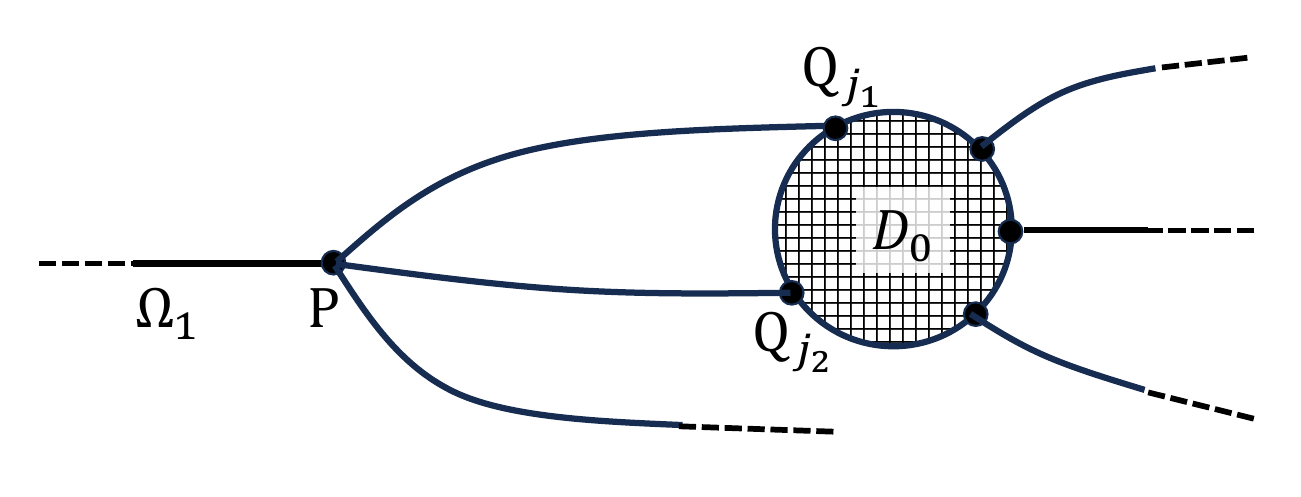}
\end{center}
\vspace{-5pt}
\caption{Joining an arbitrary graph $D_0$ to a star graph; in front ({\it left}) and behind ({\it right}).}
\label{fig:local-star}
\end{figure}

In both cases, the star graph has the center point ${\rm P}$ and $N'$ outer paths. For simplicity, we only consider the case where all the edges have the same thickness, but more general cases can be treated just similarly. We assume that \eqref{star-graph-equal-b} with $N$ replaced by $N'$ holds strictly, namely 
\begin{equation}\label{star-graph-equal-bb}
 F(1)+\left((N'-1)^2-1\right)F(a)< 0.
\end{equation}
Therefore, it is a blocking star graph in a strict sense. 
Our question is whether or not this blocking property remains to hold when another graph $D_0$ is attached at a far distance.

The following two theorems assert that the blocking property of the star graph is preserved if the distance between $D_0$ and the junction point ${\rm P}$ of the star graph is sufficiently large.

\begin{thm}\label{thm:local-star1}
Let $\Omega$ be as shown in Figure~\ref{fig:local-star} (left). More precisely, the center graph $D$ of $\Omega$ consists of an arbitrary bounded finite metric graph $D_0$ and an edge ${\rm QP}$, with the vertex ${\rm Q}$ belonging to $D_0$, 
that is, $D=D_0\cup {\rm QP}$ with ${\rm Q}\in D_0$, ${\rm P}\not\in D_0$. Assume that there are $N'-1$ outer paths $\Omega_{j_1},\ldots, \Omega_{j_{N'-1}}$ stretching from the vertex ${\rm P}$, while the outer path $\Omega_1$ stretches from a vertex of $D_0$.  Assume also that \eqref{star-graph-equal-bb} holds. If the length of the edge ${\rm QP}$ is sufficiently large, blocking occurs beyond ${\rm P}$, that is, $\PR(1,j_k)=0$ for all $k\in\{1,\ldots,N'-1\}$.
\end{thm}

\begin{thm}\label{thm:local-star2}
Let $\Omega$ be as shown in Figure~\ref{fig:local-star} (right). More precisely, let $D_0$ be an arbitrary bounded finite metric graph and ${\rm P}$ be a vertex outside $D_0$, from which $N'$ edges emanate, including the outer path $\Omega_1$. Among the remaining $N'-1$ edges emanating from ${\rm P}$, some are connected to $D_0$ at vertices ${\rm Q}_{j_k}\,(k=1,\ldots,m)$ with $1\leq m\leq N'-1$, and the rest are outer paths, if any. Thus $D=D_0\cup{\rm PQ}_{j_1}\cup\cdots\cup{\rm PQ}_{j_m}$. 
There are $N-N'+m$ outer paths stretching from $D_0$, hence the total number of outer paths of $\Omega$ is $N$, where $N$ can be larger than or smaller than or equal to $N'$. 
Assume that \eqref{star-graph-equal-bb} holds. If the length of the edge ${\rm PQ}_{j_k}$ is sufficiently large for all $k=1,\ldots,m$, blocking occurs beyond ${\rm P}$, that is, $\PR(1,j)=0$ for all $j\in\{2,\ldots,N\}$.
\end{thm}

We shall prove the above theorems by the phase plane analysis.

%%%%%%%%%%%%%
\subsection{Examples of partial and one-way propagation}\label{ss:examples}

In this section we give an example of partial propagation and that of one-way propagation. Throughout this section, we assume that all the edges have the same thickness since otherwise it would be much easier to construct such examples.

%%\begin{exmp}[Partial propagation]\label{ex:partial}
\begin{prop}[Partial propagation]\label{prop:partial}
Let $\Omega$ be as shown in Figure~\ref{fig:partial-one-way} ({\it left}). Assume that
\begin{equation}\label{partial-F}
F(1)+(2^2-1)F(a)>0>F(1)+(3^2-1)F(a).
\end{equation}
If the length of the edge ${\rm P}_1{\rm P}_2$ is sufficiently large, then the following holds:
\begin{equation}\label{partial}
\PR(1,5)=1,\quad \PR(1,j)=0\ \ (j=2,3,4).
\end{equation}
\end{prop}

\begin{proof}
We first remark that the condition \eqref{partial-F} implies that propagation occurs for a 3-star graph while blocking occurs for a 4-star graph. The fact that $\PR(1,j)=0\,(j=2,3,4)$ is therefore a direct consequence of Theorem~\ref{thm:local-star1}. It remains to show that $\PR(1,5)=1$.

Let $L$ denote the length of the edge ${\rm P}_1{\rm P}_2$. For each given $L>0$, the limit profile associated with the outer path $\Omega_1$ is denoted by $\widehat{v}_1^L$, as it depends on the parameter $L$. By Theorem~\ref{thm:dichotomy}. in order to prove $\PR(1,5)=1$, it suffices to show that $\widehat{v}_1^L({\rm P}_1)>\beta$ if $L$ is sufficiently large. Assume the contrary. Then there exists a sequence $L_k\to\infty$ such that $\widehat{v}_1^{L_k}({\rm P}_1)\leq\beta$ for $k=1,2,3,\ldots$. By choosing a subsequence if necessary, we may assume that $\widehat{v}_1^{L_k}$ converges to a stationary solution $v^\infty$ with $v^\infty({\rm P}_1)\leq\beta$, and this stationary solution is defined on a 3-star graph with the center point ${\rm P}_1$, since ${\rm P}_2$ escapes to infinity.  Furthermore, it is clear that $v^\infty$ tends to $1$ at infinity along the outer path $\Omega_1$. However, by \eqref{partial-F}, no blocking occurs for a 3-star graph, therefore it is impossible to have $v^\infty({\rm P}_1)\leq\beta$. This contradiction shows that $\widehat{v}_1^L({\rm P}_1)>\beta$ if $L$ is sufficiently large, hence $\PR(1,5)=1$. The proposition is proved.
\end{proof}

\begin{figure}[h]
\begin{center}
\includegraphics[scale=0.33]
{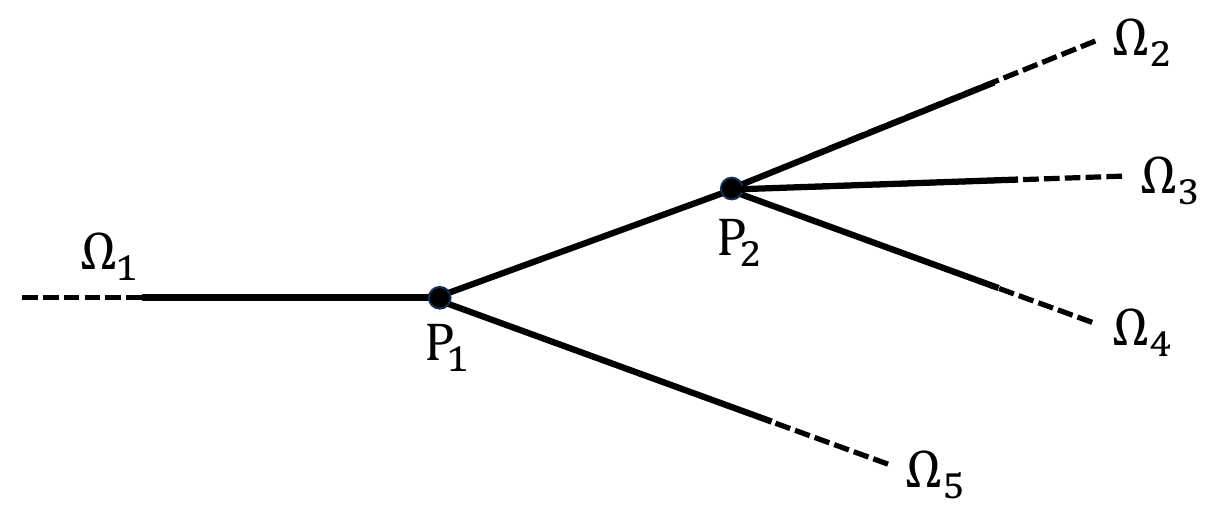}
\hspace{8pt}
%%\vspace{10pt}
\includegraphics[scale=0.33]
{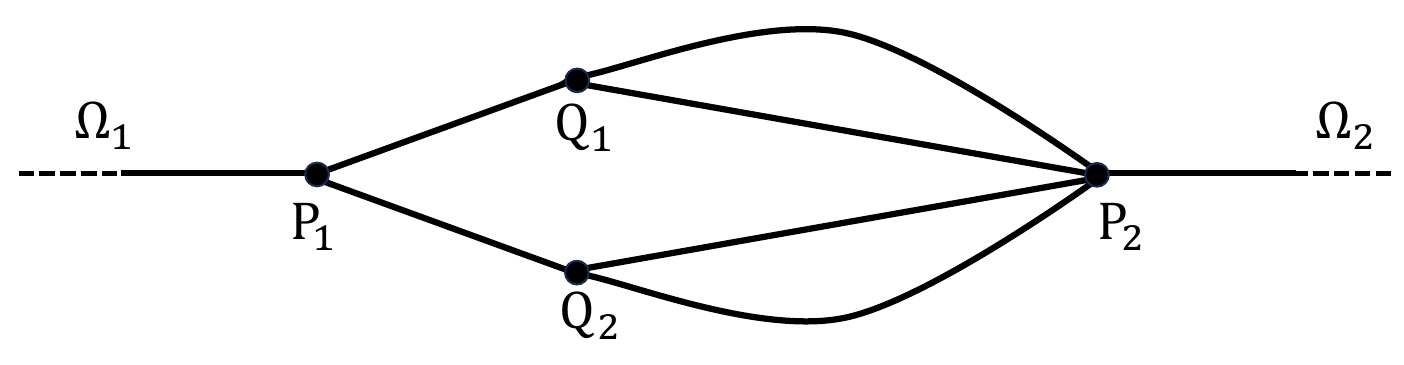}
\end{center}
\vspace{-5pt}
\caption{An example of partial propagation ({\it above}) and one-way propagation ({\it below}).}
\label{fig:partial-one-way}
\end{figure}

Next we discuss the example given in Figure~\ref{fig:partial-one-way} ({\it right}). Our result is the following:

\begin{prop}[One-way propagation]\label{prop:one-way}
Let $\Omega$ be as shown in Figure~\ref{fig:partial-one-way} ({\it right}). Assume that
\begin{equation}\label{one-way-F}
F(1)+(2^2-1)F(a)>0>F(1)+(4^2-1)F(a).
\end{equation}
If the edges ${\rm P}_1{\rm Q}_1$, ${\rm P}_1{\rm Q}_2$, ${\rm P}_2{\rm Q}_1$ and ${\rm P}_2{\rm Q}_2$ are all sufficiently long, then the following holds:
\begin{equation}\label{one-way}
\PR(1,2)=1,\quad\ \PR(2,1)=0.
\end{equation}
\end{prop}

\begin{proof}
We fist note that the condition \eqref{one-way-F} implies that propagation occurs for a 3-star graph while blocking occurs for a 5-star graph. Next we remark that  the graph in Figure~\ref{fig:partial-one-way} ({\it right}) is obtained by unifying the four outer paths in Figure~\ref{fig:IJM-JM} ({\it right}) beyond certain points that are far away. Before this unification, the graph is exactly the same as in Figure~\ref{fig:IJM-JM} ({\it right}), therefore propagation occurs from $\Omega_1$ to the remaining four outer paths if  ${\rm P}_1{\rm Q}_1$ and ${\rm P}_1{\rm Q}_2$ are sufficiently long. Therefore, by Theorem~\ref{thm:unification}, propagation still occurs after unification, namely from $\Omega_1$ to $\Omega_2$ in Figure~\ref{fig:partial-one-way} ({\it right}), provided that the edges connecting ${\rm Q}_1, {\rm Q_2}$ and ${\rm P}_2$ are sufficiently long. This proves $\PR(1,2)=1$. Next we discuss propagation/blocking from $\Omega_2$ to $\Omega_1$. Since the graph is locally close to a 5-star graph around ${\rm P}_2$, we see from Theorem~\ref{thm:local-star2} that the propagation from $\Omega_2$ is blocked, if the edges connecting ${\rm P}_2$ and ${\rm Q}_1, {\rm Q_2}$ are all sufficiently long, hence $\PR(2,1)=0$. The proposition is proved. 
\end{proof}

%%%%%%%%%%%%%
\subsection{Graph with a reservoir and incomplete invasion}\label{ss:reservoir}

In this section, we consider a graph that has a ``reservoir'' type subgraph as shown in Figure~\ref{fig:3star-reservoir}. For the clarity of the arguments, we assume that all the edges in $\Omega$ have the same thickness. Our goal is to show that the value of the limit profile $\widehat{v}_i$ within the reservoir is always close to $0$ regardless of the configurations of the remaing parts of the graph.

\begin{figure}[h]
\vspace{5pt}
\begin{center}
\includegraphics[scale=0.45]
{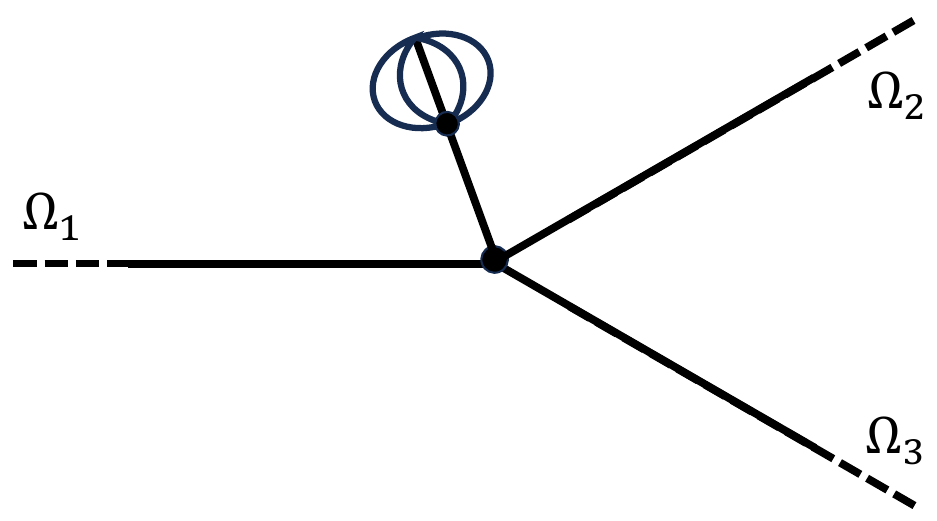}
\end{center}
\vspace{-5pt}
	\caption{A graph with a ``reservoir''.}
\label{fig:3star-reservoir}
\end{figure}

We first give a detailed description of the reservoir, which we call ${\mathcal R}_0$. As shown in Figure~\ref{fig:reservoir}. ${\mathcal R}_0$ consists of an edge $E_0={\rm P}_0{\rm Q}_0$ and a bounded finite graph $\Delta_0$ that is attached to $E_0$ via the vertex ${\rm P}_0$ of $E_0$. The graph ${\mathcal R}_0$ is connected to the rest of the graph $\Omega$ via the other vertex ${\rm P}_0$ of $E_0$, and there is no other common point between ${\mathcal R}_0$ and the rest of $\Omega$.
Let $L$ denote the length of $E_0$ and $|\Delta_0|$ the total length of the edges of $\Delta_0$. Let $\mu_1:=\mu_1(\Delta_0)$ be the smallest positive eigenvalue of the Laplacian $-\Delta$ on $\Delta_0$ as defined in Section~\ref{ss:poincare} by \eqref{mu-eigenvalue}.

\begin{figure}[h]
\vspace{5pt}
\begin{center}
\includegraphics[scale=0.45]
{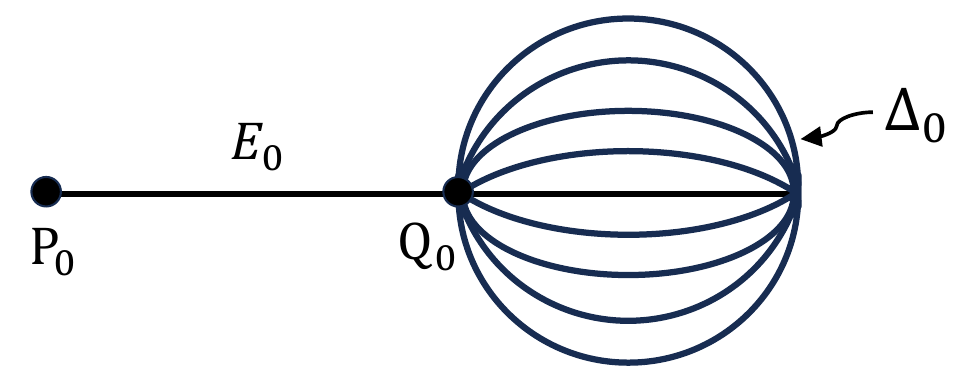}
\end{center}
\vspace{-5pt}
\caption{Detailed view of the reservoir.}
\label{fig:reservoir}
\end{figure}

Next we introduce some constants associated with the nonlinearity $f$. Let $F(s)=\int_0^s f(r)dr$ and $W(s):=-F(s)$. Then, by \eqref{f}, \eqref{F} and \eqref{beta}, we have
\[
W(0)=0,\quad W(s)>0\ (0<s<\beta),\quad W(\beta)=0,\quad W(s)<0\ (\beta<s\leq 1). 
\] 
Let $\delta,$ be a constant such that $0<\delta<a$. Then there exists a constant $\mu^*>0$ such that
\begin{equation}\label{delta-mu}
\frac{\mu^*}{2} (s-\delta)^2 + W(s) \geq \sigma\quad (\,{}^\forall s\in[0,1]\,) 
\end{equation}
for some constant $\sigma>0$. Fix such constants $\delta, \mu^*, \sigma$. 
(Note that the choice of such constants is certainly not unique, in particular, \eqref{delta-mu} is satisfied if $\mu^*$ is sufficiently large, but we try not to choose a too large $\mu^*$, the reason being clear from the statement of the theorem below.)

The following theorem holds:

\begin{thm}[Value of the solution in the reservoir]\label{thm:reservoir}
Suppose that $\Omega$ contains a reservoir type subgraph ${\mathcal R}_0:=\Delta_0\cup E_0$ as described above, and assume that $\mu_1(\Delta_0)\geq \mu^*$ and that
\begin{subnumcases}{\label{reservoir-condition}}
\label{reservoir-condition-a} %%%%%%%%%% LABEL %%%%%%%%%%
\ \frac{1}{2L}+L\left(F(1)-F(a)\right)\leq \sigma |\Delta_0| \quad \hbox{if}\ \ L\leq \left(\sqrt{2(F(1)-F(a))}\right)^{-1},\\
\label{reservoir-condition-b} %%%%%%%%%% LABEL %%%%%%%%%%
\ \sqrt{2(F(1)-F(a))} \leq \sigma |\Delta_0|\quad\ \  \hbox{if}\ \ L> \left(\sqrt{2(F(1)-F(a))}\right)^{-1}.
\end{subnumcases}
%%\begin{equation}\label{reservoir-condition}
%%\frac{1}{2L}+L\left(F(1)-F(a)\right)\leq \sigma |\Delta_0|.
%%\end{equation}
Then, for any $i\in\{1,\ldots,N\}$, the limit profile $\widehat{v}_i$ satisfies 
\begin{equation}\label{v-reservoir}
\frac{1}{|\Delta_0|}\int_{\Delta_0} \widehat{v}_i \, dx \leq \delta.
\end{equation}
\end{thm}

\begin{rem}\label{rem:reservoir}
In the special case where $\Delta_0$ is a melon-shaped graph given in Example~\ref{ex:melon}, the assumption $\mu_1(\Delta_0)\geq \mu^*$ in the above theorem is expressed as follows.
\[
\frac{\pi^2}{L_0^2}\geq \mu^*.
\]
On the other hand, we have $|\Delta_0|=mL_0$, where $m$ is the number of edges in $\Delta_0$. Therefore, the assumptions of Theorem~\ref{thm:reservoir} are fulfilled if $L_0$ is relatively small and $m$ is sufficiently large.  \qed
\end{rem}

%%%%%%%%%%%%%
%%\subsection{Complete and incomplete propagation}\label{ss:complete}

\vskip 8pt\noindent
\underbar{\bf Complete and incomplete invation}

If propagation occurs from $\Omega_i$ to any other outer paths, that is, if $\PR(i,j)=1$ for any $j\in\{1,\ldots,N\}$ with $j\ne i$, then by definition we have
\begin{equation}\label{propagation-all}
\lim_{x_j\to\infty}\widehat{v}_i{}_{\restr \Omega_j}(x_j)=1\quad (\,{}^\forall j\in\{1,\ldots,N\}).
\end{equation}

\begin{df}\label{def:complete}
Let $i\in\{1,\ldots,N\}$.  %%and assume that \eqref{propagation-all} holds. 
We say that {\it complete invasion} occurs from $\Omega_i$ if $\widehat{v}_i(x)\equiv 1$ on $\Omega$. We say that {\it incomplete invasion} occurs from $\Omega_i$ if $\widehat{v}_i$ satisfies \eqref{propagation-all} but $\widehat{v}_i(x)< 1$ on $\Omega$. 
\end{df}

If $\Omega$ has a reservoir-type subgraph satisfying the condition of Theorem~\ref{thm:reservoir}, then complete invasion cannot occur, therefore we have either blocking or incomplete invasion. 
On the other hand, if $\Omega$ is a star graph, then by Theorem~\ref{thm:star}, we have either blocking or complete invasion, hence incomplate invasion never occurs. The following theorem asserts that the same is true of perturbed star graphs:

\begin{thm}[Complete invasion on perturbed star graphs]\label{thm:complete}
Let the assumptions of Theorem~\ref{thm:perturbation-star1} hold. Then complete invasion occurs if the total length of the edges of the center graph $D$ is sufficiently small.
\end{thm}

%%%%%%%%%%%%%
\subsection{Behavior of more general solutions}\label{ss:general}

So far we discussed propagation and blocking in the sense of Definition~\ref{def:propagation}, which is based on the notion of limit profile $\widehat{v}_i(x):=\lim_{t\to\infty} \widehat{u}_i(t,x)$. In other words, our attention has been focused on whether or not the front-like solution $\widehat{u}_i(t,x)$ is blocked or not. In this section, we consider solutions of the Cauchy problem for \eqref{RD-Omega} by imposing the initial condition
\begin{equation}\label{u0}
u(0,x)=u_0(x) \quad \hbox{on}\ \ \Omega,
\end{equation}
where $u_0$ is a continuous function whose support is contained in $\Omega_i$. The following theorem shows that, if $u_0$ is relatively large, then the corresponding solution $u(t,x)$ convereges to $\widehat{v}_i(x)$ as $t\to\infty$. Thus $\widehat{v}_i$ is not just the limit of the front-like solution $\widehat{u}_i$ but is also the limit of a large class of solutions of \eqref{RD-Omega}. Therefore the notion of propagation and blocking defined in Definition~\ref{def:propagation} has much broader implications than initially intended.

Before stating our theorem, we define auxiliary functions $H(x), \Psi(x)$ on $\R$. Let $H(x), x\geq 0,$ be a function with $0\leq H<1$ that is defined uniquely by the condition:
\[
H''+f(H)=0\ \ (0<x<+\infty),\quad H(0)=0,\quad \lim_{x\to\infty}H(x)=1.
\]
This function corresponds to the portion of the stable manifold of $(1,0)$ in the phase-portrait in Figure~\ref{fig:phase-portrait} (in the region $0\leq v<1,\,v_x>0$). Next let $\Psi(x)$ be a solution of the problem:
\[
\Psi'' + f(\Psi)=0\ \ (|x|<R),\quad \Psi(\pm R)=0,\quad 0<\Psi(x)<1\ \ (|x|<R).
\]
It is well known that such a solution exists if $R>0$ is sufficiently large, and it corresponds to the portion of the orbit that lies between the homoclinic orbit and the stable and unstable manifolds of $(1,0)$ in the phase portrait. In particular, we have $\Psi(0)>\beta.$ 

Next, for each $b\in\R$, we define a function $\Psi^b(x)$ by
\begin{equation}\label{Psi-b}
\Psi^b(x)=
\begin{cases}
\,\Psi(x-b) \ & (b-R\leq x \leq b+R),\\
\ 0 \ & (|x-b|>R). 
\end{cases}
\end{equation}
Then $\Psi^b$ is a subsolution of the equation $\partial_t u = \partial_x^2 u+f(u)$ on $\R$. 
The following theorem holds:

\begin{thm}[Behavior of more general solutions]\label{thm:general}
Let $u(t,x)$ be a solution of \eqref{RD-Omega} whose initial data $u_0(x)$ has a support in $\Omega_i$ and satisfies either of the following conditions on $\Omega_i$:
\begin{itemize}\setlength{\itemsep}{0pt}
\item[{\rm (a)}] $\Psi^b(x_i)\leq u_0{}_{\restr \Omega_i}(x_i)\leq H(x_i)$ for all $x_i\in [0,\infty)$ for some $b\geq R$;
\item[{\rm (b)}] $0\leq u_0{}_{\restr \Omega_i}(x_i)\leq H(x_i)$ for all $x_i\in [0,\infty)$ and there exists $\sigma>0$ such that $u_0{}_{\restr \Omega_i}(x_i)\geq a+\sigma$ on some interval $I\subset (0,\infty)$ whose length is sufficiently large.
\end{itemize}
Then $\lim_{t\to\infty} u(t,x)=\widehat{v}_i(x)$ on $\Omega$.
\end{thm}

The above theorem is an analogue of \cite[Theorem 3]{BHM2025}, in which one of the present authors studied front propagation in multi-dimensional bistable reaction-diffusion equations through a perforated wall.

%%%%%%%%%%%%%%%%%%%%%%%%%%%
%%%%%%%%%%%%%%%%%%%%%%%%%%%
\section{Preliminaries}\label{s:preliminaries}

%%%%%%%%%%%%%
\subsection{Phase portrait and the pulse solution}\label{ss:phase-portrait}

Let us recall that a function $v(x)$ is a solution of the stationary problem
\begin{equation}\label{stationary2}
\Delta_\Omega v+ f(v)=0\quad\hbox{on}\ \ \Omega
\end{equation}
if it is continuous on $\Omega$ and satisfies the equation $\partial_x^2 v + f(v)=0$ on each edge and the Kirchhoff condition \eqref{Kirchhoff} (or, more generally, \eqref{Kirchhoff2}) at each vertex.  Figure~\ref{fig:phase-portrait} shows a typical phase-portrait for the equation $\partial_x^2 v + f(v)=0$ on $\R$ when $f$ is a bistable nonlinearity satisfying \eqref{f}. Thus the behavior of the solution of equation \eqref{stationary2} on each edge of $\Omega$ is represented by an orbit, or its portion, in this phase portrait. Here $\beta$ is the constant that appears in \eqref{beta}. 
The vertical dotted line at the center of Figure~\ref{fig:phase-portrait} indicates the line $v=a$. Since $\partial_x(\partial_x v)=\partial_x^2 v =-f(a)=0$ on this line, all the orbits in the phase portrait cross this line horizontally, and, along each orbit, $v_x$ attains either the maximum or the minimum on this line.  

\begin{figure}[h]
\begin{center}
\includegraphics[scale=0.45]
{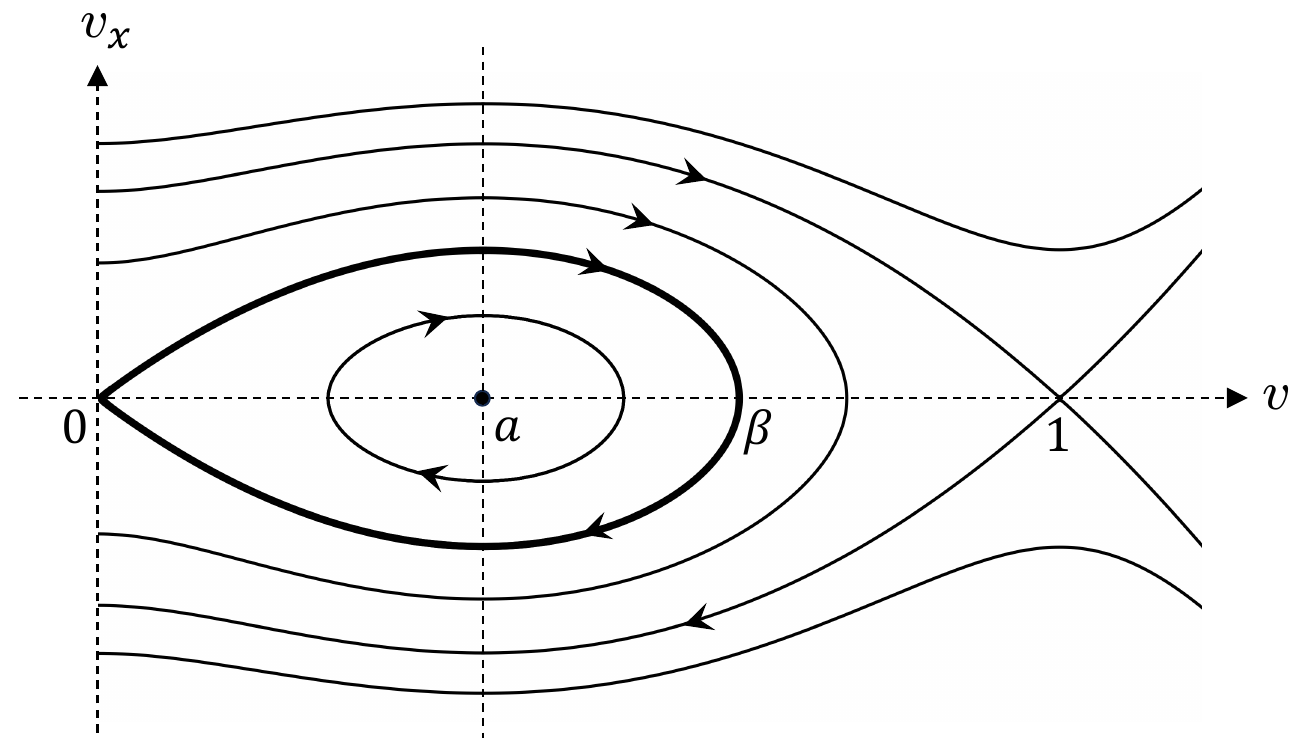}
\end{center}
\vspace{-8pt}
\caption{Phase portrait for $\partial_x^2 v + f(v)=0$ on $\R$.}
\label{fig:phase-portrait}
\end{figure}

The orbit that is marked in a thick line in Figure~\ref{fig:phase-portrait} is the homoclinic orbit emanating from $(0,0)$, which corresponds to the symmetrically decreasing solution $V(x)$, which we call the {\it pulse solution} (Figure~\ref{fig:pulse-solution}). Clearly, any spatial shift of $V$ corresponds to the same homoclinic orbit, but we define $V$ to be the one that takes its maximun at $x=0$. As one can see from the phase portrait, no other stationary solution tends to $0$ as $x\to +\infty$ nor as $x\to -\infty$.

\begin{figure}[h]
\begin{center}
\includegraphics[scale=0.45]
{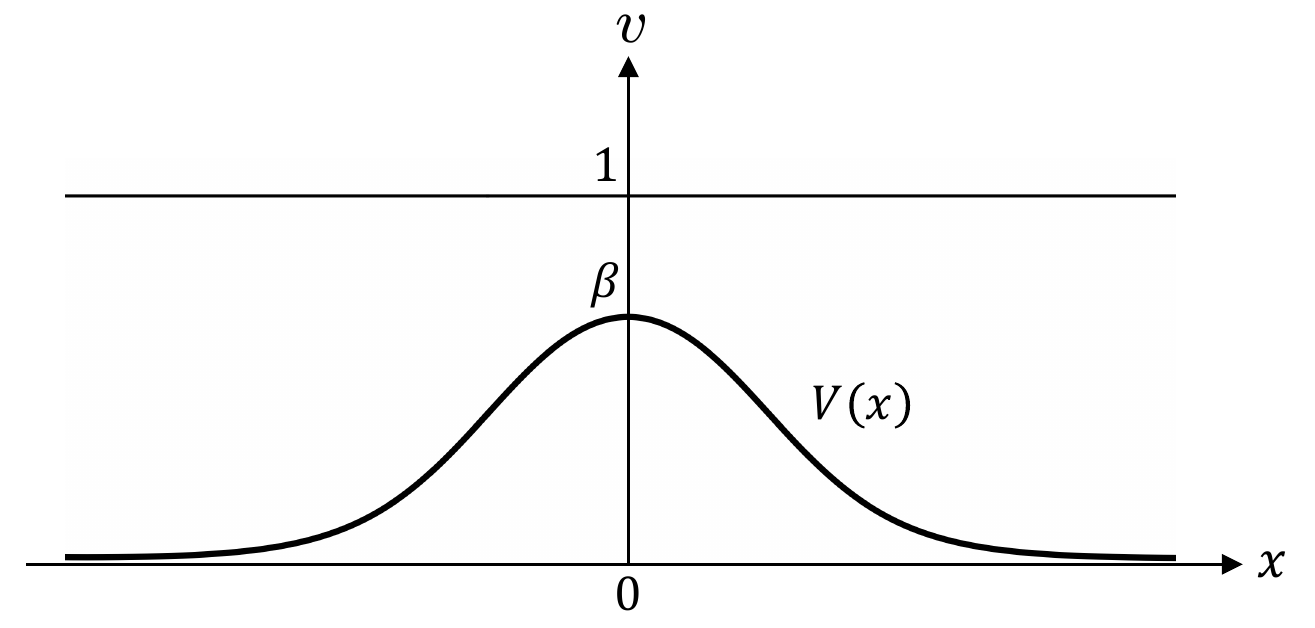}
\end{center}
\vspace{-8pt}
\caption{Symmetrically decreasing stationary solution (pulse solution).}
\label{fig:pulse-solution}
\end{figure}

%%%%%%%%%%%%%
\subsection{Continuity of $\partial_t u$ and $\Delta u$}\label{ss:ut-continuous}

Let $G$ be any metric graph. 
In Section~\ref{ss:metric-graph}, we defined the notion of a solution of \eqref{RD-G} as a function satisfying the conditions (a), (b), (c). In this section we show that if $u(t,x)$ satisfies (a), (b), (c), then automatically $\partial_t u$ and $\Delta u$ are continuous on the entire graph $G$. The continuity of $\partial_t u$ will be needed in the proof of Proposition~\ref{prop:energy}.

\begin{lem}\label{lem:ut-continuous}
Let $u(t,x)$ be a solution of \eqref{RD-G}. Then $\partial_t u$ and $\Delta_{G}u$ are continuous on $(t_0,t_1)\times G$.
\end{lem}

\begin{proof}
Continuity of $\partial_t u$ and $\partial_x^2 u$ in the interior of each edge of $G$ is clear from the standard interior estimates for one-dimensional parabolic equations. Let us show that $\partial_t u$ and $\partial_x^2 u$ are continuous on each edge up to the end points. 
%%We only give an outline of the proof.

Let ${\rm Q}$ be an arbitrary vertex of $G$ and let $E_1,\ldots,E_m$ be the edges emanating from ${\rm Q}$ with thickness $\rho_j\,(j=1,\ldots,m)$. Let $x_j\,(j=1,\ldots,m)$ denote the coordinates of $E_j$ with $x_j=0$ corresponding to ${\rm Q}$, and let $u_j(t,x_j)$ denote the restriction of $u$ to the edge $E_j\,(j=1,\ldots,m)$. We fix a constant $L>0$ that does not exceed the minimal length of $E_1,\ldots,E_m$. 

By the continuity of $u$ and the Kirchhoff condition, we have
\begin{equation}\label{u-Q}
u_1(t,0)=\cdots=u_m(t,0),\qquad \sum_{j=1}^m \rho_j\partial_{x_j}u_j(t,0)=0.
\end{equation}
Now we choose a vector $\vec{\theta}:=(\theta_1,\ldots,\theta_m)$ such that 
\begin{equation}\label{theta}
0\leq\theta_j\leq 1\ \ (j=1,\ldots,m),\qquad \sum_{j=1}^m \theta_j\rho_j =\sum_{j=1}^m (1-\theta_j)\rho_j 
\end{equation}
and define a function $w(t,x)$ on $(t_0,t_1)\times(-L,L)$ by
\begin{equation}\label{w}
w(t,x)=
\begin{cases}
\, \sum_{j=1}^m \theta_j\rho_j u_j(t,x) & \hbox{on}\ \ (t_0,t_1)\times[0,L),\\[6pt]
\, \sum_{j=1}^m (1-\theta_j)\rho_j u_j(t,-x) & \hbox{on}\ \ (t_0,t_1)\times(-L,0).
\end{cases}
\end{equation}
Then $w$ satisfies the following equation on the intervals $0<x<L$, $-L<x<0$:
\[
\partial_t w =\partial_x^2 w + g_+(t,x)\ \ \hbox{if}\ \ x>0,\quad\ 
\partial_t w =\partial_x^2 w + g_-(t,x)\ \ \hbox{if}\ \ x<0,
\]
where
\begin{equation}\label{g+-}
g_+(t,x)=\sum_{j=1}^m \theta_j\rho_j f(u_j(t,x)),\quad
g_-(t,x)=\sum_{j=1}^m (1-\theta_j)\rho_j f(u_j(t,-x)).
\end{equation}
Note also that
\[
\partial_x w(t,+0)=\sum_{j=1}^m \theta_j\rho_j u_j(t,0),\quad
\partial_x w(t,-0)=-\sum_{j=1}^m (1-\theta_j)\rho_j \partial_x u_j(t,0)
\]
Thus, by \eqref{u-Q} and \eqref{theta}, $w$ and $\partial_x w$ are both continuous at $x=0$. Therefore $w$ is a weak solution of the following equation on $(t_0,t_1)\times(-L,L)$:
\begin{equation}\label{w-eq}
\partial_t w = \partial_x^2 w + g(t,x),
\end{equation}
where $g(t,x)=g^+(t,x)$ for $x\geq 0$ and $g(t,x)=g^-(t,x)$ for $x<0$. Since $g$ and $w$ are uniformly bounded, standard interior parabolic estimates show that, for any $0<\alpha<1$ and $t'_0\in(t_0,t_1)$, $w$ is bounded in $C^{\frac{\alpha}{2}, 1+\alpha}$ on the region $(t_0',t_1)\times[-L/2,L/2]$. 
The above function $w$ depends on the choice of the vector $\vec{\theta}$ satisfying \eqref{theta}. We now choose linearly independent vectors $\vec{\theta}^{(k)}\,(k=1,\ldots,m)$ satisfying \eqref{theta} and consider corresponding functions $w^{(k)}(t,x)$ for $k=1,\ldots,m$. (Such vectors can be obtained, for example, by setting $\vec{\theta}^{(1)}=(\frac12,\ldots,\frac12)$ and $\vec{\theta}^{(k)}=\vec{\theta}^{(1)}+\ep \vec{\eta}^{(k)}$ for $k=2,\ldots,m$, where $\vec{\eta}^{(k)}\,(k=2,\ldots,m)$ form the basis of the orthogonal compliment of $(\rho_1,\ldots,\rho_m)$ and $\ep>0$ is a small real number.) 
Then $w^{(k)}(t,x)\,(k=1,\ldots,m)$ are all bounded in $C^{\frac{\alpha}{2}, 1+\alpha}$; hence, in particular, they are bounded in $C^{\frac{\alpha}{2}, \alpha}$ on $(t_0',t_1)\times[-L/2,L/2]$. 

Since each $u_j$ can be expressed as a linear combination of $w^{(k)}$, we see that the functions $u_j(t,x_j)\,(j=1,\ldots, m)$ belong to the class $C^{\frac{\alpha}{2}, \alpha}$ up to the boundary point $x_j=0$. It then follows that, for each choice of the vector $\vec{\theta}^{(k)}$, the corresponding function $g^{(k)}(t,x)$ defined as in \eqref{g+-} belongs to $C^{\frac{\alpha}{2}, \alpha}$ on the region $(t_0',t_1)\times[-L/2,L/2]$. 
%%$C^1$ in $x$ and H$\ddot{\hbox{o}}$lder continuous in $t$ around the point $x=0$. Indeed, the continuity of $\partial_x g^{(k)}$ at $x=0$ follows from the Kirchhoff conditions. 
Thus, by the interior Schauder estimates (see, for instance, \cite[Theorem 3, Chapter 3]{Fr2008}), we see that $w^{(k)}$ belongs to the class $C^{1+\frac{\alpha}{2}, 2+\alpha}$ around $x=0$, hence so do $u_1,\ldots,u_m$. This proves the continuity of $\partial_t u$ and $\partial_x^2 u$ on each edge up to the boundary points. This, together with the continuity of $u$ on $G$ implies the continuity of $\partial_t u$ at each vertex of $G$, hence it is continuous on the entire graph $G$. The continuity of $\Delta_{G}u=\partial_t u-f(u)$ on $G$ then follows. This completes the proof of the lemma. 
\end{proof}

%%%%%%%%%%%%%
\subsection{Comparison principle}\label{ss:comparison}

Like in classical reaction-diffusion equations on Euclidean domains, the comparison principle holds for the equation \eqref{RD-G} on any metric graph $G$. In this section we define the notion of super- and subsolutions for \eqref{RD-G} and state the comparison principle. 

\begin{df}[Super- and subsolutions]\label{def:super-sub}
A continuous function $w(t,x)$ defined on $[t_1,t_2)\times G$ for some time interval $[t_1,t_2)$ is called a {\it supersolution} of $\partial_t u=\Delta_{G}u+f(u)$ if it is piecewise $C^1$ on $G$ and satisfies $\partial_t w \geq \partial_x^2 w + f(w)$ on each edge of $G$ (either in the classical sense or in the weak sense), while it satisfies the following condition at every vertex of $G$:
\begin{equation}\label{super-Kirchhoff}
\sum_{i=1}^m \rho_i\frac{\partial w}{\partial \nu_i}(V) \leq 0 \quad \hbox{(super-Kirchhoff condition)},
\end{equation}
where $\partial u/\partial \nu_i(V)\,(i=1,\ldots,m)$ denotes the derivatives of $u$ along each edge $E_i$ that emanate from the vertex $V$ as shown in Figure~\ref{fig:Kirchhoff}, and $\rho_i$ denotes the thickness of the edges $E_i\,(i=1,\ldots,m)$. Similarly, $w(t,x)$ is called a {\it subsolution} of $\partial_t u=\Delta_{G}u+f(u)$ if it satisfies $\partial_t w \leq \partial_x^2 w + f(w)$ on each edge of $G$ and the following condition at every vertex of $G$:
\begin{equation}\label{sub-Kirchhoff}
\sum_{i=1}^m \rho_i\frac{\partial w}{\partial \nu_i}(V) \geq 0 \quad \hbox{(sub-Kirchhoff condition)},
\end{equation}
\end{df}

As one can see, the condition \eqref{super-Kirchhoff} implies that there is excessive math flux from the vertex $V$, while \eqref{sub-Kirchhoff} implies that there is absorption of mass at $V$. If a supersolution (resp. subsolution) $w$ is independent of $t$, we call it a {\it time-independent supersolution} (resp. {\it subsolution}).

\begin{prop}[Comparison principle]\label{prop:comparison}
Let $w^+(t,x), w^-(t,x)$ be a supersolution and a subsolution of \eqref{RD-G}, respectively, on some time interval $[t_0,t_1)$ and suppose that $w^+(t_0,x)\geq w^-(t_0,x)$ for all $x\in G$. Then $w^+(t,x)\geq w^-(t,x)$ for all $t\in[t_0,t_1),\,x\in G$.
\end{prop}

The above proposition can be shown by a method similar to that for equations on Euclidean domains, so we omit the poof. The Kirchhoff condition and the maximum principle for equations on one-dimensional intervals (including the Hopf boundary lemma) play a key role in the proof. Incidentally, \cite{Be1991} proves the maximum principle for quasilinear equations on metric graphs.

%%%%%%%%%%%%%
\subsection{Gauss--Green theorem and harmonic functions}\label{ss:harmonic}

It is well known that the Gauss-Green theorem, or simply Green's theorem, holds for metric graphs. In this section we recall this theorem and discuss basic properties of harmonic functions on metric graphs. 
The proof of these results is rather easy, so we omit it.

Let $G$ be a metric graph and $D_0$ be its bounded finite subgraph that are connected to the rest of $G$ by edges $E_1,\ldots,E_m$ that do not belong to $D_0$ but stretch from some vertices ${\rm Q}_1,\ldots,{\rm Q}_m$ of $D_0$, as shown in Figure~\ref{fig:Green}. Here the points ${\rm Q}_1,\ldots,{\rm Q}_m$ do not need to be distinct. 

\begin{figure}[h]
\vspace{5pt}
\begin{center}
\includegraphics[scale=0.43]
{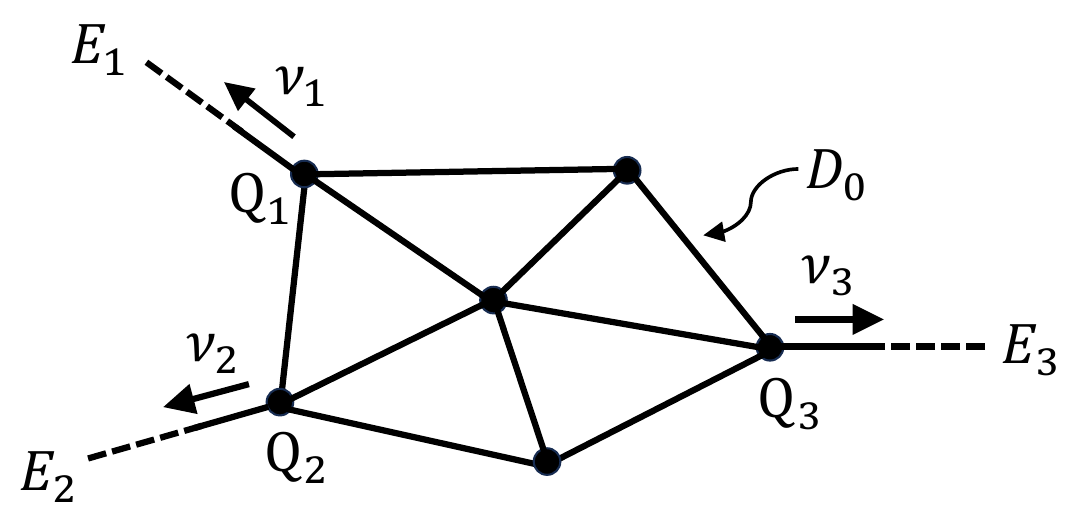}
\end{center}
\vspace{-8pt}
\caption{Graph with exit points ${\rm Q}_1,\ldots,{\rm Q}_m$.}
\label{fig:Green}
\end{figure}

Now let $u$ be a continuous function on $G$ that are $C^2$ except at the vertices and piecewise $C^1$ on $G$ and satisfies the Kirchhoff condition \eqref{Kirchhoff2} at every vertex. Then the following holds:
\begin{equation}\label{Green1}
\int_{D_0}\Delta u \hspace{1pt}dx =\sum_{i=1}^m \rho_i \hspace{1pt}\frac{\partial u}{\partial \nu_i}({\rm Q}_i),
\end{equation}
Next, let $u$ be as above and let $v$ be continuous and piecewise $C^1$. Then 
\begin{equation}\label{Green2a}
\int_{D_0}\left(\nabla v\cdot\nabla u +v\Delta u\right)dx =\sum_{i=1}^m \rho_i v({\rm Q}_i)\frac{\partial u}{\partial \nu_i}({\rm Q}_i).
\end{equation}
Here the integrals on the left-hand side of \eqref{Green1}, \eqref{Green2a} denote the sum of the integrals of $\partial^2_x u$, $\partial_x v \hspace{1pt}\partial_x u$, $v\hspace{1pt}\partial^2_x u$ on each edge of $D_0$ in the sense of \eqref{integral-E}, 
%%when the edge is identified with some interval $[0,L]\subset\R$,  
and each $\partial/\partial \nu_i\,(i=1,\ldots,m)$ denotes the derivative of the function $u$ at ${\rm Q}_i$ along the edge $E_i$ in the direction indicated in Figure~\ref{fig:Green},  
and finally $\rho_i\,(i=1,\ldots,m)$ denote the thickness of the edges $E_1,\ldots,E_m$.

The above formulas follow easily from integration by parts on each edge, namely,
\[
{\rm (a)}\ \ \int_0^L \rho \hspace{1pt}\partial_x^2 u\hspace{1pt}dx = \rho\Big[\partial_x u\Big]^{x=L}_{x=0},
\quad {\rm (b)}\ \ \int_0^L \rho\left(\partial_x v\hspace{1pt}\partial_x u +v \hspace{1pt}\partial_x^2 u\right)dx =\rho\Big[v\hspace{1pt}\partial_x u\Big]^{x=L}_{x=0},
\]
and summing them up over all the edges of $D_0$. More precisely, \eqref{Green1} follows by summing up (a) above, and the boundary terms cancel each other by the Kirchhoff condition \eqref{Kirchhoff2}, except at the vertices ${\rm Q}_i\,(i=1,\ldots,m)$. Similarly, \eqref{Green2a} follows from (b) above and the boundary terms cancell by the continuity of $u$ and the Kirchhoff condition for $v$.
By setting $u=v$ in \eqref{Green2}, we obtain
\begin{equation}\label{Green2}
\int_{D_0}\left(|\nabla u|^2 +u\Delta u\right)dx =\sum_{i=1}^m \rho_iu({\rm Q}_i)\frac{\partial u}{\partial \nu_i}({\rm Q}_i).
\end{equation}

Next we consider harmonic functions on metric graphs. A fuction $u$ on a graph $G$ is called {\it harmonic} if it satisfies Laplace's equation $\Delta_G u=0$ on $G$. This means that $u$ satisfies the equation $\partial_x^2 u=0$ on each edge along with the Kirchhoff condition at every vertex.  Therefore $u$ is a linear function on each edge. Furthermore, \eqref{Green1} and \eqref{Green2} imply the following:
\begin{equation}\label{harmonic-derivatives}
{\rm (a)}\ \ \ \sum_{i=1}^m \rho_i\frac{\partial u}{\partial \nu_i}({\rm Q}_i)=0,\qquad 
{\rm (b)}\ \ \int_{D_0}|\nabla u|^2 dx = \sum_{i=1}^m \rho_iu({\rm Q}_i)\frac{\partial u}{\partial \nu_i}({\rm Q}_i).
\end{equation}
The formula (a) means that the total math inflow and the total math outflow are equal.

\vskip 3pt\noindent
\underbar{\bf Boundary value problems for Laplace's equation}

\vspace{5pt}
We now consider boundary value problems for Laplace's equation on a bounded finite metric graph $D_0$.  We assume that there exist vertices ${\rm Q}_1,\ldots,{\rm Q}_m$ of $D_0$ where the Kirchhoff condition is not required to hold. We call these vertices {\it boundary vertices}. It is as in Figure~\ref{fig:Green}, but the outer edges $E_1,\ldots,E_m$ are just symbolic edges that are placed to make the Kirchhoff condition hold formally at each ${\rm Q}_1,\ldots,{\rm Q}_m$. In other words, $\partial u/\partial \nu_i({\rm Q}_i)\,(i=1,\ldots,m)$ are formal outer derivatives that measure the discrepancy of the Kirchhoff condition at ${\rm Q}_1,\ldots,{\rm Q}_m$.

With this in mind, we consider the following two boundary value problems. 
\begin{equation}\label{harmonic-BVP}
({\rm D})\ \begin{cases}\, \Delta_{D_0} u=0 \ \ \hbox{in}\ \ D_0,\\[2pt] 
\, u({\rm Q}_i) = a_i\ \ (i=1,\ldots,m),
\end{cases}
\quad
({\rm N})\ \begin{cases}\, \Delta_{D_0} u=0 \ \ \hbox{in}\ \ D_0,\\[2pt] 
\, \dfrac{\partial u}{\partial \nu_i}({\rm Q}_i)= b_i\ \ (i=1,\ldots,m).
\end{cases}
\end{equation}

Note that, in the problem (D) above, we assume that ${\rm Q}_1,\ldots,{\rm Q}_m$ are distinct points. On the other hand, in the problem (N), we do not require that ${\rm Q}_1,\ldots,{\rm Q}_m$ be distinct. If ${\rm Q}_{i_1},\ldots,{\rm Q}_{i_k}$ represent the same point, then the discrepancy of the Kirchhoff condition at this point is the sum of $b_{i_1},\ldots,b_{i_k}$. 
These are analogues of the Dirichlet and the Neuman problems for harmonic functions on Euclidean domains. As expected, the following propositions hold:

\begin{prop}[Dirichlet problem on a graph]\label{prop:Dirichilet-Laplace}
For any real numbers $a_1,\ldots,a_m$, the boundary value problem {\rm (N)} in \eqref{harmonic-BVP} has a unique solution.
\end{prop}

\begin{prop}[Neumann problem on a graph]\label{prop:Neumann-Laplace}
The boundary value problem {\rm (N)} in \eqref{harmonic-BVP} has a solution if and only if $\sum_{i=1}^m \rho_ib_i=0$. Furthermore, if $u$ is a solution, any solution of {\rm (N)} is given in the form $u+C$ for some constant $C$.
\end{prop}

As in the case of Dirichlet problems on a Euclidean domain, the existence of the solution of (D) in \eqref{harmonic-BVP} can be shown by minimizing the integral $\int_{D_0}|\nabla u|^2 dx$ under the given boundary conditions. One easily sees that the minimizer automatically satisfies the Kirchhoff condition. The uniqueness follows from \eqref{Green2} (b). 

To prove the existence of a solution for problem (N), it suffices to consider the case $m=2$, for each pair of points ${\rm Q}_i, {\rm Q}_j$, with $b_i=\rho_i^{-1}, b_j=-\rho_j^{-1}$, then to superimpose those solutions with appropriate weights. The solution for the case $m=2$ can be obtained by solving the problem (D) with the condition $u({\rm Q}_i)=1, u({\rm Q}_j)=0$, and by multiplying this solution with an appropriate constant. The fact that $u+C$ exhausts all the solutions follows from \eqref{Green2} (b).

\vskip 8pt\noindent
\underbar{\bf Estimate of the maximal flux of harmonic functions}

\vspace{5pt}
Next we derive an estimate on the maximal flux of harmonice functions. Recall that the mass flux on each edge is given by $-\rho\hspace{1pt}\partial_x u$, where $\rho$ denotes the thickness of the edge. 

\begin{prop}\label{prop:maximal-flux-harmonic}
Let $D$ be the center graph of $\Omega$, and let $u$ be a harmonic function on $D$. Then
\begin{equation}\label{max-flux-harmonic}
\max_{D} \rho\left|\partial_x u\right| \leq \sum_{i=1}^N \max\left(\rho_i\frac{\partial u}{\partial \nu_i}({\rm P}_i),0\right)=\frac12 \sum_{i=1}^N \rho_i\left|\frac{\partial u}{\partial \nu_i}({\rm P}_i)\right|,
\end{equation}
where $\partial/\partial_{\nu_i}$ denotes the derivative along the outer path $\Omega_i$ at ${\rm P}_i$.
\end{prop}

The estimate \eqref{max-flux-harmonic} states that the maximal mass flux on $D$ does not exceed the total sum of incoming mass flux from the outer paths. As $u$ is harmonic, there is no creation or absorption of mass in $D$, therefore the above statement is intuitively clear. We shall prove this estimate by using the Gauss-Green formula \eqref{Green1}.

\begin{proof}
Let $E_0$ denote the edge where the maximum of the flux is attained. (Since $u$ is harmonic, the flux is constant on each edge.) If $\partial_x u=0$ on $E_0$, then there is nothing we need to prove. So we assume that $\partial_x u\ne 0$ and choose the direction of $E_0$ along which $u$ is strictly increasing.

Now we connect edges of $D$ to construct a path that starts from $E_0$ and along which $u$ is strictly increasing.  We consider all such paths, and define a subgraph $D_0$ of $D$ by collecting all the edges and vertices of $D$ that belong to at least one of those paths. Thus any edge and vertex of $D_0$ can be reached from $E_0$ by a path along which $u$ is strictly increasing, and $D_0$ is the maximal subgraph of $D$ with this property.

Let $E_1,\ldots,E_m$ denote the edges of $D$ that emanate from some vertices of $D_0$ but do not belong to $D_0$, and let ${\rm Q}_j\,(j=1,\ldots,m)$ be the vertex of $D_0$ from which the edge $E_j$ emanates. (See Figure~\ref{fig:D0}, in which the edges of $D_0$ are shown in thick lines and the edges $E_1,\ldots,E_m$ are marked in thick dotted lines.)  
%%(Here the vertices ${\rm Q}_1,\ldots,{\rm Q}_m$ need not be distinct points, but the edges $E_1,\ldots,E_m$ are distinct.) 
Then, by the construction of $D_0$, the derivative of $u$ at each ${\rm Q}_j$ along the edge $E_j\,(j=1,\ldots,m)$ in the direction toward the interior of $E_j$ is either $0$ or negative, since otherwise it would contradict the maximality of $D_0$. The outer paths of $\Omega$ emanating from some vertices of $D_0$ are denoted by $\Omega_{i_1},\ldots, \Omega_{i_k}$, with exit points ${\rm P}_{i_1},\ldots,{\rm P}_{i_k}$. 

\begin{figure}[h]
\vspace{8pt}
\begin{center}
\includegraphics[scale=0.46]
{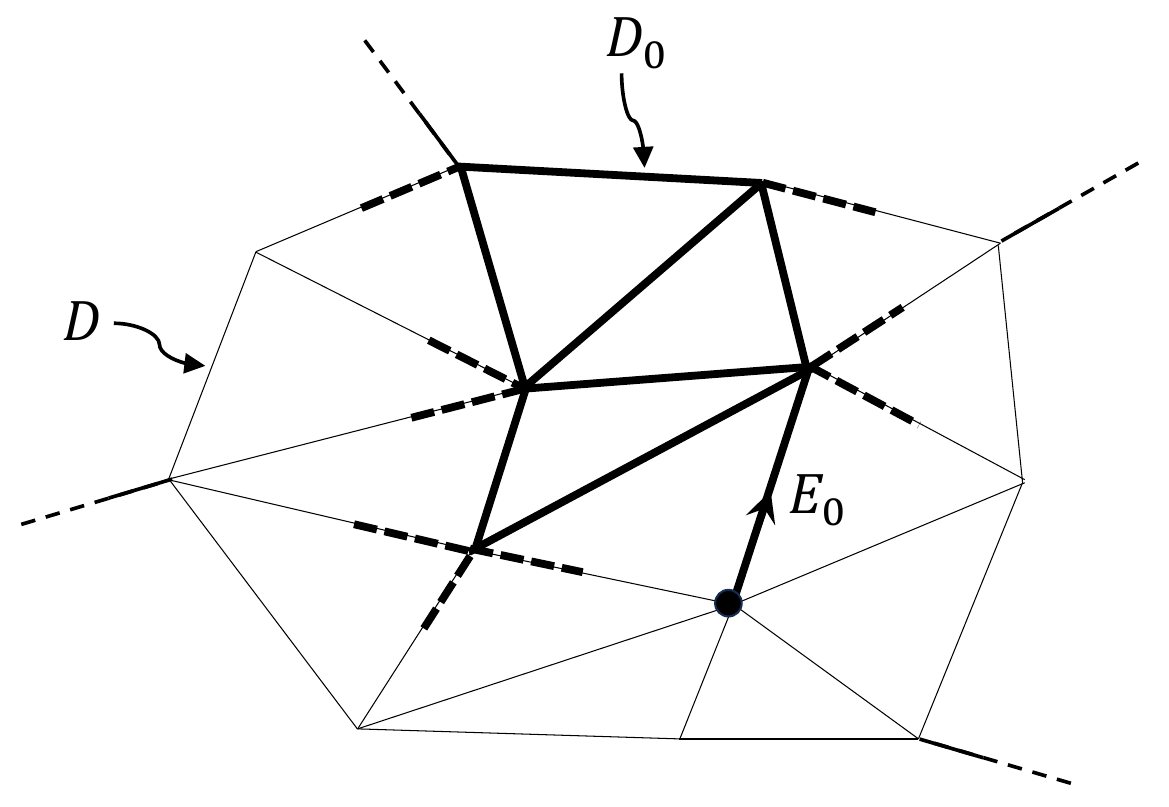}
\end{center}
\vspace{-8pt}
\caption{Image of the subgraph $D_0$ created by the edge $E_0$..}
\label{fig:D0}
\end{figure}

Let ${\rm Q}_0$ be an endpoint of the edge $E_0$. By \eqref{harmonic-derivatives} (a), we have
\[
-\hat\rho_0\hspace{1pt}\partial_x u({\rm Q}_0)+\sum_{j=1}^m \hat\rho_{j}\partial_{\nu} u({\rm Q}_{j})+\sum_{r=1}^k \rho_{i_r}\partial_{\nu} u({\rm P}_{i_r})=0.
\]
Here $\partial_{\nu}$ denotes the derivative along the edges $E_j\,(j=1,\ldots,m)$ or along the outer paths $\Omega_{i_1},\ldots, \Omega_{i_k}$ toward the exterior direction, and $\hat\rho_j\,(j=0,\ldots,m)$ denote the thickness of the edges $E_j$, while $\rho_i$ denotes that of the outer path $\Omega_i$.  
Since $\partial_{\nu} u({\rm Q}_{j})\leq 0$, we obtain
\[
\hat\rho_0\hspace{1pt}\partial_x u({\rm Q}_0) \leq \sum_{r=1}^k \rho_{i_r}\partial_{\nu} u({\rm P}_{i_r}),
\]
which implies \eqref{max-flux-harmonic}. The proof of the proposition is complete.
\end{proof}

The estimate \eqref{max-flux-harmonic} will be used in the proof of Theorem~\ref{thm:perturbation-star2}, which is concerned with pertubation of blocking star graphs.
Note that this estimate is a special case of \eqref{gradient-max1}, which deals with a solution of $\Delta_\Omega u+f(u)=0$.

%%%%%%%%%%%%%
\subsection{Poincar\'e inequality and eigenvalue problems}\label{ss:poincare}

It is well known that Poincar\'e inequality holds on metric graphs. More precisely, 
let $D_0$ be a bounded finite metric graph with boundary vertices ${\rm Q}_1,\ldots,{\rm Q}_m$ as in the previous section. Then there exists a constant $\lambda>0$ such that for any continuous and piecewise $C^1$ function $w$ on $D_0$ satisfying $w({\rm Q}_i)=0\,(i=1,\ldots,m)$, the following inequality holds:
\begin{equation}\label{Poincare}
\int_{D_0}|\nabla w|^2 dx \geq \lambda \int_{D_0} w^2 dx.\quad \hbox{(Poincar\'e inequality)}
\end{equation}
The existence of such a constant $\lambda$ can easily be proved by using the standard Poincar\'e inequality in one space dimension. The optimal such constant $\lambda$ coincides with the first eigenvalue of Laplacian on $D_0$ under the $0$ Dirichlet boundary conditions at ${\rm Q}_1,\ldots,{\rm Q}_m$:
\begin{equation}\label{lambda-eigenvalue}
-\Delta_{D_0}\varphi=\lambda\varphi \ \ \hbox{on}\ \ D_0,\quad 
\varphi({\rm Q}_i)=0\ \ (i=1,\ldots,m).
\end{equation}

Next we consider the case where $D_0$ is a general bounded finite metric graph for which no boundary vertices are specified. Then there exists a constant $\mu>0$ such that
\begin{equation}\label{Poincare2}
\int_{D_0}|\nabla w|^2 dx \geq \mu \int_{D_0} w^2 dx.\quad \hbox{(Poincar\'e--Wirtinger inequality)}
\end{equation}
for any continuous and piecewise $C^1$ function $w$ on $D_0$ satisfying
$\int_{D_0} w \hspace{1pt}dx =0.$

The above inequalities can easily be entended to $H^1$ functions by the limiting procedure. As one can anticipate, the optimal constant $\mu$ in \eqref{Poincare2} coincides with the second eigenvalue (or the smallest positive eigenvalue) of $-\Delta$ on $D_0$, that is,
\begin{equation}\label{mu-eigenvalue}
 -\Delta_{D_0} \varphi=\mu\varphi \ \ \hbox{on}\ \ D_0.
\end{equation} 
Note that no boundary conditions are imposed on the problem \eqref{mu-eigenvalue}, which means that the Kirchhoff condition is satisfied at every vertex of $D_0$. 
In what follows, we shall denote the smallest positive eigenvalue of \eqref{mu-eigenvalue} by $\mu_1(D_0)$.

As one can see, \eqref{lambda-eigenvalue} and \eqref{mu-eigenvalue} are an analogue of Dirichlet and Neumann eigenvalue problems on Euclidean domains. General eigenvalue problems on metric graphs that include \eqref{lambda-eigenvalue} and \eqref{mu-eigenvalue} as special cases are studied in detail by von Below~\cite{Be1988}. Among other things, it is shown that eigenvalues are discrete.

\begin{exmp}[Melon-shaped graph]\label{ex:melon}
Let $D_0$ be a graph consisting of two vertices ${\rm A}, {\rm B}$ and edges $E_1,\ldots,E_m$ of equal length $L_0$ and thickness connecting ${\rm A}, {\rm B}$ as shown in Figure~\ref{fig:melon}. Then
\[
\mu_1(D_0)=\frac{\pi^2}{L_0^2}
\]
and the corresponding eigenfunction $\varphi(x)$ of \eqref{mu-eigenvalue} is given in the form
\[
\varphi{}_{\restr E_i}(x_i)=C \cos\left(\frac{\pi x_i}{L_0}\right)+ C_i\sin\left(\frac{\pi x_i}{L_0}\right)\ \ \hbox{on}\ \ E_i\quad (x_i\in[0,L_0],\,i=1,\ldots,m),
\]
where $\varphi{}_{\restr E_i}$ denotes the restriction of $\varphi$ on each edge $E_i$ and $x_i$ denotes the variable on $E_i$ when it is identified with the interval $[0,L_0]$. The constants $C,C_1,\ldots,C_m$ are not all zero simultaneously, and $C_1+\cdots+C_m=0$. Thus $\mu_1(D_0)$ is an eigenvalue of multiplicity $m$. 

The above formula can be shown easily by recalling the continuity of $\varphi$ and the Kirchhoff condition at ${\rm A, B}$. The details are omitted.
 From the above formula, we see that $\mu_1(D_0)$ does not depend on the number of edges $m$.
\qed
\end{exmp}

\begin{figure}[h]
\begin{center}
\includegraphics[scale=0.45]
{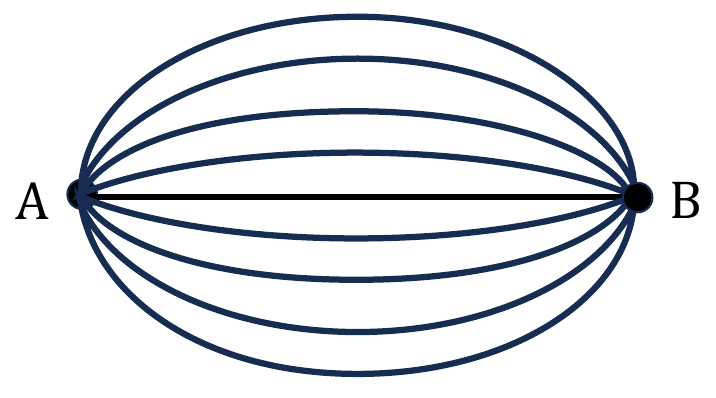}
\end{center}
\vspace{-8pt}
\caption{Melon-shaped graph.}
\label{fig:melon}
\end{figure}

%%%%%%%%%%%%%
\subsection{Local energy}\label{ss:energy}

Let $D_0$ be a bounded finite metric graph with boundary points ${\rm Q}_1,\ldots,{\rm Q}_m$. As we explained before \eqref{harmonic-BVP}, these are the points where the Kirchhoff condition is not required to hold. We consider the equation
\begin{equation}\label{RD-D0}
\partial_t u=\Delta_{D_0} u+f(u)\quad\hbox{on}\ \ D_0
\end{equation}
under the Dirichlet boundary conditions
\begin{equation}\label{Dirichlet-BC}
u(t,{\rm Q}_i)=b_i\quad (i=1,\ldots,m),
\end{equation}
where $b_i\,(i=1,\ldots,m)$ are given constants. 
As an analogue of the energy functional for classical reaction-diffusion equations on Euclidean domains, we define the energy functional for \eqref{RD-D0} by
\begin{equation}\label{energy}
J[u]:=\int_{D_0} \left(\frac{1}{2}|\nabla u|^2 -F(u)\right) dx,
\end{equation}
where $F$ is the primitive of $f$ as defined in \eqref{F}. Just as in the case of the energy functional on Euclidean domains, the following proposition holds:

\begin{prop}\label{prop:energy}
Let $u(t,x)$, with $(t,x)\in [0,T)\times D_0$, be a solution of \eqref{RD-D0} under the boundary conditions \eqref{Dirichlet-BC}. Then the following holds for all $t\in(0,T)$:
\begin{equation}\label{dJdt}
\frac{d}{dt} J[u(t,\cdot)]=-\int_{D_0} \left(\partial_t u\right)^2 dx.
\end{equation}
Hence $J[u(t,\cdot)]$ is strictly monotone decreasing in $t\in(0,T)$ unless $u$ is a stationary solution. 
\end{prop}

\begin{proof}
By the Gauss--Green formula \eqref{Green2a}, we have
\[
\frac{d}{dt} J[u(t,\cdot)]
=\int_{D_0} \big(\nabla \partial_t u\cdot\nabla u-f(u)\partial_t u\big) dx
=\int_{D_0} \big(-\partial_t u\Delta_{D_0} u-f(u)\partial_t u\big) dx=-\int_{D_0} \left(\partial_t u\right)^2 dx.
\]
Here, in applying the formula \eqref{Green2a}, we used the fact that $\partial_t u$ is continuous, as stated in Lemma~\ref{lem:ut-continuous}, and that $u$ satisfies the Kirchhoff condition \eqref{Kirchhoff2}. The proposition is proved.
\end{proof}

If the initial data $u(0,x)$ is $C^1$ on $D_0$, then $u(t,\cdot)\to u(0,\cdot)$ in $C^1(D_0)$ as $t\to 0$, hence $J[u(t,\cdot)]\to J[u(0,\cdot)]$ as $t\to 0$. Thus we have the following corollary:

\begin{cor}\label{cor:energy}
Let $u$ be as in Proposition~\ref{prop:energy} and assume that $u(0,x)$ is $C^1$ on $D_0$. Then 
$J[u(0,\cdot)]\geq J[u(t,\cdot)]$ for $t\in[0,T)$. 
\end{cor}

\begin{rem}\label{rem:energy}
If $w$ is continuous and piecewise $C^1$ on $D_0$, there is a sequence of $C^1$ functions $w_n\,(n=1,2,3,\ldots)$ such that $w_n\to w$ uniformly on $D_0$ and that $J[w_n]\to J[w]$ as $n\to\infty$. Therefore, the conclusion of Corollary~\ref{cor:energy} holds if $u(0,x)$ is continuous and piecewise $C^1$ on $D_0$.
\end{rem}

%%%%%%%%%%%%%
\subsection{Gradient estimates}\label{ss:maximal}

In this section we derive estimates of the gradient $\partial_x v$ of solutions of \eqref{stationary2} that depend only on the nonlinearity $f$, the number of outer paths $N$, and the thickness and the total length of the edges of the center graph $D$. These estimates will play an important role in the perturbation analysis in Section~\ref{ss:proof-perturbation}. We first establish an $H^1$ estimate, then we derive an $L^\infty$ estimate.

Let $v_0$ be a solution of the stationary problem \eqref{stationary2} satisfying $0\leq v_0\leq 1$. Here $\Omega$ is a graph defined in \eqref{Omega}. Then by the Gauss-Green formula \eqref{Green2}, we get
\[
\int_{D} |\nabla v_0|^2 dx = \int_{D} v_0 f(v_0) dx + \sum_{i=1}^N \rho_i v_0({\rm P}_i)\frac{\partial v_0}{\partial \nu_i}({\rm P}_i),
\]
where $\rho_i$ denote the thickness of the outer paths $\Omega_i\,(i=1,\ldots,N)$. 
Since $v_0$ lies between $0$ and $1$ on the outer edges $\Omega_1,\ldots,\Omega_N$ all the way to infinity, we see from the phase-portrait in Figure~\ref{fig:phase-portrait} that the value of $(v_0,\partial_x v_0)$ along each $\Omega_i$ forms a portion of the orbit that lies between the stable and the unstable manifolds of $(1,0)$. These manifolds are characterized by
\[
\frac12 \left(\partial_x v\right)^2 + F(v)=F(1),
\]
or
\[
\left|\partial_x v\right|=\sqrt{2(F(1)-F(v))} \leq \sqrt{2(F(1)-F(a))}.
\]
Consequently, we have
\[
\left|\frac{\partial v_0}{\partial \nu_i}({\rm P}_i)\right|\leq \sqrt{2\left(F(1)-F(a)\right)}\quad (i=1,\ldots,N).
\]
Thus we obtain 
\begin{equation}\label{gradient-estimate}
\int_{D} |\nabla v_0|^2 dx \leq f_{\max}|D| + N\rho_{\max}\sqrt{2\left(F(1)-F(a)\right)},
\end{equation}
where $|D|$ denotes the total weighted sum of the length of the edges of $D$, namely
\[
|D|:=\int_D dx = \sum_{E\in{\mathcal E}_D} \mu(E)L(E),
\]
where ${\mathcal E}_D$ denotes the set of all the edges of $D$, while $\mu(E)$ and $L(E)$ denote, respectively, the thickness and the length of the edge $E$  
%% (that is, the length times the thickness of each edge), 
and $f_{\max}:=\max_{0\leq s\leq 1}f(s)$, $\rho_{\max}:=\max(\rho_1,\ldots,\rho_N)$.

Next we derive an $L^\infty$ estimate of $\partial_x v_0$ on the center graph $D$. This is a generalization of the estimate \eqref{max-flux-harmonic} for harmonic functions. 
As $D$ is compact, and since $v_0$ is piecewise $C^1$ by the definition of solutions of \eqref{stationary2}, $|\partial_x v_0|$ attains a maximum on $D_0$. 
Let ${\rm Q}_0$ be the point where the maximum is attained. If ${\rm Q}_0$ is not a vertex of $D$, we add this point as a new vertex of $D$.

Now we construct a subgraph $D_0$ of $D$ following the procedures similar to the proof of Proposition~\ref{prop:maximal-flux-harmonic}. More precisely: 
\begin{itemize}
\item We consider all the paths in $D$ starting from ${\rm Q}_0$ along which $v_0$ is strictly increasing.
\end{itemize}
We extend each path as long as possible until it is no longer possible to extend it. There are two cases where the path terminates. The first case is when $\partial_x v_0$ becomes $0$ at some point. If that point is not a vertex of $D$, we add this point as a new vertex of $D$. The other case is when the path comes to one of the exit points ${\rm P}_1,\ldots, {\rm P}_N$ and there is no way to extend the path within $D$ while keeping $v_0$ strictly increasing. As one can see from the Kirchhoff conditions, there is no other situation where the path terminates. We consider all such paths, and define a subgraph $D_0$ of $D$ by collecting all the edges and vertices of $D$ that belong to at least one of those paths, as we have done in the proof of Proposition~\ref{prop:maximal-flux-harmonic}. Thus any edge and vertex of $D_0$ can be reached from ${\rm Q}_0$ by a path along which $v_0$ is strictly increasing, and $D_0$ is the maximal subgraph of $D$ with this property.

Let $E_1,\ldots,E_m$ denote the edges of $D$ that emanate from some vertices of $D_0$ but do not belong to $D_0$, and let ${\rm Q}_j$ be the vertex of $D_0$ from which the edge $E_j$ emanates. (The points ${\rm Q}_j\,(j=1,\ldots,m)$ may not necessarily be distinct.) 
Then, by the maximality of $D_0$, the derivative of $v_0$ at each ${\rm Q}_j$ in the direction toward the interior of $E_j$ is either $0$ or negative. Next let ${\rm P}_{i_1},\ldots,{\rm P}_{i_k}$ denote the exit points of $\Omega$ that belong to $D_0$ and let $\Omega_{i_1},\ldots, \Omega_{i_k}$ be the corresponding outer paths.

Now we apply the Gauss-Green formula \eqref{Green1} to $v_0$ on $D_0$, to obtain
\[
-\hat\rho_0\hspace{1pt}\partial_x v_0({\rm Q}_0)+\sum_{j=1}^m \hat\rho_{j}\partial_{\nu} v_0({\rm Q}_{j})+\sum_{r=1}^k \rho_{i_r}\partial_{\nu} v_0({\rm P}_{i_r})
=\int_{D_0}\Delta_{D_0}v_0(x)dx=-\int_{D_0} f(v_0(x)) dx.
\]
Here $\partial_{\nu}$ denotes the derivative along the edges $E_j\,(j=1,\ldots,m)$ or along the outer paths $\Omega_{i_1},\ldots, \Omega_{i_k}$ toward the exterior direction, and $\hat\rho_j\,(j=0,\ldots,m)$ denote the thickness of the edges $E_j$, while $\rho_i$ denotes that of the outer path $\Omega_i$.  
Since $\partial_{\nu} v_0({\rm Q}_{j})\leq 0$, we obtain
\[
\hat\rho_0\hspace{1pt}\partial_x v_0({\rm Q}_0) \leq \sum_{r=1}^k \rho_{i_r}\partial_{\nu} v_0({\rm P}_{i_r})+\int_{D_0} f(v_0(x)) dx,
\]
hence
\begin{equation}\label{gradient-max1}
\max_{D} \rho\left|\partial_x v_0\right| \leq \sum_{i=1}^N \max\left(\rho_i\frac{\partial v_0}{\partial \nu_i}({\rm P}_i),0\right)+f_{\max}|D|.
\end{equation}
As we have seen above, $\partial_{\nu}v_0(P_i)\leq \sqrt{2(F(1)-F(a))}$. Hence the following proposition holds:

\begin{prop}\label{prop:gradient-max}
Let $v_0$ be a solution of \eqref{stationary2}. Then
\begin{equation}\label{gradient-max}
\max_{D} \rho\left|\partial_x v_0\right| \leq f_{\max}|D|+N\rho_{\max}\sqrt{2(F(1)-F(a))}.
\end{equation}
where $f_{\max}$ denotes the maximum of $f(s)$ on $0\leq s\leq 1$ and $\rho_{\max}:=\max(\rho_1,\ldots,\rho_N)$.
\end{prop}

%%%%%%%%%%%%%%%%%%%%%%%%%%%
%%%%%%%%%%%%%%%%%%%%%%%%%%%
\section{Proof of the main results}\label{s:proof-main}

%%%%%%%%%%%%%
\subsection{Existence and uniqueness of front-like solutions}\label{ss:proof-front}

In this section we prove Theorem~\ref{thm:u-hat} on the existence and uniqueness of the front-like solution $\widehat{u}_i$ satisfying \eqref{u-hat}.  The existence is already proved in \cite{JM2019} for star graphs, and basically the same proof applies to our case without much modifying the arguments. The idea is to construct an ordered pair of super- and subsolutions both of which converge to the traveling wave solution $\phi(-x_i-ct)$ as $t\to-\infty$. So we omit the existence proof and focus on the uniqueness and time monotonicity. Note that a very similar result on the existence of a front-like entire solution can also be found in \cite{BHM2009, BHM2025}, which deal with front propagation in $\R^N$ in the presence of obstacles.

We shall prove the proposition below, from which the rest of Theorem~\ref{thm:u-hat} follows:

\begin{prop}\label{prop:uniqueness}
Let $\widehat{u}_i(t,x)$ be an entire solution of \eqref{RD-Omega} satisfying \eqref{u-hat}. Then $\widehat{u}_i$ is strictly monotone increasing in $t$. Furthermore, there exists only one entire solution that satisfies \eqref{u-hat}.
\end{prop}

A result very similar to the above proposition can be found in the recent paper by one of the present authors \cite[Proposition 5.1]{BHM2025}. Our arguments below are adopted from \cite{BHM2025} with some modifications.
Let $\delta_0\in(0,\frac12)$ and $\sigma_0>0$ be constants such that 
\begin{equation}\label{delta0}
f'(s)\leq -\sigma_0 \quad \hbox{for}\ \ s\in [0,\delta_0]\cup[1-\delta_0,1].
\end{equation}
Such constants $\delta_0, \sigma_0$ exist since $f$ satisfies \eqref{f}. 
The following two lemmas will play a key role in the proof of the above proposition, but these lemmas will also play an important role in later sections, such as in the proof of the minimality of the limit profile $\widehat{v}_i$ in Section~\ref{ss:minimality}.

\begin{lem}[Comparison of ancient solutions]\label{lem:comparison-ancient}
Let $u(t,x)$ be a subsolution and $\widetilde{u}(t,x)$ be a supersolution of the equation $\partial_t u=\partial_x^2 u + f(u)$ on $(-\infty, T]\times[0,+\infty)$ for some $T\in\R$ satisfying $0\leq u\leq 1,\, 0\leq \widetilde{u}\leq 1$. 
 \begin{itemize}
\item[{\rm (i)}] 
Assume that there exists a smooth function $a_1(t)\geq 0$ such that
\[
u(t, a_1(t))<\widetilde{u}(t, a_1(t))\quad ({}\forall t\in (-\infty,T]),
\]
\[
1-\delta_0\leq \widetilde{u}(t,x)\leq 1 \ \ \left({}\forall t\in (-\infty,T], \, x\geq a_1(t)\right).
\]
Then
\[
u(t,x)<\widetilde{u}(t,x) \quad \hbox{for all}\ \ t\in (-\infty,T] \ \, \hbox{and}\ \, x\geq a_1(t).
\]
\item[{\rm (ii)}] 
Assume that there exists a smooth function $a_2(t)\geq 0$ such that
\[
u(t, a_2(t))<\widetilde{u}(t, a_2(t))\quad ({}\forall t\in (-\infty,T]),
\]
\[
0\leq u(t,x)\leq \delta_0 \ \ \left({}\forall t\in (-\infty,T], \, x\geq a_2(t)\right).
\]
Then
\[
u(t,x)<\widetilde{u}(t,x) \quad \hbox{for all}\ \ t\in (-\infty,T] \ \, \hbox{and}\ \, x\geq a_2(t).
\]
\end{itemize}
\end{lem}

\begin{proof}
We first prove (i). Let $w(t,x):=u(t,x)-\widetilde{u}(t,x)$.  Then $w$ is a subsolution of
\[
\partial_t w=\partial_{x_i}^2 w+h(t,x)w \quad (t\in (-\infty,T],\, x > a_1(t)),
\]
where
\[
h(t,x)=\int_0^1 f'\left(su(t,x)+(1-s) \widetilde{u}(t,x)\right) ds,
\]
satisfying the boundary condition
\begin{equation}\label{w(alpha)}
w(t,a_1(t))<0\quad (t\in(-\infty,T]).
\end{equation}
Also, since $1-\delta_0\leq \widetilde u\leq 1$ and $u\leq 1$, we have $w\leq \delta_0$. 
It suffices to show that $w<0$.

Suppose that $w(t,x)\geq 0$ for some $t\in(-\infty,T]$, $x > a_1(t)$. Then we have $1-\delta_0\leq \widetilde{u}(t,x)\leq u(t,x)\leq 1$, which implies $h(t,x)\leq -\sigma_0$.  It follows that
\begin{equation}\label{w-sub}
w_t\leq\Delta w -\sigma_0 w \quad \hbox{wherever}\ \ w\geq 0.
\end{equation}
Now we choose $T_1<T$ arbitrarily and define a function $\eta(t,x)=\delta_0 e^{-\sigma_0(t-T_1)}$. Then
\[
\eta_t=-\sigma_0 \eta=\Delta\eta -\sigma_0 \eta,\quad (t\in [T_1,T],\, x>a_1(t)).
\]
Note also that
\[
\begin{split}
& \eta(T_1,x)=\delta_0 \geq w(T_1,x)\quad  (x\geq a_1(T_1)),\\
& \eta(t, a_1(t))>0>w(t, a_1(t))\ \ (t\in[T_1,T]).
\end{split}
\]
Thus, in view of \eqref{w-sub}, $\eta$ acts as an upper barrier for $w$. Consequently,
\[
w(t,x)\leq \eta(t,x)=\delta_0 e^{-\sigma_0 (t-T_1)}\quad (t\in [T_1,T],\, x\geq a_1(t)).
\]
Recall that $T_1\in (-\infty,T)$ is arbitrary. Letting $T_1\to -\infty$, we obtain
\[
w(t,x)\leq 0 \quad \hbox{for all}\ \ t\in (-\infty,T],\, x\geq  a_1(t).
\]
Since $w$ is not identically $0$, the strong maximum principle implies $w<0$. This proves statement~(i). 
Statement (ii) can be proved similarly. Define $w(t,x):=u(t,x)-\widetilde{u}(t,x)$ as above. Then $w(t,a_2(t))<0$ for $t\in(-\infty,T]$, and, since $0\leq u\leq\delta_0$ and $\widetilde{u}\geq 0$, we have $w\leq \delta_0$. The rest of the proof is precisely the same as above. This completes the proof of Lemma~\ref{lem:comparison-ancient}.
.\end{proof}

The next lemma is similar to the above lemma, but it is concerned with a solution on $\Omega$. 
Before stating the lemma, we introduce the following notation for each $i\in\{1,\ldots,N\}$:
\[
\Omega_i[b]:=\{x\in\Omega_i \mid 0\leq x_i \leq b\}. 
\]

\begin{lem}[Comparison of ancient solutions on $\Omega$]\label{lem:comparison-ancient2}
Let $u(t,x)$ be a subsolution and $\widetilde{u}(t,x)$ be a supersolution of \eqref{RD-Omega} that are defined for $t\in(-\infty, T]$ for some $T\in\R$ and satisfy $0\leq u\leq 1,\, 0\leq \widetilde{u}\leq 1$. Fix $i\in\{1,\ldots,N\}$ and let $u{}_{\restr \Omega_i}(t,x_i), \widetilde{u}{}_{\restr \Omega_i}(t,x_i)$ denote, respectively, the restriction of $u,\widetilde{u}$ onto the outer path $\Omega_i$.
Assume that there exists a smooth function $b(t)\geq 0$ such that
\[
u_{\restr \Omega_i}(t, b(t))<\widetilde{u}_{\restr \Omega_i}(t, b(t))\quad ({}\forall t\in (-\infty,T]),
\]
\[
0\leq u(t,x)\leq \delta_0 \ \ \left({}\forall t\in (-\infty,T],\,  x\in \Omega_i[b(t)]\cup\left(\Omega\setminus\Omega_i\right)\right).
\]
Then
\[
u(t,x)<\widetilde{u}(t,x) \quad \hbox{for all}\ \ t\in (-\infty,T], \, x\in \Omega_i[b(t)]\cup\left(\Omega\setminus\Omega_i\right)
\]
\end{lem}

\begin{proof}
The proof is basically the same as that of Lemma~\ref{lem:comparison-ancient} (ii). We define  $w(t,x):=u(t,x)-\widetilde{u}(t,x)$ as before. Then $w$ is a subsolution of the equation
\[
\partial_t w=\Delta w+h(t,x)w \quad (\,t\in (-\infty,T],\, x \in \Omega_i[b(t)]\cup\left(\Omega\setminus\Omega_i\right)\,).
\]
What we have to show is $w<0$. By the assumption, we have $w\leq \delta_0$. As before, we choose $T_1<T$ arbitrarily and define a function $\eta(t,x)=\delta_0 e^{-\sigma_0(t-T_1)}$. Since $\eta$ is spatially uniform, it satisfies the Kirchhoff condition on $\Omega_i[b(t)]\cup\left(\Omega\setminus\Omega_i\right)$. Therefore, $\eta$ acts as an upper barrier for $w$ on this part of the graph $\Omega$. The rest of the proof is the same as that of Lemma~\ref{lem:comparison-ancient}, so we omit it. The proof of the lemma is complete.
\end{proof}

Now we are ready to prove the time monotonicity and uniqueness of $\widehat{u}_i$. 

\begin{proof}[Proof of Proposition~\ref{prop:uniqueness}] 
We begin with the proof of the time-monotonicity. Let $\delta_0, \sigma_0$ be as in \eqref{delta0}. We recall that $\phi$ satisfies the condition \eqref{phi(0)}, that is, $\phi(0)=a$. Let $L>0$ be such that
\[
1-\frac{\delta_0}{2}\leq \phi(z)<1 \ \ \hbox{for}\ \ z\in(-\infty, -L],\quad 
0<\phi(z)\leq \frac{\delta_0}{2}\ \ \hbox{for}\ \ z\in [L,+\infty)
\]
and define $a_1(t)=-ct+L$, $b(t)=-ct-L$. Then
\[
1-\frac{\delta_0}{2}\leq\phi(-x-ct)< 1\ \ \hbox{if}\ \ x\geq a_1(t),\quad 
 0<\phi(-x-ct)\leq \frac{\delta_0}{2}\ \ \hbox{if}\ \ x\leq b(t).
\]
Now let $\tau>0$ be a constant. Then, since $\phi$ is monotone decreasing, we have
\[
  \phi(-x-c(t+\tau))-\phi(-x-ct)>0.
\]
As $\widehat{u}_i{}_{\restr \Omega_i}(t,x_i)$ converges to $\phi(-x_i-ct)$ and $\widehat{u}_i{}_{\restr\Omega_i}(t+\tau,x_i)$ to $\phi(-x_i-c(t+\tau))$ as $t\to-\infty$ uniformly on $\Omega_i $, we see that, for any sufficiently small $\tau>0$, there exists $T<0$ such that
\begin{equation}\label{ubar-ubar-a}
  1-\delta_0 \leq \widehat{u}_i{}_{\restr\Omega_i}(t,x_i),\,\widehat{u}_i{}_{\restr\Omega_i}(t+\tau,x_i)<1\quad \hbox{for}\ \ t\in(-\infty,T],\, x_i\geq a_1(t),
\end{equation}
\begin{equation}\label{ubar-ubar-b}
  0< \widehat{u}_i(t,x),\,\widehat{u}_i(t+\tau,x)\leq\delta_0 \quad \hbox{for}\ \ t\in(-\infty,T],\, x\in \Omega_i[b(t)]\cup\left(\Omega\setminus\Omega_i\right),
\end{equation}
\begin{equation}\label{ubar-ubar-ab}
  \widehat{u}_i{}_{\restr\Omega_i}(t+\tau,x_i)-\widehat{u}_i{}_{\restr\Omega_i}(t,x_i)>0 \quad \hbox{for}\ \ t\in(-\infty,T],\, a_1(t)\geq x_i\geq b(t).
\end{equation}
Combining \eqref{ubar-ubar-a} and \eqref{ubar-ubar-ab}, and applying Lemma~\ref{lem:comparison-ancient} (i) to $\widetilde{u}(t,x_i):=\widehat{u}_i{}_{\restr\Omega_i}(t+\tau,x_i)$, $u(t,x_i):=\widehat{u}_i{}_{\restr\Omega_i}(t,x_i)$, we see that
\[
\widehat{u}_i{}_{\restr\Omega_i}(t+\tau,x_i)>\widehat{u}_i{}_{\restr\Omega_i}(t,x_i) \quad \hbox{for}\ \ t\in(-\infty,T],\, x_i\in [b(t),\infty).
\]
Similarly, combining \eqref{ubar-ubar-b} and \eqref{ubar-ubar-ab} (with $x_i=b(t)$, and applying Lemma~\ref{lem:comparison-ancient2}, we get
\[
\widehat{u}_i(t+\tau,x)>\widehat{u}_i(t,x)\hbox{for}\ \ t\in(-\infty,T],\, x\in \Omega_i[b(t)]\cup\left(\Omega\setminus\Omega_i\right).
\]
Combining the above two inequalities, we obtain $\widehat{u}_i(t+\tau,x)>\widehat{u}_i(t,x)$ on $\Omega$ for all $t\in(-\infty,T]$. The same inequality holds for all $t>T$ by the comparison principle. Hence $\widehat{u}_i(t+\tau,x)>\widehat{u}_i(t,x)$ for all $t\in\R$ and $x\in \Omega$. Since $\tau>0$ is arbitrary, this prove the monotonicity of $\widehat{u}_i$ in $t$.

Next we prove the uniqueness. Suppose that $\widehat{U}_i$ satisfies the same condition as \eqref{u-hat}. 
Then, by setting $\widetilde{u}(t,x)=\widehat{U}_i(t+\tau,x)$, $u(t,x)=\widehat{u}_i(t,x)$ and repeating the same argument as above, we get $\widehat{U}_i(t+\tau,x)>\widehat{u}_i(t,x)$ on $\Omega$. Letting $\tau\to 0$, we obtain $\widehat{U}_i(t,x)\geq \widehat{u}_i(t,x)$. By reversing the role of $\widehat{U}_i$ and $\widehat{u}_i$ and repeating the same argument, we see that $\widehat{u}_i(t,x)\leq \widehat{U}_i(t,x)$ on $\Omega$. Consequently, $\widehat{U}_i=\widehat{u}_i$. This proves uniquness. 
The proof of Proposition~\ref{prop:uniqueness} is complete.
\end{proof}

%%%%%%%%%%%%%
\subsection{Proof of dichotomy and minimality}\label{ss:proof-dichotomy}

In this section we prove Theorem~\ref{thm:dichotomy} (Dichotomy theorem) and Theorem~\ref{thm:minimal} (Minimality theorem). We start with the following lemma:

\begin{lem}\label{lem:below-V}
Let $i,j\in\{1,\ldots,N\}$ with $i\ne j$. Suppose that $\widehat{v}_i({\rm P}_j)\leq \beta$. Then
\begin{equation}\label{vi<V}
\widehat{v}_i{}_{\restr\Omega_j}(x_j)\leq V(x_j)\quad\hbox{for all}\ \ x_j\geq 0,
\end{equation}
where $V$ denotes the pulse solution shown in Figure~\ref{fig:pulse-solution}.
\end{lem}

\begin{proof}
Let $\delta_0>0$ be as defined in \eqref{delta0}. Since $\widehat{u}_i{}_{\restr\Omega_j}(t,x_j)$ converges to $0$ as $t\to-\infty$ uniformly in $x_j\geq 0$, there exists $T\in\R$ such that
\[
0<\widehat{u}_i{}_{\restr\Omega_j}(t,x_j)\leq \delta_0 \quad \hbox{for all} \ \ t\in(-\infty,T],\,x_j\in[0,\infty).
\] 
Consequently, by Lemma~\ref{lem:comparison-ancient} (ii), we get
\[
\widehat{u}_i{}_{\restr\Omega_j}(t,x_j)\leq V(x_j)\quad \hbox{for all} \ \ t\in(-\infty,T],\,x_j\in[0,\infty).
\]
By the assumption $\widehat{v}_i({\rm P}_j)\leq \beta$ and the time monotonicity of $\widehat{u}_i(t,x_j)$, we have
\[
\widehat{u}_i{}_{\restr\Omega_j}(t,0)\leq V(0)\quad \hbox{for all}\ \ t\in\R.
\]
Thus, by the comparison principle, we see that $\widehat{u}_i{}_{\restr\Omega_j}(t,x_j)\leq V(x_j)$ for all $t\in\R, \,x_j\in[0,\infty)$. Letting $t\to\infty$, we obtain \eqref{vi<V}. The lemma is proved.
\end{proof}

\begin{proof}[Proof of Theorem~\ref{thm:dichotomy}]
If $\widehat{v}_i({\rm P}_j)\leq \beta$, then, by \eqref{vi<V}, we have $\widehat{v}_i{}_{\restr\Omega_j}(x_j)\to 0$ as $x_j\to \infty$. Next suppose that $\widehat{v}_i({\rm P}_j)> \beta$. In this case, for $\widehat{v}_i(x_j)$ to stay betweeen $0$ and $1$ for all $x_j\geq 0$, this solution must lie on the stable manifold of $(1,0)$ as we see from the phase portrait in Figure~\ref{fig:phase-portrait}. Consequently $\widehat{v}_i(x_j)\to 1$ as $x_j\to \infty$. This proves the dichotomy \eqref{dichotomy}.  Furthermore, if $\widehat{v}_i({\rm P}_j)\leq \beta$, the solution  $\widehat{v}_i{}_{\restr\Omega_j}$ must lie on the homoclinic orbit since otherwise it cannot stay positive. Therefore $\widehat{v}_i{}_{\restr\Omega_j}(x_j)=V(x_j+b)$ for some $b\geq 0$. Finally, the monotonicity of $\widehat{v}_i$ on each outer path $\Omega_j$ is clear from the above arguments. The theorem is proved.
\end{proof}

\begin{proof}[Proof of Theorem~\ref{thm:minimal}]
By the assumption on $v$, there exists $a_1\geq 0$ such that $1-\delta_0\leq v_{\restr\Omega_i}(x_i)\leq 1$ for all $x_i\geq a_1$. Choose $T_i\in\R$ such that $\widehat{u}_i{}_{\restr\Omega_i}(t,a_1)<v{}_{\restr\Omega_i}(a_1)$ for all $t\in(-\infty,T_i]$. Then, by Lemma~\ref{lem:comparison-ancient} (i), we have
\begin{equation}\label{ui<v}
\widehat{u}_i{}_{\restr\Omega_i}(t,x_i)<v{}_{\Omega_i}(x_i)\quad \hbox{for all} \ \ t\in(-\infty,T_i],\ x_i\in[a_1,\infty).
\end{equation}

Next we compare $v$ and $\widehat{u}_i$ on each outer path $\Omega_j\,(j\ne i)$. As we see from the phase portrait in Figure~\ref{fig:phase-portrait}, the solution $v{}_{\restr\Omega_j}(x_j)$ either tends to $0$ as $x_j\to\infty$ or stays away from $0$. If $v{}_{\restr\Omega_j}(x_j)$ tends to $0$ as $x_j\to\infty$, then by the same argument as in the proof of Lemma~\ref{lem:below-V}, we have
\begin{equation}\label{ui<v2}
\widehat{u}_i{}_{\restr\Omega_j}(t,x_j)<v{}_{\restr\Omega_j}(x_j)\quad \hbox{for all} \ \ t\in(-\infty,T_j],\ x_j\in[0,\infty).
\end{equation}
for some $T_j\in\R$. If, on the other hand, $v{}_{\restr\Omega_j}(x_j)$ stay away from $0$ for all $x_j\geq 0$, then there certainly exists some $T_j\in\R$ such that \eqref{ui<v2} holds. Thus, in either case, \eqref{ui<v2} holds for for all $j\ne i$ and for some $T_j\in\R$. Choose $T\in\R$ such that $T\leq \min(T_1,\ldots,T_N)$ and that 
\[
\widehat{u}_i(t,x)<v(x)\quad\hbox{for all}\ \ t\in (-\infty,T],\ x\in D\cup \Omega_i[a_1],
\]
where $D$ is the center graph and $\Omega_i[a_1]:=\{x\in\Omega_i\mid 0\leq x_i\leq a_1\}$. Such $T$ cetainly exists since $\widehat{u}(t,x)$ converges to $0$ as $t\to-\infty$ uniformly on $D\cup\Omega[a_i]$. Combining the above inequality and \eqref{ui<v}, \eqref{ui<v2}, we obtain $\widehat{u}_i(t,x)<v(x)$ for all $t\in(-\infty,T],\,x\in\Omega$. Thus, by the comparison principle, we have $\widehat{u}_i(t,x)<v(x)$ for all $t\in\R,\,x\in\Omega$. Letting $t\to\infty$, we get $\widehat{v}_i\leq v$. The theorem is proved.
\end{proof}

%%%%%%%%%%%%%
\subsection{Proof of stability and transient properties}\label{ss:proof-stability}

\begin{proof}[Proof of Theorem~\ref{thm:stability}]
Assume that $\lambda^R<0$ for some $R>0$, and let $\varphi^R>0$ be the corresponding eigenfunction. Let $\Phi^R$ denote the function on $\Omega$ obtained by extending $\varphi^R$ by $0$ outside $\Omega^R$. Then, by the standard linearization arguments, $\widehat{v}_i-\ep\Phi^R$ is a time-independent supersolution of \eqref{RD-Omega} for any sufficiently small $\ep>0$. Choose $\ep$ small enough so that $\widehat{v}_i-\ep\Phi^R>0$ on $\Omega$. Since $\widehat{u}_i(t,x)<\widehat{v}_i(x)$ and since $\widehat{u}_i(t,x)\to 0$ as $t\to-\infty$ uniformly on $\Omega^R$, there exists $T\in\R$ such that $\widehat{u}_i(T,x)<\widehat{v}_i-\ep\Phi^R$ on $\Omega$. Then, by the comparison principle, $\widehat{u}_i(t,x)<\widehat{v}_i-\ep\Phi^R$ for all $t\geq T$, which contradicts the fact that $\widehat{u}_i(t,x)\to \widehat{v}_i(x)$ as $t\to +\infty$. This contradiction shows that $\lambda^R\geq 0$ for all $R>0$. Since $\lambda^R$ is strictly decreasing in $R>0$, we have $\lambda^R>0$ for all $R>0$. The theorem is proved.
\end{proof}

\begin{proof}[Proof of Corollary~\ref{cor:no-two-stationary}]
Let us first show that, on any edge $E$ of $\Omega$, the set of points where $\partial_x\widehat{v}_i$ vanishes is a discrete set (or empty).  To see this, suppose that this set has an accumulation point ${\rm Q}$ on $E$. Then $\partial_x\widehat{v}_i$ and $\partial_x^2\widehat{v}_i$ both vanish at ${\rm Q}$. Since $\widehat{v}_i$ is a solution of the equation $\partial_x^2 v+f(v)=0$, $\widehat{v}_i$ must be a constant. Furthermore, since $\widehat{v}_i>0$, we have either $\widehat{v}_i\equiv a$ or $\widehat{v}_i\equiv 1$, but neither of these is possible by the assumption. This proves the above claim.

Now suppose that there exists an edge $E$ of $\Omega$ such that $\partial_x\widehat{v}_i$ vanishes at more than one point. Then we can choose distinct points ${\rm A}, {\rm B}$ on $E$ such that $\partial_x\widehat{v}_i$ vanishes at these points but does not vanish on the portion of $E$ between ${\rm A}, {\rm B}$. We identify $E$ with an interval $[0,L]\in\R$, and let $0\leq x_{\rm A}, x_{\rm B}\leq L$ correspond to the points ${\rm A}, {\rm B}$. Then the function $\psi(x):=\partial_x \widehat{v}_i(x)$ satisfies
\begin{equation}\label{eigen-AB}
\partial_x^2 \psi + f'(\widehat{v}_i)\psi = 0,\ \ \psi\ne 0 \ \ (x_{\rm A}<x<x_{\rm B}),\quad \psi(x_{\rm A})=\psi(x_{\rm B})=0.
\end{equation}
Let $\lambda_{{\rm A}{\rm B}}$ denote the principal eigenvalue of the problem \eqref{EP-OmegaR} with $\Omega^R$ replaced by the line segment ${\rm A}{\rm B}$. Then \eqref{eigen-AB} implies $\lambda_{{\rm A}{\rm B}}=0$. As one can easily see by the maximum principle (and as is the case with classical Dirichlet eigenvalue problems), the principal eigenvalue decreases as the domain expands. Consequently, we have $0=\lambda_{{\rm A}{\rm B}}>\lambda^R$, but this contradicts Theorem~\ref{thm:stability}. The corollary is proved.
\end{proof}

Next we prove transient properties: This is an easy consequence of the minimality theorem.

\begin{proof}[Proof of Theorem~\ref{thm:transient}]
%%Let $\widehat{v}_i, \widehat{v}_j$ denote, respectively, the limit profile associated with $\widehat{u}_i, \widehat{u}_j$. 
By the assumption $\PR(i,j)=1$, we have $\lim_{x_j\to\infty}\widehat{v}_i{}_{\restr\Omega_j}(x_j)=1$. Since, by Theorem~\ref{thm:minimal}, $\widehat{v}_j$ is minimal among all the stationary solutions satisfying $\lim_{x_j\to\infty}v{}_{\restr\Omega_j}(x_j)=1$, we see that $\widehat{v}_j\leq \widehat{v}_i$. Next, since $\PR(j,k)=1$, we have  $\lim_{x_k\to\infty}\widehat{v}_j{}_{\restr\Omega_k}(x_k)=1$. Consequently, $\lim_{x_k\to\infty}\widehat{v}_i{}_{\restr\Omega_k}(x_k)=1$. The theorem is proved.
\end{proof}

%%%%%%%%%%%%%
\subsection{Perturbation analysis}\label{ss:proof-perturbation}

\begin{proof}[Proof of Theorem~\ref{thm:limit-blocking}]
Let $\widehat{v}_{m,i}$ denote the limit profile of the front-like solution of \eqref{RD-m} that approaches from the outer path $\Omega_{m,i}$. 
%%By Theorem~\ref{thm:dichotomy}, what we have to prove is that $\widehat{v}_{m,i}({\rm P}_{m,j})\leq\beta$ for all sufficiently large $m$. 
For notational simplicity, in what follows, we write $\widehat{v}_{m,i}$ as $v_m$. Since we are assuming $\PR_m(i,j)=0$ for $m=1,2,3,\ldots$, we have
\begin{equation}\label{v_m(P)}
v_m({\rm P}_{m,j})\leq\beta\quad (m=1,2,3,\ldots).
\end{equation}

We first consider perturbations of type (a) in Assumption~\ref{ass:perturbation-D}, in which only the length of the edges of $D$ is perturbed, while the topology of $\Omega$ remains unchanged.
Let $E$ be an arbitrary edge of $D$ and $E_m\,(m=1,2,3,\ldots)$ be the corresponding edge of $D_m$. Let $L, L_m$ be the length of $E, E_m$, respectively. By the assumption, we have $L_m\to L$ as $m\to\infty$. We identify the edge $E$ with the interval $[0,L]\subset\R$, and $E_m$ with the interval $[0,L_m]\subset\R$. Then the restriction of $v_m$ on $E_m$ is a $C^1$ function on $[0,L_m]$ and satisfies the following equation:
\[
\partial_x^2 v_m + f(v_m)=0\quad (0<x<L_m).
\]
Hence $\partial_x^2 v_m$ can be extended as a continuous function on $[0,L_m]$. We differentiate this equation:
\[
\partial_x^3 v_m +f'(v_m)\partial_x v_m=0\quad (0<x<L_m).
\]
Since the total length of the edges of $D_m$ is uniformly bounded as $m$ varies, we see from Proposition~\ref{prop:gradient-max} that $|\partial_x v_m|$ is uniformly bounded, hence so is $|\partial_x^3 v_m|$. Now we define
\[
\tilde{v}_m(x):=v_m(\eta_m x),\quad \hbox{where}\ \ \eta_m=L_m/L.
\]
Then $\tilde{v}_m$ is defined on the interval $[0,L]$ and satisfies the equation
\[
\partial_x^2 \tilde{v}_m + \eta_m^2 f(\tilde{v}_m)=0\quad(0<x<L).
\]
Since $L_m\to L$ as $m\to\infty$, we have $\eta_m\to 1$. The uniform boundedness of  $|\partial_x^3 v_m|$ implies that $\tilde{v}_m$ is uniformly bounded in $C^3([0,L])$. Consequently, after replacing $\tilde{v}_m$ by its subsequence if necessary, $\tilde{v}_m$ converges to a function $v_\infty$ as $m\to\infty$, uniformly in $C^2$, and it satisfies
\[
\partial_x^2 v_\infty + f(v_\infty)=0\quad (0<x<L).
\]
By repeating the same procedure on every edge of $D_m$, and arguing similarly on the outer paths $\Omega_{m,1},\ldots,\Omega_{m.N}$ (for which there is no need to consider $\eta_m$), we obtain a function $v_\infty$ that is defined on every edge of $\Omega$. Since $v_m$ is continuous on $\Omega^{(m)}$ and satisfies the Kirchhoff condition at every vertex of $\Omega^{(m)}$, and since $\eta_m\to 1$, the same properties are inherited by $v_\infty$ on $\Omega$. Consequently $v_\infty$ is a solution of \eqref{stationary} on the entire $\Omega$. This means that
\[
\Delta_{\Omega}v_\infty +f(v_\infty)=0\quad\hbox{on}\ \ \Omega.
\]
It also follows from \eqref{v_m(P)} that
\begin{equation}\label{v_infty(P)}
v_\infty({\rm P}_j)\leq\beta.
\end{equation}
Furthermore, the restriction of $v_m$ on $\Omega_{m,i}$ lies on the stable manifold of $(1,0)$ in the phase portrait for every $m=1,2,3,\ldots$, the same is true of the restriction of $v_\infty$ on $\Omega_i$. Therefore 
\[
\lim_{x_i\to\infty}v_\infty{}_{\restr\Omega_i}(x_i)=1.
\]
Thus, by Theorem~\ref{thm:minimal}, $\widehat{v}_i\leq v_\infty$ on $\Omega$. Combining this and \eqref{v_infty(P)}, we obtain $\widehat{v}_i({\rm P}_j)\leq\beta$, which implies $\PR(i,j)=0$, as claimed.

Next we consider perturbations of type (b) in Assumption~\ref{ass:perturbation-D}. As before, we choose an arbitrary edge $E$ of $D$, and let ${\rm A}, {\rm B}, {\rm C}_m, {\rm C}'_m, \Sigma_m$ be as in Figure~\ref{fig:perturbation-edge} ({\it above}). As $m\to\infty$, the total length of the edges of $\Sigma_m$, denoted by $|\Sigma_m|$, tends to $0$, and the points ${\rm C}_m, {\rm C}'_m$ both converge to the same point on $E$, which we call ${\rm C}_\infty$. This can be regarded as a newly added vertex, so $E$ is divided into two edges ${\rm A}{\rm C}_\infty$, ${\rm C}_\infty{\rm B}$ after the limiting procedure. Arguing similarly to the case (a) above, we can construct a function $v_\infty$ that satisfies the equation $\partial_x^2 v_\infty + f(v_\infty)=0$ on the edges ${\rm A}{\rm C}_\infty$, ${\rm C}_\infty{\rm B}$. In order for $v_\infty$ to be a solution of this equation on the entire edge $E$, we have to show that $v_\infty$ and $\partial_x v_\infty$ are continuous at ${\rm C}_\infty$. The continuity of $v_\infty$ can easily be shown by using the fact that $|\partial_x v_m|$ is uniformly bounded on $D_m$ (see Proposition~\ref{prop:gradient-max}) and that $|\Sigma_m|\to 0$ as $m\to\infty$. The continuity of $\partial_x v_\infty$ follows from the Gauss--Green formula \eqref{Green1} and the fact that $|\Sigma_m|\to 0$ as $m\to\infty$. Hence $\partial_x^2 v_\infty + f(v_\infty)=0$ on the entire edge $E$. The rest of the proof is the same as that for the case (a) above, so we omit the details.

Finally, we consider perturbations of type (c) in Assumption~\ref{ass:perturbation-D}. In this case, $|\Sigma_m|\to 0$ as $m\to\infty$, and $\Sigma_m$ shrinks to the vertex ${\rm Q}$. As before,  by a limiting procedure we obtain a function $v_\infty$ that satisfies $\partial_x^2 v_\infty + f(v_\infty)=0$ on each edge emanating from ${\rm Q}$. What we have to show is the continuity of $v_\infty$ at ${\rm Q}$ and the Kirchhoff conditions. Again the former can be shown by using the uniform boundedness of $\partial_x v_m$ and the fact that $|\Sigma_m|\to 0$ as $m\to\infty$, while the Kirchhoff condition follows from the Gauss--Green formula \eqref{Green1} and $|\Sigma_m|\to 0$ as $m\to\infty$. The rest of the proof is the same as above. The proof of Theorem~\ref{thm:limit-blocking} is complete.
\end{proof}

The proof of Theorem~\ref{thm:limit-blocking2} is similar to the case (a) of the proof of Theorem~\ref{thm:limit-blocking} but is simpler, so we leave the details of the proof to the reader. Corollary~\ref{cor:propagation-open} is just a contraposition of the above theorems, so we omit the proof. 

\begin{proof}[Proof of Theorem~\ref{thm:unification}]
Let $\widetilde{\Omega}$ be the graph obtained by unifying the outer paths $\Omega_{j_1}\ldots,\Omega_{j_m}$, and $L$ be the minimum of the length of the paths ${\rm P}_{j_k}{\rm Q}_{j_k}\,(k=1,\ldots,m)$, that is,
\[
L:=\min\big(|{\rm P}_{j_1}{\rm Q}_{j_1}|,\ldots,|{\rm P}_{j_m}{\rm Q}_{j_m}|\big).
\]
Let $\widetilde{V}_i$ denote the limit profile associated with the outer path $\Omega_i$ for the equation on $\widetilde{\Omega}$. We first show that the following holds if $L$ is sufficiently large:
\begin{equation}\label{V(P-jk)}
\widetilde{V}_i({\rm P}_{j_k})>\beta \quad(k=1,\ldots,m).
\end{equation}
Suppose the contrary. Then there exists a sequence of points $({\rm Q}^n_{j_1}\ldots,{\rm Q}^n_{j_m})$ with $L_n \to\infty$ and some $k\in\{1,\ldots,m\}$ such that the corresponding limit profile $\widetilde{V}_i^n\,(n=1,2,3,\ldots)$ satisfies
\[
\widetilde{V}_i^n({\rm P}_{j_k})\leq \beta.
\]
As we let $n\to \infty$, the function $\widetilde{V}_i^n$ converges (after taking a subsequence) to a stationary solution $\widetilde{V}_i^\infty$ on $\Omega$.  Clearly $\widetilde{V}_i^\infty\to 1$ at infinity along $\Omega_i$. Consequently, by Theorem~\ref{thm:minimal}, we have $\widehat{v}_i\leq \widetilde{V}_i^\infty$. This, together with the above inequality, implies $\widehat{v}_i({\rm P}_{j_k})\leq \beta$, but this contradicts the assumption that $\PR(i,j_k)=1$ on $\Omega$. This contradiction proves \eqref{V(P-jk)}. 

Next we prove that \eqref{V(P-jk)} implies
\begin{equation}\label{V(Q-j0)}
\widetilde{V}_i(\widetilde{\rm Q}_{j_0}) >\beta.
\end{equation}
Suppose the contrary and assume that
\begin{equation}\label{V(Q-j0)2}
\widetilde{V}_i(\widetilde{\rm Q}_{j_0}) \leq \beta.
\end{equation}
Then, since $\widetilde{V}_i({\rm P}_{j_k})>\beta$, the solution trajectories of $\widetilde{V}_i$ along the eges ${\rm P}_{j_k}\widetilde{\rm Q}_{j_0}\,(k=1,\ldots,m)$ lie outside the homoclinic orbit (see Figure~\ref{fig:phase-portrait}), therefore, in order for \eqref{V(Q-j0)2} to hold, $\widetilde{V}_i(x)$ must be monotone decreasing along the edge ${\rm P}_{j_k}\widetilde{\rm Q}_{j_0}\,(k=1,\ldots,m)$ at least near the point $\widetilde{\rm Q}_{j_0}$. Hence $\partial_{x_k}\widetilde{V}_i(\widetilde{\rm Q}_{j_0})<0\,(k=1,\ldots,m)$, where $\partial_{x_k}$ denotes the derivative along the edge ${\rm P}_{j_k}\widetilde{\rm Q}_{j_0}$ in the direction from ${\rm P}_{j_k}$ to $\widetilde{\rm Q}_{j_0}$. By the Kirchhoff condition, we have
\begin{equation}\label{Kirchhoff-Vi}
\sum_{k=1}^m \partial_{x_k}\widetilde{V}_i\widetilde{\rm Q}_{j_0})=\partial_{x_0}\widetilde{V}_i(\widetilde{\rm Q}_{j_0}),
\end{equation}
where $\partial_{x_0}$ denotes the derivative along the outer path $\widetilde{\Omega}_{j_0}$ in the direction from $\widetilde{\rm Q}_{j_0}$ to infinity. Since $\partial_{x_k}\widetilde{V}_i(\widetilde{\rm Q}_{j_0})<0\,(k=1,\ldots,m)$, the following inequality holds:
\begin{equation}\label{derivative-Q-j0}
\partial_{x_k}\widetilde{V}_i(\widetilde{\rm Q}_{j_0})> \partial_{x_0}\widetilde{V}_i(\widetilde{\rm Q}_{j_0}) \quad\hbox{for}\ \ k=1,\ldots,m.
\end{equation}

Next we define a map $h:\Omega\to \widetilde{\Omega}$ as follows. If $x\in\Omega$ does not belong to the outer paths $\Omega_{j_1}\ldots,\Omega_{j_m}$, then $h(x)=x$, where this portion of $\Omega$ and the corresponding portion of $\widetilde{\Omega}$ are identified.. If $x$ belongs to the paths ${\rm P}_{j_k}{\rm Q}_{j_k}\,(k=1,\ldots,m)$, then $h$ is an isometry from the path ${\rm P}_{j_k}{\rm Q}_{j_k}$ onto the edge ${\rm P}_{j_k}\widetilde{\rm Q}_{j_0}$ of $\widetilde{\Omega}$. Finally, if $x$ belongs to the portion of the outer paths $\Omega_{j_1}\ldots,\Omega_{j_m}$ beyond the points ${\rm Q}_{j_1},\ldots,{\rm Q}_{j_m}$, then $h$ is an isometry from these portions onto $\widetilde{\Omega}_{j_0}$. 
Now we define a function $V(x)$ on $\Omega$ as follows:
\[
V(x)=\widetilde{V}_i(h(x)).
\]
Then $V$ is continuous and satisfies the equation $\Delta_{\Omega}V+f(V)=0$ on $\Omega$ except at the points ${\rm Q}_{j_1},\ldots,{\rm Q}_{j_m}$, where $V$ has a negative derivative gap because of \eqref{derivative-Q-j0}. Consequently, $V$ is a supersolution of equation \eqref{stationary}. Furthermore, $V$ tends to $1$ at infinity along $\Omega_i$, since $V$ coincides with $\widetilde{V}_i$ on this outer path. Therefore, by Theorem~\ref{thm:minimal}, we have $\widehat{v}_i\leq V$. As we are assuming \eqref{V(Q-j0)2}, $V$ tends to $0$ at infinity along the outer paths $\Omega_{j_1},\ldots,\Omega_{j_m}$, hence so does $\widehat{v}_i$, but this contradicts the assumption that $\PR(i,j_1)=\cdots=\PR(i,j_m)=1$. This contradiction proves \eqref{V(Q-j0)}, which implies that $\widetilde{V}_i\to 1$ along $\widetilde{\Omega}_{j_0}$. The theorem is proved.
\end{proof}

\begin{rem}\label{rem:symmetric-binding2}
If $\Omega$ is symmetric with respect to $\Omega_{j_1},\ldots,\Omega_{j_m}$ as mentioned in Remark~\ref{rem:symmetric-binding}, then we have $\partial_{x_1}\widetilde{V}_i({Q}_{j_0})=\cdots=\partial_{x_m}\widetilde{V}_i({Q}_{j_0})$, hence \eqref{Kirchhoff-Vi} immediately implies \eqref{derivative-Q-j0}, from which the conclusion of the theorem follows as we have seen above. Therefore, under such symmetry, we do not need to assume that the lengths of the paths ${\rm P}_{j_k}{\rm Q}_{j_k}\,(k=1,\ldots,m)$ be long. 
\qed
\end{rem}

%%%%%%%%%%%%%
\subsection{Analysis of star graphs}\label{ss:proof-star-graphs}

We start with proving Theorem~\ref{thm:star} and its corollaries. We first present the following lemma:

\begin{lem}\label{lem:star-v}
Let $\Omega$ be a star graph. Then the following two conditions are equivalent:
\begin{itemize}\setlength{\itemsep}{0pt}
\item[{\rm (a)}] $\PR(i,j)=0$ for all $j\ne i$;
\item[{\rm (b)}] There exists a stationary solution $v$ on $\Omega$ such that $v\to 1$ along the outer path $\Omega_i$, while $v\to 0$ along all other outer paths $\Omega_j$.
\end{itemize}
\end{lem}

\begin{proof}
If (a) holds, then (b) holds for $v=\widehat{v}_i$, therefore the claim (a) $\Rightarrow$ (b) is obvious. 
Next suppose that (b) holds. Then by Theorem~\ref{thm:minimal}, $\widehat{v}_i\leq v$. Therefore $\widehat{v}_i\to 0$ along all the outer paths $\Omega_j$ with $j\ne i$, which implies (a). The lemma is proved.
\end{proof}

\begin{proof}[Proof of Theorem~\ref{thm:star}] 
We first note that, by Theorem~\ref{thm:dichotomy},
\begin{equation}\label{v(P)}
\begin{cases}
\,\widehat{v}_i({\rm P})\leq \beta \ \Leftrightarrow\ \PR(i,j)=0 \ \hbox{for all} \ \ j\ne i,\\%%\qquad
\,\widehat{v}_i({\rm P})> \beta \ \Leftrightarrow\ \PR(i,j)=1 \ \hbox{for all} \ \ j\ne i.
\end{cases}
\end{equation}

Suppose that there exists a stationary solution $v$ satisfying the condition (b) in Lemma~\ref{lem:star-v}. Then the solution trajectory of $v$ along the outer path $\Omega_i$ lies on the stable manifold of $(1,0)$ in the phase portrait (Figure~\ref{fig:phase-portrait}), while the trajectories of $v$ along all other outer paths $\Omega_j$ lie on the homoclinic orbit. Observe that the indentity (a) below holds on the stable manifold of $(1,0)$, while (b) holds on the homoclinic orbit:
\[
{\rm (a)}\ \ \ \frac12 \left(\partial_x v\right)^2 + F(v)=F(1),\qquad
{\rm (b)}\ \ \ \frac12 \left(\partial_x v\right)^2 + F(v)=0.
\]
Now we set $\xi:=\widehat{v}_i({\rm P})$, where ${\rm P}$ denotes the center point of $\Omega$, and let $0\leq x_i<\infty$ be the coordinates on $\Omega_i$, and $0\leq x_j<\infty$ be those on $\Omega_j$, $j\ne i$. Then we have
\[
\partial_{x_i}\widehat{v}_i({\rm P})=\sqrt{2(F(1)-F(\xi))},\quad 
\partial_{x_j}\widehat{v}_j({\rm P})=-\sqrt{-2F(\xi)}\ \ (j\ne i).
\]
Thus, by the Kirhhoff condition, it holds that
\[
\rho_i \sqrt{2(F(1)-F(\xi))}=\sum_{j\ne i} \rho_j \sqrt{-2F(\xi)},
\]
or, equivalently,
\[
\rho_i^2 \left(F(1)-F(\xi)\right)=-\Big(\sum_{j\ne i}\rho_j\Big)^2 F(\xi). 
\]
Using the constant $R_i$, we can rewrite the above equality as
\begin{equation}\label{xi}
F(1)+\big(R_i^2-1\big)F(\xi)=0.
\end{equation}
Thus, if a stationary solution $v$ satisfying (b) exists, then there exists $0<\xi<\beta$ such that \eqref{xi} holds. The converse is also true since \eqref{xi} is equivalent to the Kirchhoff condition at ${\rm P}$. As $\xi$ varies over the interval $(0,\beta)$, $F(\xi)$ varies over $[F(a),0)$ and attains its minimum at $\xi=a$. Therefore, \eqref{xi} holds for some $\xi\in(0,\beta)$ if and only if %%$R_i>1$ and 
\begin{equation}\label{F(a)}
F(1)+\big(R_i^2-1\big)F(a)\leq 0.
\end{equation}
Consequently, \eqref{F(a)} is a necessary and sufficiently condition for a stationary solution $v$ satisfying (b) to exist, and, by Lemma~\ref{lem:star-v}, this is equivalent to $\PR(i,j)=0$ for all $j\ne i$. By \eqref{v(P)}, this is also equivalent to $\widehat{v}_i({\rm P})\leq \beta$. In other words, if $F(1)+\big(R_i^2-1\big)F(a)> 0$, we have  $\widehat{v}_i({\rm P})> \beta$, which implies $\PR(i,j)=1$ for all $j\ne i$. This proves the criterion \eqref{star-graph-unequal}.

Finally, let us show that $\widehat{v}_i\equiv 1$ if $\PR(i,j)=1$ for all $j\ne i$. In this case, the solution trajectories of $\widehat{v}_i$ along the outer paths $\Omega_1,\ldots,\Omega_N$ all lie on the stable manifold of $(1,0)$, therefore $\partial_{x_j}\widehat{v}_i({\rm P})\geq 0$ for all $j=1,\ldots,N$. In the mean while, by the Kirchhoff condition, we have
\[
\sum_{j=1}^N \partial_{x_j}\widehat{v}_i({\rm P})=0.
\]
Consequently, $\partial_{x_j}\widehat{v}_i({\rm P})=0$ for $j=1,\ldots,N$, which implies $\widehat{v}_i\equiv 1$ on $\Omega$. This completes the proof of Theorem~\ref{thm:star}.
\end{proof}

Next we prove the results on perturbation of star graphs. As Theorem~\ref{thm:perturbation-star1} is a special case of Corollary~\ref{cor:propagation-open}, there is no need for a proof, so we begin with proving Theorem~\ref{thm:perturbation-star2}. 

\begin{proof}[Proof of Theorem~\ref{thm:perturbation-star2}] 
Let $\Omega'_1,\ldots,\Omega'_N$ denote the outer paths of $\Omega'$ and ${\rm P}_j\,(j=1,\ldots,N)$ be the exit points from which $\Omega'_j$ emanates. (These points are not necessarily distinct.). We shall construct a supersolution $v^+$ of the stationary problem $\Delta_{\Omega'}v+f(v)=0$ on $\Omega'$ such that $v^+\to 1$ at infinity along $\Omega'_i$, while $v^+\to 0$ at infinity along $\Omega'_j$ for all $j\ne i$. Once such a supersolution is shown to exist, then by Theorem~\ref{thm:minimal}, the desired conclusion follows.

We first note that the condition \eqref{star-graph-unequal-bb} implies that
\[
\rho_i \sqrt{2(F(1)-F(a))}<\sum_{j\ne i} \rho_j \sqrt{-2F(a)},
\]
therefore there exists a constant $\delta>0$ such that
\begin{equation}\label{Kirchhoff-delta}
\rho_i \sqrt{2(F(1)-F(a))}=\sum_{j\ne i} \rho_j \left(\sqrt{-2F(a)}-\delta\right).
\end{equation}
Clearly, $\sqrt{-2F(a)}-\delta>0$. 
Now we define constants $b_1,\ldots, b_N$ by
\[
b_i=\sqrt{2(F(1)-F(a))}>0,\quad b_j=-\sqrt{-2F(a)}+\delta<0\ \ (j\ne i)
\]
and consider a solution $h(x)$ of the following boundary value problem:
\begin{equation}\label{h-equation}
\begin{cases}
\, \Delta_{D} h=0 \ \ \hbox{in}\ \ D,\\[2pt] 
\, \dfrac{\partial h}{\partial \nu_i}({\rm P}_j)= b_j\ \ (j=1,\ldots,N),
\end{cases}
\end{equation}
where $\partial h/\partial\nu_i({\rm P}_j)$ denotes the formal outer derivative of $h$ at ${\rm P}$ in the direction of the outer path $\Omega'_j\,(j=1,\ldots,N)$. More precisely, these are symbolic outer derivatives that measure the discrepancy of the Kirchhoff condition at ${\rm P}_1,\ldots,{\rm P}_N$. By \eqref{Kirchhoff-delta}, we have $\sum_{j=1}^N \rho_j b_j=0$, therefore, by Proposition~\ref{prop:Neumann-Laplace}, such a solution exists and is unique under the conditon
\[
h({\rm P}_i)=a.
\]
Since a harmonic function on a bounded finite metric graph attains its maximum at the boundary points, and since $b_j<0$ for $j\ne i$, the maximum of $h$ on $D$ is attained at ${\rm P}_i$. Consequently, 
\[
h(x)\leq h({\rm P}_i)=a\quad (x\in D).
\]
Let $\xi_j:=h({\rm P}_j)$ for $j=1,\ldots,N,\,j\ne i$. Then $\xi_j\leq a\,(j\ne i)$. Since the gradient of $h$ is uniformly bounded by Proposition~\ref{prop:gradient-max}, there exists a constant $C>0$ such that
\begin{equation}\label{a-xi}
a-C|D|\leq \xi_j\leq a\quad(j\ne i),
\end{equation}
where $|D|$ denotes the total length of the edges of $D$. In particular, $\xi_j>0$ if $|D|$ is sufficiently small.  Since a harmonic function attains its minimum at the boundary points, we obtain
\[
0<\min_{j\ne i} \xi_j \leq h(x) \leq h({\rm P}_i)=a\quad (x\in D),
\]
provided that $|D|$ is sufficiently small. 
Hence $f(h(x))\leq 0$ on $D$, which implies
\begin{equation}\label{h-super}
\Delta_{D}h+f(h)\leq 0\quad \hbox{on}\ \ D.
\end{equation}

Next let $U$ be a solution of the following problem:
\begin{equation}\label{U}
\partial_x^2 U + f(U)=0\ \ (x>0),\quad U(0)=a,\quad U(x)\to 1\ \ \hbox{as}\ \ x\to\infty.
\end{equation}
Clearly such a solution exists uniquely. Furthermore,
\begin{equation}\label{U'}
\partial_x U(0)=\sqrt{2(F(1)-F(a))} = b_i.
\end{equation}

Now we define $v^+(x)$ on $\Omega'$ as follows:
\[
\begin{cases}
\, v^+{}_{\restr\Omega'_i}(x_i)=U(x_i)\ \ (0\leq x_i<\infty),\\[2pt]
\, v^+{}_{\restr D}(x)=h(x)\ \ (x\in D), \\[2pt]
\, v^+{}_{\restr\Omega'_j}(x_j)=V(x_j+c_j)\ \ (0\leq x_j<\infty),\ j\ne i,
\end{cases}
\]
where $V$ denotes the pulse solution shown in Figure~\ref{fig:pulse-solution}, and $c_j\,(j\ne i)$ are constants such that
\[
V(c_j)=\xi_j.
\]
This pulse solution corresponds to the homoclinic orbit, and \eqref{a-xi} implies
\[
\partial_{x}V(c_j)=-\sqrt{-2F(\xi_j)}=-\sqrt{-2F(a)} +{\mathcal O}(|D|).
\]
Consequently, if $|D|$ is sufficiently small, it holds that
\[
\partial_{x_j}v^+{}_{\restr\Omega'_j}({\rm P}_j)=-\sqrt{-2F(a)} +{\mathcal O}(|D|)<-\sqrt{-2F(a)}+\delta=b_j \ \ \ (j\ne i).
\]
The above inequality and the boundary condition of $h$ at ${\rm P}_j\,(j\ne i)$ imply that $v^+$ satisfies the super-Kirchhoff condition at ${\rm P}_j\,(j\ne i)$, while \eqref{U'} implies that $v^+$ satisfies the Kirchhoff condition at ${\rm P}_i$. In view of these and \eqref{h-super}, we see that $v^+$ is a supersolution of $\Delta_{\Omega'}v+f(v)=0$ on $\Omega'$ and from the construction it is clear that $v^+\to 1$ at infinity along $\Omega'_i$, while $v^+\to 0$ at infinity along $\Omega'_j\,(j\ne i)$. Thus, by Theorem~\ref{thm:minimal}, we have $\PR'(i,j)=0$ for $j\ne i$. This completes the proof of the theorem.
\end{proof}

Before starting to prove Theorems~\ref{thm:local-star1} and \ref{thm:local-star2}, we introduce some notation. Let $\delta$ be as in \eqref{Kirchhoff-delta}, which reduces to the form under the assumption of equal edge thickness:
\begin{equation}\label{Kirchhoff-delta-1}
\sqrt{2(F(1)-F(a))}=(N'-1)\big(\sqrt{-2F(a)}-\delta\big),
\end{equation}
or, equivalently,
\begin{equation}\label{Kirchhoff-delta-2}
\sqrt{2(F(1)-F(a))}+(N'-1)\delta=(N'-1)\sqrt{-2F(a)}.
\end{equation}
Let $U_1(x)$ be a solution of the following initial-value problem:
\[
\partial_x^2 U_1 + f(U_1)=0\ \ (x>0),\quad U_1(0)=a,\quad \partial_x U_1(0)=\sqrt{2(F(1)-F(a))}+(N'-1) \delta.
\]
Then the solution trajectory of $U_1$ lies above the stable manifold of $(1,0)$, and there exists $M_1>0$ such that
\begin{equation}\label{M1}
a\leq U_1<1\ \ (0\leq x<M_1),\quad U_1(M_1)=1. 
\end{equation}

Next let $V_1(x)$ be a solution of the following initial-value problem:
\[
\partial_x^2 V_1 + f(V_1)=0\ \ (x>0),\quad V_1(0)=a,\quad \partial_x V_1(0)=-\sqrt{-2F(a)}.
\]
The solution trajectory of $V_1$ forms a portion of the homoclinic orbit, and $V_1(x)\to 0$ as $x\to\infty$. We also define a function $V_2$ which is a solution of 
\[
\partial_x^2 V_2 + f(V_2)=0\ \ (x>0),\quad V_2(0)=a,\quad \partial_x V_2(0)=-\sqrt{-2F(a)}+\delta.
\]
The solution trajectory of $V_2$ lies inside the homoclinic orbit, and it corresponds to a periodic orbit. Therefore, there exists $M_2>0$ such that
\begin{equation}\label{M2}
\partial_x V_2(x)<0\ \ (0\leq x<M_2),\quad \partial_x V_2(M_2)=0. 
\end{equation}

\begin{proof}[Proof of Theorem~\ref{thm:local-star1}]
Assume that the length of the edge ${\rm Q}{\rm P}$ is larger than $M_1$, where $M_1$ is as in \eqref{M1}, and let ${\rm R}$ denote the point on the edge ${\rm Q}{\rm P}$ such that the length of the path ${\rm R}{\rm P}$ is equal to $M_1$. Next we define coordinates $0\leq y\leq M_1$ on the path ${\rm Q}{\rm P}$, with $y=0$ corresponding to the point ${\rm P}$, and $y=M_1$ to ${\rm R}$. The coordinates on the outer path $\Omega_{j_k}$ are denoted by $0\leq x_{j_k}<\infty$. Now we define a function $v^+_1$ on $\Omega$ as follows.
\[
\begin{cases}
\,v^+_1{}_{\restr\Omega_{j_k}}(x_{j_k})=V_1(x_{j_k})\ \ (0\leq x_{j_k}<\infty)\ \ \hbox{for} \ k=1,\ldots,N'-1,\\[2pt]
\,v^+_1{}_{\restr{\rm P}{\rm Q}}(y)=U_1(y)\ \ (0\leq y\leq M_1)\\[2pt]
\,v^+_1 (x)=1 \quad\hbox{otherwise}
\end{cases}
\]
Then $v^+_1$ satisfies the Kirchhoff condition at ${\rm P}$ by \eqref{Kirchhoff-delta-2}, has a negative derivative gap at ${\rm R}$, and it is a stationary solution in the rest of $\Omega$. Hence $v^+_1$ is a supersolution of $\Delta_{\Omega}v+f(v)=0$ on $\Omega$ and satisfies $v^+_1=1$ on $\Omega_i$ while $v^+_1\to 0$ along the outer paths $\Omega_{j_k}\,(k=1,\ldots,N'-1)$. Thus, by Theorem~\ref{thm:minimal}, we have $\PR(i,j_k)=0$ for $k=1,\ldots,N'-1$. The theorem is proved.
\end{proof}

\begin{proof}[Proof of Theorem~\ref{thm:local-star2}]
Assume that the length of the edge ${\rm P}{\rm Q}_{j_k}$ is larger than $M_2$ in \eqref{M2} for $k=1,\ldots,m$. If $m<N'-1$, let $\Omega_{j_{m+1}},\ldots,\Omega_{j_{N'-1}}$ denote the remaining outer paths emanating from ${\rm P}$. (If $m=N'-1$, then no such outer paths exist.) Let ${\rm R}_k\,(k=1,\ldots,N'-1)$ denote the points on the edges ${\rm P}{\rm Q}_{j_k}$ (for $1\leq k\leq m$) or on the outer paths $\Omega_{j_k}$ (for $m+1\leq k \leq N'-1$) such that the length of the path ${\rm P}{\rm R}_{j_k}$ is equal to $M_2$ for $k=1,\ldots,N'-1$. The coordinates on the edge ${\rm P}{\rm Q}_{j_k}$ is denoted by $x_{j_k}$ for $k=1,\ldots,m$ and those on the outer path $\Omega_{j_k}$ are also denoted by $x_{j_k}$ for $m+1\leq k \leq N'-1$. Now we define a function $v^+_2$ on $\Omega$ as follows.
\[
\begin{cases}
\,v^+_2{}_{\restr\Omega_1}(x_1)=U(x_1)\ \ (0\leq x_1<\infty),\\[2pt]
\,v^+_2{}_{\restr{\rm P}{\rm Q}_{j_k}}(x_{j_k})=V_2(x_{j_k})\ \ (0\leq x_{j_k}\leq M_2)\ \ \hbox{for} \ k=1,\ldots,m,\\[2pt]
\,v^+_2{}_{\restr\Omega_{j_k}}(x_{j_k})=V_2(x_{j_k})\ \ (0\leq x_{j_k}\leq M_2)\ \ \hbox{for} \ k=m+1,\ldots,N'-1,\\[2pt]
\,v^+_2 (x)=V_2(M_2) \quad\hbox{otherwise},
\end{cases}
\]
where $U$ is the function defined in \eqref{U}. Then $v^+_2$ satisfies the Kirchhoff condition at ${\rm P}$ by \eqref{Kirchhoff-delta-1}, is $C^1$ at ${\rm R}_{j_k}\,(k=1,\ldots,N'-1)$ and is a super solution beyond the points ${\rm R}_{j_k}\,(k=1,\ldots,N'-1)$ since $V_2(M_2)<a$. Hence $v^+_2$ is a supersolution of $\Delta_{\Omega}v+f(v)=0$ on $\Omega$ and satisfies $v^+_2\to 1$ along $\Omega_1$ while $v^+_2\to V_2(M_2)<a$ along all other outer paths. Thus, by Theorem~\ref{thm:minimal}, the desired conclusion follows. The theorem is proved.
\end{proof}

%%%%%%%%%%%%%
%%\subsection{Proof of partial and one-way propagation}\label{ss:proof-examples}

%%In this section we prove Propositions~\ref{prop:partial} and \ref{prop:one-way} on the partial and one-way propagation for the graphs shown in Figure~\ref{fig:partial-one-way}. Our arguments rely on the estimates given in the paper \cite{JM2024}. 

%%%%%%%%%%%%%
\subsection{Estimates on the reservoir and incomplete invasion}\label{ss:estimates-reservoir}

In this section we first estimate the value of the limit profile $\widehat{v}_i$ on the reservoir type subgraph ${\mathcal R}_0$ shown in Figures~\ref{fig:3star-reservoir}, \ref{fig:reservoir} and prove Theorem~\ref{thm:reservoir}. Our method is to construct an upper barrier with a variational argument. We shall then prove Theorem~\ref{thm:complete}.

\begin{proof}[Proof of Theorem~\ref{thm:reservoir}]
Let $0\leq y\leq L$ denote the coordinates on the edge $E_0$, with $y=0$ corresponding to ${\rm P}_0$ 
and $y=L$ to ${\rm Q}_0$. 
Let $L^*:=(2(F(1)-F(a))^{-1/2}$. We divide the proof between the case $L\leq L^*$ and the case $L>L^*$. 

In the case $L\leq L^*$, we consider the following initial-boundary value problem on ${\mathcal R}_0$.
\begin{equation}\label{IBV-R0}
\begin{cases}
\,\partial_t u=\Delta_{{\mathcal R}_0} u + f(u) & (t>0,\,x\in{\mathcal R}_0),\\
\,u(0,x)=u_0(x)& (x\in{\mathcal R}_0)\\
\,u(t,{\rm P}_0)=1 & (t\geq 0).
\end{cases}
\end{equation}
Let $u(t,x)$ be the solution of \eqref{IBV-R0} for the following initial data:
\[
\begin{cases}
\, u_0{}_{\restr E_0}(y)=1-y/L \ \ (0\leq y \leq L),\\
\, u_0{}_{\restr\Delta_0}=0.
\end{cases}
\]
We define the local energy $J[u]$ on ${\mathcal R}_0$ by \eqref{energy} with $D_0$ replaced by ${\mathcal R}_0$. Then, by Remark~\ref{rem:energy}, we have
\begin{equation}\label{J[u]}
J[u(t,\cdot)]\leq J[u_0] =\frac{1}{2L}-\int_{E_0}F(u_0)dx < \frac{1}{2L}-L F(a).
\end{equation}
We claim that
\begin{equation}\label{u<delta}
\frac{1}{|\Delta_0|}\int_{\Delta_0} u(t,x)dx <\delta \quad\hbox{for all}\ \ t\geq 0,
\end{equation}
where $|\Delta_0|$ denotes the total length of the edges of $\Delta_0$ and $\delta>0$ is the constant in \eqref{delta-mu}. Suppose the contrary. Then, since \eqref{u<delta} holds for $t=0$, there exists $t_1>0$ such that
\[
\frac{1}{|\Delta_0|}\int_{\Delta_0} u(t_1,x)dx =\delta,
\]
which implies
\[
\int_{\Delta_0} \left(u(t_1,x)-\delta\right)dx =0.
\]
Then, by \eqref{Poincare2}, we have
\[
\int_{\Delta_0}|\nabla u(t_1,x)|^2 = \int_{\Delta_0}|\nabla \left(u(t_1,x)-\delta\right)|^2 dx
\geq \mu_1(\Delta_0)\int_{\Delta_0}\left(u(t_1,x)-\delta\right)^2 dx,
\]
where $\mu_1(\Delta_0)$ is the smallest positive eigenvalue of $-\Delta$ on $\Delta_0$ (see \eqref{mu-eigenvalue}) Hence, by \eqref{delta-mu},
\[
\int_{\Delta_0}\left(\frac12|\nabla u(t_1,x)|^2-F(u(t_1,x))\right)dx\geq \sigma|\Delta_0|. 
\] 
Combining this and \eqref{J[u]}, we obtain
\[
\frac{1}{2L}-L F(a)> J[u(\cdot,t_1)] \geq \sigma|\Delta_0|+\int_{E_0}\left(\frac12|\nabla u(t_1,x)|^2-F(u(t_1,x))\right)dx\geq \sigma|\Delta_0|-L F(1).
\]
Hence
\[
\frac{1}{2L}+L (F(1)-F(a)) > \sigma|\Delta_0|,
\]
but this contradicts the assumption in \eqref{reservoir-condition-a}. This contradiction proves \eqref{u<delta}. 

As in the case of reaction-diffusion equations on bounded Euclidean domains, the existence of the local energy $J[u]$ implies that there exists a sequence $0<t_1<t_2<\cdots\to\infty$ that $u(t_k,x)$ converges to a stationary solution $V(x)$ of \eqref{IBV-R0}. 
By \eqref{u<delta}, $V$ satisfies
\begin{equation}\label{V<delta}
\frac{1}{|\Delta_0|}\int_{\Delta_0} V(x)dx \leq \delta.
\end{equation}
Clearly $V({\rm P}_0)=1$. We extend $V$ outside ${\mathcal R}_0$ by $V=1$, to form a function $\widetilde{V}$ on $\Omega$. Then $\widetilde{V}$ is continuous on $\Omega$, has a negative derivative gap at ${\rm P}_0$ and satisfies $\Delta\widehat{V}+f(V)=0$ elsewhere. Therefore $\widetilde{V}$ is a super solution of the equation $\Delta_{\Omega}v+f(v)=0$ on $\Omega$. Furthermore, $\widetilde{V}=1$ on $\Omega_i$. Hence, by Theorem~\ref{thm:minimal}, we have $\widehat{v}_i\leq \widetilde{V}$. 
This, together with \eqref{V<delta}, proves \eqref{v-reservoir} for the case $L\leq L^*$.  

Next we consider the case where $L>L^*$. In this case, let ${\rm P}'_0$ be the point on $E_0$ where $y=L-L^*$, and let ${\mathcal R}'_0$ be the union of $\Delta_0$ and the path ${\rm P}'_0{\rm Q}_0$, whose length is $L^*$. We consider the same initial-boudary value problem as \eqref{IBV-R0} on ${\mathcal R}'_0$ with ${\rm P}_0$ replaced by ${\rm P}'_0$. Arguing precisely as above, we see that if \eqref{u<delta} does not hold, then we would have
\[
\frac{1}{2L^*}+L^* (F(1)-F(a)) > \sigma|\Delta_0|.
\]
The left-hand side of the above inequality is equal to $\sqrt{2(F(1)-F(a))}$, hence it contradicts the assumption in \eqref{reservoir-condition-b}. This contradiction proves \eqref{u<delta}. The rest of the arguments is precisely the same as above. This completes the proof of Theorem~\ref{thm:reservoir}.
\end{proof}

\begin{proof}[Proof of Theorem~\ref{thm:complete}]
Since we are assuming $\PR(i,j)=1$, we have $\widehat{v}_i({\rm P})>\beta$ for $j\ne i$ (see Theorem~\ref{thm:dichotomy}). Fix such $j$. By Proposition~\ref{prop:gradient-max}, the gradient of $\widehat{v}_i$ is uniformly bounded by a constant that is independent of the choice of the center graph $D$ so long as $|D|$ is not too large. Therefore, there is a constant $C>0$ such that
\[
\min_{x\in D} \widehat{v}_i(x)\geq \widehat{v}_i({\rm P}_j)-C|D|>\beta -C|D|.
\]
Therefore, $\widehat{v}_i(x)>a$ for all $x\in D$ if $|D|$ is sufficiently large. By the Gauss-Green formula \eqref{Green1}, we have
\[
\sum_{k=1}^N\frac{\partial \widehat{v}_i}{\partial \nu_k}({\rm P}_k)=\int_D \Delta_{D}\widehat{v}_i(x)dx
=-\int_D f\left(\widehat{v}_i(x)\right)dx.
\]
Since $\widehat{v}_i(x)>a$ on $D$, we have $-f(\widehat{v}_i(x))\leq 0$ on $D$. In the mean while, the outer derivatives $\partial \widehat{v}_i/\partial \nu_k({\rm P}_k)$ are all non-negative. Therefore, for the above equality to hold, we have
\[
\frac{\partial \widehat{v}_i}{\partial \nu_k}({\rm P}_k)=0\ \ (k=1,\ldots,N),\quad f(\widehat{v}_i(x))=0\ \ \hbox{on}\ \ D.
\]
This is possible only if $\widehat{v}_i\equiv 1$ on $\Omega$. The theorem is proved.
\end{proof}

%%%%%%%%%%%%%
\subsection{Proof of Theorem \ref{thm:general} on the Cauchy problem}\label{ss:proof-general}

In this section we prove Theorem~\ref{thm:general} on the long-time behavior of solutions of the Cauchy problem \eqref{RD-Omega}, \eqref{u0}. 

\begin{proof}[Proof of Theorem~\ref{thm:general}]
We first note that the following inequality holds:
\begin{equation}\label{H<v}
H(x_i)<\widehat{v}_i{}_{\restr\Omega_i}(x_i)\quad(0\leq x_i<\infty).
\end{equation}
Indeed, since the solution trajectories of $H$ and $\widehat{v}_i{}_{\restr\Omega_i}$ both lie on the stable manifold of $(1,0)$ in the phase portrait and since $0=H(0)<\widehat{v}_i{}_{\restr\Omega_i}(0)$, there exists $b>0$ such that $\widehat{v}_i{}_{\restr\Omega_i}(x_i)=H(x_i+b)$. This, together with $H'>0$, implies \eqref{H<v}. 

We begin with the proof under assumption (a). We extend $\Psi^b$ as a function on $\Omega$ by setting 
$\Psi^b(x)=\Psi^b(x_i)$ if $x\in\Omega_i$ and $\Psi^b(x)=0$ if $x\in \Omega\setminus \Omega_i$. 
Then this is a time-independent subsolution of \eqref{RD-Omega}. Let $U^b(t,x)$ denote the solution of \eqref{RD-Omega} with initial data $\Psi^b$. Then $U^b(t,x)$ is monotone increasing in $t$, therefore it converges to a stationary solution $V^b(x)$. By \eqref{H<v}, we have $\Psi^b<\widehat{v}_i$, therefore
\[
\Psi^b(x)<V^b(x)\leq \widehat{v}_i(x)\quad\hbox{on}\ \ \Omega.
\] 
Now we consider a family of subsolutions $\Psi^{b'}$ for $b'\geq b$. When $b'=b$, we have $\Psi^{b'}<V^b$ as mentioned above. As we continuously increase the parameter $b'$, the strong maximum principle on the half line $\Omega_i$ implies that the graph of $\Psi^{b'}$ can never touch that of $V^b$ from below. Hence
\[
\Psi^{b'}(x)<V^b(x)\quad \hbox{for all}\ \ b'\geq b.
\]
Consequently, $V^b_{\restr\Omega_i}(x_i)>\beta$ for all sufficiently large $x_i>0$. This is possible only if the solution trajectory of $V^b_{\restr\Omega_i}$ lies on the stable manifold of $(1,0)$, therefore
\[
\lim_{x_i\to\infty}V^b_{\restr\Omega_i}(x_i)=1.
\]
Then, by Theorem~\ref{thm:minimal}, we have $V^b\leq \widehat{v}_i$, hence $V^b=\widehat{v}_i$. Since $u(t,x)$ lies between $U^b(t,x)$ and $\widehat{v}_i$ by \eqref{H<v}, we see that $u(t,x)\to\widehat{v}_i(x)$ as $t\to\infty$.

Next we assume (b). This is a condition that is well known as a sufficient condition for a solution of the equation $\partial_t u=\partial_x^2 u + f(u)$ on $\R$ to converge to $1$ as $t\to\infty$ (\cite[Theorem 3.1]{FM1977}). Using the same argument, one can show that there exists $b>0$ and $T>0$ such that
\[
u(x,T)> \Psi^b(x).
\]
We also have $u(x,T)<\widehat{v}_i(x)$ by \eqref{H<v}. Thus, by the comparison principle,
\[
U^b(t,x)<u(t+T,x)<\widehat{v}_i(x)\quad\hbox{for}\ \ t\geq 0,\ x\in\Omega.
\]
Letting $t\to\infty$, we obtain $u(t,x)\to\widehat{v}_i(x)$ as $t\to\infty$. The proof of the theorem is complete.
\end{proof}

\vskip 15pt

%%%%%%%%%%%%%%%%%%%%%%
%%%%%%%%%%%%%%%%%%%%%%

\noindent
{\bf Acknowledgements}

\vspace{5pt}
The present work was initiated in 2023, when H.~M. visited S.~J. at Hokkaido University. H.~M. was partially supported by JSPS KAKENHI 21H00995 and 23K20807.

%%%%%%%%%%%%%%%%%%%%%%
%%%%%%%%%%%%%%%%%%%%%%
%%\section{Declarations}\label{s:declaration}

%%\begin{itemize}
%%\item Funding: Hiroshi Matano was supported by JSPS KAKENHI 21H00995 and 23K20807.
%%\item Conflict of interest/Competing interests:  Hiroshi Matano is a guest editor of this special issue. 
%%There are no other conflicts of interest. 
%%\end{itemize}

%%%%%%%%%%%%%%%%%%%%%%
%%%%%%%
%%\newpage

%%%%%%%%%%%%%%%%%%%%%%%

\vskip 15pt
\noindent
\underbar{Hiroshi Matano}: (corresponding author)\\
Meiji Institute for Advanced Study of Mathematical Sciences, Meiji University, 4-21-1 Nakano, Tokyo 164-8525, Japan\\
email: matano@meiji.ac.jp %%\quad ORCID: 0000-0002-9944-4677

\vskip 10pt
\noindent
\underbar{Shuichi Jimbo}:\\
Department of Mathematics, Hokkaido University, Sapporo 060-0810, Japan\\
email: jimbo@math.hokudai.ac.jp %%\quad ORCID: 

\end{document}